\documentclass[11pt,a4paper,reqno]{amsart}

\usepackage{amsfonts,amsmath,amssymb}
\usepackage{dsfont}
\usepackage{enumerate}
\usepackage{xcolor}
\usepackage[all]{xy}

\usepackage[T1]{fontenc}
\usepackage[latin9]{inputenc}
\usepackage{mathrsfs}
\usepackage{mathtools}
\usepackage{ulem}
\usepackage{enumitem}
\usepackage[unicode=false,pdfusetitle,
bookmarks=true,bookmarksnumbered=false,bookmarksopen=false,
breaklinks=false,pdfborder={0 0 1},backref=false,colorlinks=false]
{hyperref}

\def\bb#1\eb{\textcolor{blue}{#1}} %
\def\br#1\er{\textcolor{red}{#1}} %
\def\bm#1\em{\textcolor{purple}{#1}} %

\usepackage[mathscr]{euscript}

\newcommand{\dv}{\dot{\partial}}
\newcommand{\canon}{\mathbb{C}}
\newcommand{\dive}{\mathrm{div}}
\newcommand{\ZZ}{\mathcal{Z}}
\newcommand{\an}{\Gamma}
\newcommand{\N}{\mathrm{N}}
\newcommand{\LC}{\N^{L}}
\newcommand{\G}{\mathrm{G}}
\newcommand{\sprL}{\G^{L}}
\newcommand{\torN}{\mathrm{Tor}}
\newcommand{\di}{n}
\newcommand{\C}{\mathrm{C}}
\newcommand{\la}{\mathrm{Lan}}

\newcommand{\RN}{\mathcal{R}}
\newcommand{\RL}{\mathcal{R}^L}
\newcommand{\covan}{\nabla^\an}
\newcommand{\covN}{\nabla^\N}
\newcommand{\covL}{\nabla^L}
\newcommand{\ri}{\mathrm{Ric}}
\newcommand{\riL}{\mathrm{Ric}^L}
\newcommand{\B}{\mathcal{B}^{\J}}
\newcommand{\J}{\mathcal{J}}
\newcommand{\dom}{D}
\newcommand{\accion}{\mathscr{S}^{\dom}}
\newcommand{\vol}{d\mu}
\newcommand{\lie}{\mathscr{L}}

\newcommand{\sol}{\mathrm{Sol}_L(A)}
\newcommand{\solsim}{\mathrm{Sol}^{\mathscr{S}\mathrm{ym}}_L(A)}
\newcommand{\solprop}{\mathrm{Sol}_L(\overline{A})}
\newcommand{\prsim}{\Pi^{\mathscr{S}\mathrm{ym}}}
\newcommand{\Z}{\mathcal{W}}
\newcommand{\s}{\sigma^{\Z}}
\newcommand{\K}{\mathcal{K}^{\Z}}
\newcommand{\n}{\nu}
\newcommand{\f}{f}
\newcommand{\ka}{\kappa}
\newcommand{\II}{\bar{c}}
\newcommand{\JJ}{\bar{d}}
\newcommand{\AAA}{\bar{c}}
\newcommand{\BB}{\bar{d}}
\newcommand{\SR}{\Sigma^{\R}}
\newcommand{\SL}{\Sigma^{F_p}}
\newcommand{\R}{\mathbf{r}}
\newcommand{\z}{z}
\newcommand{\A}{\mathcal{A}}
\newcommand{\commentt}[1]{}
\newcommand{\MA}{M_A}
\newcommand{\MAA}{M_{\overline{A}}}

\newtheorem{thm}{Theorem}[section]
\newtheorem{prop}[thm]{Proposition}
\newtheorem{lem}[thm]{Lemma}
\newtheorem{cor}[thm]{Corollary}
\theoremstyle{definition}
\newtheorem{defn}[thm]{Definition}

\newtheorem{rem}[thm]{Remark}

\title[ palatini formalism in pseudo-finsler geometry  ]{The Einstein-Hilbert-Palatini formalism \\ in Pseudo-Finsler Geometry}

\author[M. A. Javaloyes]{Miguel \'Angel Javaloyes}
\address{Departamento de 
	Matem\'aticas, \hfill\break\indent
	Universidad de Murcia, \hfill\break\indent
	Campus de Espinardo,\hfill\break\indent
	30100 Espinardo, Murcia, Spain}
\email{majava@um.es}

\author[M. S\'anchez]{Miguel S\'anchez}
\address{Departamento de Geometr\'{\i}a y Topolog\'{\i}a, Facultad de Ciencias \& \hfill\break\indent
IMAG (Centro de Excelencia Mar\'{\i}a de Maeztu) \hfill\break\indent
	Universidad de Granada, 18071 Granada, 
	Spain}
\email{sanchezm@ugr.es}

\author[F. F. Villaseñor]{Fidel F. Villaseñor}

\address{Departamento de Geometr\'{\i}a y Topolog\'{\i}a, Facultad de Ciencias \& \hfill\break\indent
IMAG (Centro de Excelencia Mar\'{\i}a de Maeztu) \hfill\break\indent
	Universidad de Granada, 18071 Granada, 
	Spain}
\email{fidelfv@ugr.es}

\thanks{
MAJ was partially supported by the project  PGC2018-097046-B-I00 funded by MCIN/ AEI /10.13039/501100011033/ FEDER ``Una manera de hacer Europa'' and Fundaci\'on S\'eneca project with reference 19901/GERM/15.  This work is a result of the activity developed within the framework of the Programme in
	Support of Excellence Groups of the Regi\'on de Murcia, Spain, by Fundaci\'on S\'eneca, Science and Technology Agency of the Regi\'on de Murcia.
MS and FFV were partially supported by 
the project  PID2020-116126GB-I00 funded by MCIN/ AEI /10.13039/501100011033, by the project PY20-01391 (PAIDI 2020)  funded by Junta de Andaluc\'{\i}a---FEDER and by the framework of IMAG-Mar\'{\i}a de Maeztu grant CEX2020-001105-M funded by MCIN/AEI/ 10.13039/50110001103. FFV is partially supported also by an FPU grant (Formación de Profesorado Universitario) from the Spanish Ministerio de Universidades.
}
\thanks{2020 {\it MSC.} Primary: 58J60, 
83D05; Secondary: 53C21, 35A15. }
\begin{document}   

\begin{abstract}
	A systematic development of the so-called Palatini formalism is carried out for pseudo-Finsler metrics $L$ of any signature. Substituting in the classical Einstein-Hilbert-Palatini functional the scalar curvature by the Finslerian Ricci scalar constructed with an independent nonlinear connection $\N$, the   affine and metric   equations for $(\N,L)$ are obtained. In Lorentzian signature with vanishing mean Landsberg tensor $\la_i$,  both the Finslerian Hilbert metric equation and the classical Palatini conclusions are recovered  by means of  a combination of techniques  involving  the (Riemannian) maximum principle and an original argument about divisibility and fiberwise analyticity.  Some of these  findings  are also extended to classical Riemannian solutions by using the eigenvalues of a Laplacian. When $\la_i\neq 0$, the Palatini conclusions fail necessarily, however, a good number of properties of the solutions remain. The framework and proofs are built up in detail.    
\end{abstract}
\maketitle

\tableofcontents



%
%
%
%
%
%
%
%
%
%
%
%
\section{Introduction}

Recently, the interest in Finslerian modifications of General Relativity has grown  \cite{BaDra12,observers,Cas21,CaSta18,EZDGH20,FuPa16,HaPer19,HPV20,KSS14,LeLi17,LLN17,ming3,RaVa15}  motivated in part by the role of Finsler Geometry in the Standard-Model Extension \cite{EdKos18,Kos11,KRT12} and Lorentz violation. The   search for an extension of the Einstein equations to this setting emerges as a fundamental issue. A first way to find them is  to consider Finslerian generalizations of the Einstein tensor $\bf G$,  having several  alternatives  \cite{LiCh07,miron-anastasiei,Rutz,Sta09,Vacaru}. A  second  way is provided by Hilbert's variational approach, developed by Hohmann, Pfeifer, Voicu   and Wohlfarth \cite{HPV,HPV20,pfeifer-wolfhart}, these authors take the natural generalization $\mathscr{S}$ of the Hilbert functional. This $\mathscr{S}$ is given by the integral of the $0$-homogeneized   (Finslerian) Ricci scalar of any Lorentz-Finsler metric $L$   for   a given manifold $M$ (see \cite{HPV21} for a general framework dealing with action functionals of arbitrary homogeneous fields).   The corresponding  Euler-Lagrange equation leads to a  scalar which, when restricted to Lorentzian metrics, yields naturally a tensor field; this tensor is not  exactly equal to $\bf G$,  
but it  still  leads to the same vacuum equations for such metrics. The aim of the present article is to deepen in the  variational approach to the Einstein equations  by considering the so-called Palatini formalism\footnote{This is the usual name in textbooks, even though the approach was actually invented in 1925 by Einstein \cite{Francaviglia}. Anyway, the name is maintained here so that it is distinguished from more general metric-affine formalisms.} for pseudo-Finsler metrics of arbitrary signature, paying special attention to the Lorentzian and positive definite cases. Let us notice that there are also some works that study Finslerian
Einstein manifolds with a variational approach, such as \cite{ChSh08} (which overcomes certain issues encountered in\footnote{See D. Bao's report in Mathematical Reviews, MR1365208 (99m:53130).} \cite{Akbar}). In particular,   in \cite{ChSh08} the authors   use a similar functional to that of \cite{HPV,pfeifer-wolfhart} but dividing by  the total volume in a positive definite setting. Another different approach is the one in \cite{Asanov}, where, indeed, the author explores several possibilities, using in particular the concept of osculation.   Finally, beyond pseudo-Finsler geometry, in \cite{Sta20} variational equations for any Sasaki-type metric on the tangent bundle of $M$ are derived by taking the Palatini formalism into account.  

Recall that the classical Palatini approach considered the affine connection $\nabla$ and the pseudo-Riemanian metric $g$ as independent variables for the Hilbert functional and, given $g$, it recovered its Levi-Civita connection $\nabla^g$ as the unique symmetric solution of the Euler-Lagrange affine equation for $\nabla$ (the properties of the non-symmetric ones are also known \cite{palatini}). This was a milestone for the mathematical foundations of Relativity because it ensured that the connection $\nabla$ which describes gravity is the same one as  the connection $\nabla^g$ which provides the critical points of the length or energy functionals for curves. Thus, light rays and free falling particles are  unequivocally  described by this unique connection. In the Finslerian setting, to ensure such a consistency is a much more  prioritary  task, because there is a huge freedom  when looking for  associated (linear or nonlinear) connections.

Consistently, here we will maintain the functional 
$\mathscr{S}$ but its variables will be  the  nonlinear connection $\N$ and the pseudo-Finsler metric $L$. Notice that no other kind of (linear) Finsler connection is required for the construction of the Ricci scalar. That is, $(\N,L)$ is enough for our functional  and we remain formally close to the classical Palatini setting, thus obtaining  coupled \textit{affine} \eqref{eq:affine equation} and \textit{metric} \eqref{eq:metric equation} Palatini equations.  However, further functionals should be tractable with the basic ingredients that we will develop.  

The  central  question is, given $L$, to what extent its associated nonlinear $\LC$ is the unique affine solution $\N$. In the pseudo-Riemannian case, a simple argument shows that all of these can be written as $\nabla^g+ \mathcal{A}\otimes\mathrm{Id}$, where the arbitrary  1-form $\mathcal{A} \equiv \mathcal{A}_i(x)$ ($\mathrm{Id}\equiv \delta^i_j$ is the identity tensor) determines the torsion \cite{palatini}. In the Finslerian case, the torsion part of $\N$ becomes $\mathcal{A}\otimes \canon$ with $\mathcal{A} \equiv \mathcal{A}_i(x,y)$ ($\canon\equiv y^a\partial_{y^a}$ is Liouville's) and the problem is reduced to the case of symmetric $\N$.  That is, as a first result (Th. \ref{thm:reduction symmetric}, Cor. \ref{cor:projection}): 
\begin{quote}
{\bf Theorem A}. Given a pseudo-Finsler metric $L$, the solutions of the affine equation  have a fibered structure on the  symmetric solutions with fiber isomorphic  to the space of  anisotropic (0-homogeneous) $1$-forms  $\mathcal{A}$, 
so that,  for each solution $\N$, there is a unique symmetric one $\prsim(\N)$ such that $\N=\prsim(\N) + \mathcal{A}\otimes \canon$ for some $\mathcal{A}$.  
\end{quote}

However, the symmetric case is not trivial, as $\N$ is governed by a PDE at each $p\in M$. Even more, the following subtlety appears for global uniqueness at $p$: when $L$ is indefinite, its domain $A\subseteq\mathrm{T}M\setminus\mathbf{0}$ is naturally  conic,  being $L_{\partial A}=0$, as the indicatrix (and some homogeneous elements) becomes ill-defined at $\partial A$. Notice also that, in Lorentzian signature, $A$ would correspond to the future-directed timelike directions, and the restriction to these (including the  future-directed  lightlike directions as a limit) is well motivated by physical interpretations  \cite{observers}.   However, we will develop  (fiberwise)  global techniques which work for \textit{proper} solutions, i.e., smoothly extendible to $\partial A$ (defns.~\ref{defn:proper}, \ref{DEF_5.1}). The fibered structure in Theorem A is naturally transferred to the proper solutions (Prop. \ref{prop:symmetric proper reduction}) and we prove the existence of a unique fibre in relevant general cases such as the following (see Th. \ref{thm: analytic uniqueness}):

\begin{quote}
{\bf Theorem B}.  Any analytic proper indefinite pseudo-Finsler metric $L$ admits  at most one analytic proper symmetric solution $\N$ of the affine variational equation \eqref{eq:affine equation}.
\end{quote} 
The proof relies on an original divisibility argument which is developed in full detail (Lem. \ref{lem:powers of L}). Moreover, we emphasize that the essential property at this point is just \textit{fiberwise analyticity} (Def. \ref{DEF_5.5}, Rems. \ref{REM_5.6}, \ref{REM_5.10}). This is much weaker than analyticity and, indeed, it holds trivially for all the smooth (non-analytic) affine and pseudo-Riemannian elements.

We also give other arguments, based on the maximum principle and the eigenvalues of the Laplacian, which  yield  some extensions of Th. A without fiberwise analyticity (Th. \ref{thm:C_i =00003D 0 Lorentz-Finsler}, Cor \ref{COR_5.14}), as well as applications to the positive definite case (Th. \ref{thm:C_i = 0 riemannian}). 
These arguments provide also the proof of the following result (Th. \ref{TH_5.17}), which is relevant for the  metric Palatini equation. 
\begin{quote}
{\bf Theorem C.} Let $L$ be a   (properly)   Lorentz-Finsler metric  and $\N$ any nonlinear connection  smoothly extendible to $\partial A$ with Ricci scalar  $\ri$. If the Einstein-type scalar $\left(\di+2\right)\ri-L\,g^{ab}\,\ri_{\cdot a\cdot b}$ vanishes, then $\ri$ vanishes too.
\end{quote}
Indeed,  when the mean Landsberg tensor $\la_i$ vanishes,  as it occurs in the classical case, this equation agrees with the one obtained by the Hilbert approach (i.e., the aforementioned in \cite{HPV}). So, the result above is relevant for the consistency of the vacuum Einstein equations. In comparison with the elementary pseudo-Riemannian case (Rem. \ref{REM_5.18}),  where it is valid in any signature, our result is technically more complicated and has a properly Finslerian applicability. As the aforementioned  results,  it  relies on Lem.~\ref{lem:equation f},  also proven in full detail.
 
To complete the approach, one should check at what extent the natural (Berwald) nonlinear connection $\LC$ associated with $L$ plays a   role  similar to that which $\nabla^g$ plays in the classical Palatini setting.   Notice that $N^L$ is naturally associated with the geodesic spray of $L$, so this issue is related to the Palatini physical interpretations about free falling observers. The solution involves  the \textit{Landsberg} tensor $\la$ 
or, more precisely, the mean Landsberg $\la_i= \la^a_{  ai}$ (see Cor. \ref{cor:projection}, Rem. \ref{REM_4.15}, Prop. \ref{prop:metric compatibility 2}, Rem.~\ref{rem:metric compatbility geodesics}):
 
\begin{quote}
 {\bf Theorem D}.  Given a pseudo-Finsler $L$, its nonlinear Berwald connection $\LC$ is a solution of the affine variational equation \eqref{eq:affine equation} iff $\la_i = 0$.

In this case, any other solution $\N$  shares  its  pregeodesics with $\LC$ iff it lies in the same fiber, i.e.,  $\N=\LC + \mathcal{A}\otimes \canon$ for some $\mathcal{A}$; then, it shares  geodesics iff   $\mathcal{A}_a\,y^a=0$.  

Otherwise, when $\la_i$ does not vanish identically, neither $\LC$ is a solution nor any solution $\N$ can share pregeodesics with $\LC$.

In any case, when $L$ and $\N$ are proper,  any $\N$-geodesic $\gamma$ has constant sign of $L(\dot\gamma)$.  Moreover, in the Lorentz-Finsler case (no matter how $\la_i$ is), the causal character (timelike, lightlike) of the $\N$-geodesics does not change,  the lightlike $\N$-geodesics coincide with the corresponding $L$-geodesics  and, hence, the lightlike $\N$-pregeodesics are the cone (pre-)geodesics inherent to the $L$-cone structure.

\end{quote} 
It is worth pointing out that the properties about sharing geodesics and pregeodesics hold not only for the fiber  of $\LC$ but also for any other fiber of solutions (with independence of $\la_i$). Moreover, further compatibility   conditions   of $\nabla$ and $L$ appear for connections differing only in some $\mathcal{A}\otimes \canon$   from a symmetric one   (not necessarily solutions), see Prop. \ref{prop:metric compatibility 1}. As a summary of all these results:
\begin{quote}
\textit{When  $\la_i = 0$, the fibered structure of the affine solutions, the fact that $\LC$ determines one of such fibers,  
the uniqueness of this fiber under mild conditions (properness, fiberwise analyticity), the subsequent  status of $\LC$ as the unique symmetric solution, and the fact that all  these solutions share pregeodesics (those   of $L$), recover and extend naturally all the conclusions of the classical Palatini formalism for the connection (apart from those for the metric, at least in the vacuum case). However, no such extension is possible when $\la_i \neq 0$.}
\end{quote}
 As commented above in Theorem D, when  $\la_i \not= 0$,  the solutions $\N$ of the affine equation do not share   pregeodesics   with $L$. This fact can have several interpretations. Taking into account that the main goal of the Hilbert functional is to obtain the Einstein field equations, one could infer that the solutions $\N$ are very suitable for computing them. Nevertheless, it is not clear which is the best connection to compute the trajectories of the Finsler spacetime. The connections $\N$   relate more closely the Jacobi equation to our field equation, whereas   the geodesics of $L$ satisfy a variational principle. 

From the  technical viewpoint, we introduce detailedly all the elements we need, which are spread in the literature under different viewpoints and implicit frameworks. Full proofs of the results are also provided (including straightforward but lengthy computations) to permit traceability. 

With this spirit, in \S \ref{s.2} the required ingredients on Finsler Geometry and anisotropic calculus are introduced. The so-called   \textit{Finslerian connections} \cite{dahl, minguzzi}, i.e., pairs $(\N,\nabla^*)$ composed  by a nonlinear $\N$ and a linear connection $\nabla^*$, the latter for the vertical bundle $\mathrm{V}A\longrightarrow A$, do not really enter into our work;  
instead, anisotropic connections \cite{anisotropic, mediterranean} will suffice and will introduce a simple and intuitive Koszul derivative directly on $M$. Anyway, any anisotropic connection $\nabla$ can be identified canonically with a vertically trivial $\nabla^*$ (see \cite{gelocor} for this and other results linking both approaches), so the readers tied to this classical  framework  can rewrite our computations in the way they prefer. In \S \ref{s.3}, the  metric-affine (Palatini)  variational calculus is developed. Here, independently, $L$ yields  the indicatrix $\left\{L=1\right\}$ and a volume element, while $\N$ yields the Ricci scalar (Remark \ref{rem:vol independent of N}). Full details of the proofs of the affine and metric equations, as well as of the crucial divergence formula in the suitably projectivized space, are provided in the Appendices. In \S
 \ref{sec:affine equation}, the study of the solutions for $\N$ is reduced to the symmetric case, including the fibered structure of the space of solutions and the properties shared by the elements of each fiber (Cor. \ref{cor:projection}). Moreover, a detailed study of the different types of metric and geodesic compatibility  for the solutions  is carried out (Props. \ref{prop:metric compatibility 1}, \ref{prop:metric compatibility 2}, \ref{prop:metric compatibility 3}).
Finally,  in \S \ref{sec:proper solutions}, the  main results on proper solutions are distributed into two subsections, the first one on techniques related to divisibility  by $L$  (eventually using fiberwise analyticity), and the second one related to the maximum principle. Using both types of results, the classical solutions 
are revisited in the last subsection.   
 
\section{Standard geometric objects}\label{s.2}
The main aim of this section is to fix notation and conventions. 

Let $M$ be a connected\footnote{Only for simplicity. In general, all of our developments are valid on each connected component of $M$.} smooth\footnote{This will mean $\mathcal{C}^\infty$ and all the objects will be smooth. Nevertheless, some results may not need so much regularity. For instance, those of \S \ref{sec:elliptic} only require a finite number of vertical derivatives existing with continuity at each $p\in M$.} manifold of dimension\footnote{In dimension $1$ our action functional would trivialize.} $n\geq2$. The Einstein convention is employed, the indices $a$, $b$, $c$, $d$, $e$, $i$, $j$, $k$, $l$ run in the set $\left\{1,...,\di\right\}$, and for clarity, we use $i$, $j$, $k$ as free indices and $a$, $b$, $c$, $d$, $e$ as summation indices. Charts $(U,x=(x^1,...,x^\di))$ for $M$ induce natural charts $(\mathrm{T}U,(x,y)=(x^1,...,x^\di,y^1,...,y^\di))$ for $\mathrm{T}M$. Putting $\partial_i:=\partial/\partial x^i$ and $\dot \partial_i:=\partial/\partial y^i$, under a change $(U,x)\rightsquigarrow(\bar{U},\bar{x})$, 
\[
\bar{\partial}_i=\frac{\partial x^a}{\partial \bar{x}^i}\,\partial_a+\bar{y}^b\,\frac{\partial^2 x^a}{\partial \bar{x}^b\,\partial \bar{x}^i}\,\dot{\partial}_a,\qquad\dot{\bar{\partial}}_i=\frac{\partial x^{a}}{\partial\bar{x}^{i}}\,\dot{\partial}_a
\]
as local vector fields on $\mathrm{T}M$. Let $A\subseteq\mathrm{T}M$ be open with $\pi(A)=M$ for $\pi$ the natural projection. The restriction $\pi_A\vcentcolon A\longrightarrow M$ defines a fibered manifold with \textit{fibers} $A_p:=A\cap\mathrm{T}_pM$ ($p\in M$) and \textit{vertical distribution} $\mathrm{V}A\longrightarrow A$,
\[
\mathrm{V}_vA:=\mathrm{Ker}\,\mathrm{T}_v\pi_A=\mathrm{T}_v(A_{\pi(v)})=\mathrm{Span}\left\{ \left.\dot{\partial}_i\right|_{v}\right\}\subseteq\mathrm{T}_vA
\] 
($v\in A$, where $\mathrm{T}_v\pi_A$ is the \textit{tangent map} or \textit{differential} of $\pi_A$). The reader is referred to \cite{krupka} for the general theory of fibered manifolds. We shall employ the framework of the anisotropic tensors \cite{anisotropic,mediterranean}; especially, the viewpoint and conventions of \cite{gelocor} can be helpful for the reader. An \textit{$r$-contravariant $s$-covariant $A$-anisotropic tensor} is a section $T$ of the pullback bundle 
\[
\pi_A^\ast(\overset{r)}{\bigotimes}\mathrm{T}M\otimes\overset{s)}{\bigotimes}\mathrm{T}^\ast M)\longrightarrow A;
\] 
we denote by $\mathcal{T}_s^r(\MA)$ the space of such sections. They have locally the form 
\[
T_v=T_{b_1,...,b_s}^{a_1,...,a_r}(v)\left.\partial_{a_1}\right|_{\pi(v)}\otimes...\otimes\left.\partial_{a_r}\right|_{\pi(v)}\otimes\mathrm{d}x^{b_1}_{\pi(v)}\otimes...\otimes\mathrm{d}x^{b_s}_{\pi(v)}
\]
for certain $T_{j_1,...,j_s}^{i_1,...,i_r}(x,y)$'s defined on $A\cap\mathrm{T}U$ that transform tensorially under $(U,x)\rightsquigarrow(\bar{U},\bar{x})$. 
There is a \textit{vertical isomorphism} identifying anisotropic with vertical vector fields on $A$: 
\begin{equation} \label{e_isomorfismo vertical}
X_{v}=X^{a}(v)\left.\partial_a\right|_{\pi(v)}\in\mathrm{T}_{\pi(v)}M\longleftrightarrow X_{v}^{\mathrm{V}}=X^{a}(v)\left.\dot{\partial}_a\right|_{v}\in\mathrm{V}_{v}A
\end{equation}
 (notice that when the $X^i$'s are constant on a fiber $A_p$, this formula makes explicit the identification between the vertical spaces at the different $v\in A_p$). In particular, the canonical anisotropic vector $\canon\in\mathcal{T}_0^1(\MA)$ defined by 
\begin{equation}
	\canon_v=v=y^a(v)\left.\partial_a\right|_{\pi(v)}
	\label{eq:canon}
\end{equation}
corresponds to the  \textit{Liouville vector field} $\canon^{\mathrm{V}}$ \cite{minguzzi,bucataru,HPV} (note that in the last two references $\canon$ is used for what we denote $\canon^{\mathrm{V}}$). The vertical derivatives 
\[
T_{j_1,...,j_s\,\cdot j_{s+1}}^{i_1,...,i_r}(x,y):=\dv_{j_{s+1}}T_{j_1,...,j_s}^{i_1,...,i_r}(x,y)=\frac{\partial T_{j_1,...,j_s}^{i_1,...,i_r}}{\partial y^{j_{s+1}}}(x,y)
\]
define a new anisotropic tensor: the \textit{vertical differential} of $T$; we denote it by $\dv T\in\mathcal{T}_{s+1}^r(\MA)$ and by $\dv_X T\in\mathcal{T}_s^r(\MA)$ its contraction with $X$ in the new index. For instance, 
\[
\dv_\canon T=y^{b_{s+1}}\,T_{b_1,...,b_s\,\cdot b_{s+1}}^{a_1,...,a_r}\,\partial_{a_1}\otimes...\otimes\partial_{a_r}\otimes\mathrm{d}x^{b_1}\otimes...\otimes\mathrm{d}x^{b_s}.
\]

 An anisotropic tensor $T$ can actually be \textit{isotropic}, in that $T_{j_1,...,j_s}^{i_1,...,i_r}(x,y)=T_{j_1,...,j_s}^{i_1,...,i_r}(x)$. This is equivalent to the constancy of the restriction $T_p$ to each fiber $A_p$ ($p\in M$). Hence, it means that $T$ reduces to a tensor field on $M$, which we will not distinguish notationally from $T$ itself. 

\subsection{Homogeneous tensors}  The following three notions of (positive) homogeneity are extracted from \cite{anisotropic} and \cite[Defs. 1.5.2 and 1.5.3]{bucataru} respectively. 

\begin{defn} \label{def:homogeneity}  $A$ is \textit{conic} if $A\subseteq\mathrm{T}M\setminus\mathbf{0}$ and $\lambda\,v\in A$ for all $v\in A$, $\lambda\in\mathbb{R}^+$. In such a case, let $\alpha\in\mathbb{R}$.
	\begin{enumerate}
		\item $T\in\mathcal{T}_s^r(M_A)$ is \textit{$\alpha$-homogeneous} if $T_{\lambda\,v}=\lambda^{\alpha}\,T_v$. That is, its coordinates are $\alpha$-homogeneous (in $y$):   $T_{j_1,...,j_s}^{i_1,...,i_r}(x,\lambda y)=\lambda^\alpha\, T_{j_1,...,j_s}^{i_1,...,i_r}(x,y)$.
		\item A vector field $\mathscr{X}$ on $A$ is \textit{$\alpha$-homogeneous} if $\mathscr{X}_{\lambda\,v}=\lambda^{\alpha-1}\left(\mathrm{T}h_{\lambda}\right)_v(\mathscr{X}_v)$, where $h_\lambda\vcentcolon A\longrightarrow A$, $h_\lambda(v)=\lambda\,v$. That is, if $\mathscr{X}=\mathscr{X}^a\,\partial_a +\mathscr{X}^{\di+a}\, \dot\partial_a$, then $\mathscr{X}^i(x,y)$ and $\mathscr{X}^{\di+i}(x,y)$ are, resp., $(\alpha-1)$- and $\alpha$-homogeneous.
		\item An $s$-form $\omega$ on $A$ is \textit{$\alpha$-homogeneous} if $\left(\mathrm{T}h_{\lambda}\right)_v^{\ast}(\omega_{\lambda\,v})=\lambda^{\alpha}\,\omega_v$, ($\ast$ means pullback). That is, if $\omega_{i_1,...,i_{\mu}\mid j_1,...,j_{\nu}}$ is the component of $\omega$ on $\mathrm{d}x^{i_1}\wedge...\wedge\mathrm{d}x^{i_{\mu}}\wedge\mathrm{d}y^{j_1}\wedge...\wedge\mathrm{d}y^{j_{\nu}}$ ($\mu+\nu=s$), then $\omega_{i_1,...,i_{\mu}\mid j_1,...,j_{\nu}}(x,y)$ is $(\alpha-\nu)$-homogeneous.
	\end{enumerate}
	Moreover, $\mathrm{h}^{\alpha}\mathcal{T}_s^r(M_A)$ and $\mathrm{h}^{\alpha}\mathcal{F}(A):=\mathrm{h}^{\alpha}\mathcal{T}_0^0(M_A)$ will denote the space of $\alpha$-homogeneous anisotropic tensors and functions, resp. 
\end{defn}

Clearly, $\dv\vcentcolon\mathrm{h}^{\alpha}\mathcal{T}_s^r(\MA)$ $\longrightarrow\mathrm{h}^{\alpha-1}\mathcal{T}_{s+1}^r(\MA)$ is a well-defined linear morphism. The items (i) and (ii) are consistent with the identification of anisotropic  and vertical vector fields in \eqref{e_isomorfismo vertical}. In particular, both $\canon$ and $\canon^\mathrm{V}$ are $1$-homogeneous, whereas any isotropic tensor field ($T_{j_1,...,j_s}^{i_1,...,i_r}(x,y)=T_{j_1,...,j_s}^{i_1,...,i_r}(x)$) is $0$-homogeneous. The homogeneities of the coordinates of a $1$-form $\omega=\omega_{a\mid}\,\mathrm{d}x^{a}+\omega_{\mid a}\,\mathrm{d}y^{a}$ are switched with respect to those of $\mathscr{X}=\mathscr{X}^a\,\partial_a +\mathscr{X}^{\di+a}\, \dot\partial_a$ in concordance with the intrinsic meanings of $\mathscr{X}^{i}= 0$ and $\omega_{\mid i}=0$.  The above expressions in coordinates and Euler's Theorem yield directly the following characterizations (consistently with \cite[(6)]{anisotropic} and \cite[Ths. 1.5.2 and 1.5.3]{bucataru}). 

\begin{prop} \label{prop:euler}  Assume that $A$ is conic. Then: 
\begin{enumerate}  
		\item  $T\in\mathcal{T}_s^r(M_A)$ is in $\mathrm{h}^{\alpha}\mathcal{T}_s^r(M_A)$ if and only if $\dv_\canon T=\alpha\,T$, i.e., 
		\[
		y^{b_{s+1}}\,T_{b_1,...,b_s\,\cdot b_{s+1}}^{a_1,...,a_r}(x,y)=\alpha\,T_{b_1,...,b_s}^{a_1,...,a_r}(x,y).
		\] 
		
		\item  A vector field $\mathscr{X}$ on $A$ is $\alpha$-homogeneous if and only if its Lie derivative along the Liouville field satisfies $\lie_{\canon^{\mathrm{V}}}(\mathscr{X})=\left(\alpha-1\right)\mathscr{X}$. 
		
		\item  An $s$-form $\omega$ on $A$ is $\alpha$-homogeneous if and only if $\lie_{\canon^{\mathrm{V}}}(\omega)=\alpha\,\omega$. 
\end{enumerate} 
\end{prop} 

The \textit{positive projectivization} of the conic $A$ plays the same role in our variational calculus as in \cite{HPV}. We denote it by $\mathbb{P}^+A$, so that $\mathbb{P}\vcentcolon A\longrightarrow\mathbb{P}^+A$, $v\longmapsto \mathbb{P}^+v$, is the natural projection. The $0$-homogeneous $s$-forms on $A$ induce $\left(s-1\right)$-forms on $\mathbb{P}^+A$. This correspondence was implicitly taken into account in the notation of \cite{HPV}, but we state it in ours for the reader's convenience. 

\begin{prop}  \label{prop:inducing forms} Assume that $A$ is conic, and let  $\omega$ be a  $0$-homogeneous $s$-form and $\mathscr{X}$ a $1$-homogeneous vector field there. Then:
	\begin{enumerate}
		\item \label{aa} The interior product  $\mathscr{X}\lrcorner\omega$ is $0$-homogeneous as well. 
		\item In the case $\mathscr{X}=\canon^{\mathrm{V}}$, this interior product is the pullback of a unique $(s-1)$-form on $\mathbb{P}^+A$. We denote this one by $\underline{\omega}$, so that 
		\begin{equation}
			\canon^\mathrm{V}\lrcorner\omega=\left(\mathbb{P}^+\right)^\ast\underline{\omega}.
			\label{eq:underline}
		\end{equation}
		Moreover, $\underline{\omega}$ vanishes at $\mathbb{P}^+v\in\mathbb{P}^+A$ if and only if $\canon^\mathrm{V}\lrcorner\omega$ vanishes at one, and hence all, representatives $v$ of $\mathbb{P}^+v$.  
		\item The exterior differential $\mathrm{d}\omega$ is $0$-homogeneous too with 
		\[
		\underline{\mathrm{d}\omega}=-\mathrm{d}\underline{\omega}.
		\]
	\end{enumerate} 
\end{prop} 

\begin{proof}
	 (i) This is clear from the expression in coordinates of $\mathscr{X}\lrcorner\omega$ and Def. \ref{def:homogeneity} (iii). 
	
	 (ii) In order to define $\underline{\omega}$ at $\mathbb{P}^+v\in\mathbb{P}^+A$, one has to specify how it acts on $s$ vectors in $\mathrm{T}_{\mathbb{P}^+v}\mathbb{P}^+A$. As $\mathrm{T}_{v}\mathbb{P}^+\vcentcolon\mathrm{T}_{v}A\longrightarrow\mathrm{T}_{\mathbb{P}^+v}\mathbb{P}^+A$ is onto, those are always of the form $\mathrm{T}_{v}\mathbb{P}^+u_1$, ..., $\mathrm{T}_{v}\mathbb{P}^+u_s$ for some $u_1,...,u_s\in\mathrm{T}_{v}A$. And as \eqref{eq:underline} must be satisfied, the only possibility is to define
	\[
	\begin{split}
	\underline{\omega}_{\mathbb{P}^+v}(\mathrm{T}_{v}\mathbb{P}^+u_1,...,\mathrm{T}_{v}\mathbb{P}^+u_s)(=:\left\{\left(\mathbb{P}^+\right)^\ast\underline{\omega}\right\}_v(u_1,...,u_s))=&\left(\canon^\mathrm{V}\lrcorner\omega\right)_v(u_1,...,u_s) \\
	=&\omega_v(\canon^\mathrm{V}_v,u_1,...,u_s)
	\end{split}
	\]
	(where $\canon^\mathrm{V}_v$ is just $v$ under the natural identification $\mathrm{T}_{\pi(v)}M\equiv\mathrm{V}_v A\subseteq\mathrm{T}_v A$, recall \eqref{eq:canon}). Finally, it is straightforward to see that this definition is consistent: the property $\mathrm{Ker}\,\mathrm{T}_{v}\mathbb{P}^+=\mathrm{Span}\left\{\canon^\mathrm{V}_v\right\}$ allows one to check that it is independent of the representatives $u_{\mu}$ of $\mathrm{T}_{v}\mathbb{P}^+u_{\mu}$, whereas the properties $\left(\mathrm{T}h_{\lambda}\right)_v^{\ast}(\omega_{\lambda\,v})=\omega_v$ and $\canon^\mathrm{V}_{\lambda\,v}=\left(\mathrm{T}h_{\lambda}\right)_v(\canon^\mathrm{V}_v)$ allow one to check that it is independent of the representative $v$ of $\mathbb{P}^+v$. Finally, from the construction with arbitrary $\left\{u_1,...,u_s\right\}$, it is clear that $\underline{\omega}_{\mathbb{P}^+v}=0$ if and only if $\omega_v(\canon^\mathrm{V}_v,-,...,-)=0$. 
	
	 (iii) Prop. \ref{prop:euler} (iii), Cartan's formula for the Lie derivative and $\lie_{\canon^\mathrm{V}}(\omega)=0$ give the $0$-homogeneity of $\mathrm{d}\omega$: 
	\[
	\lie_{\canon^\mathrm{V}}(\mathrm{d}\omega)=\canon^\mathrm{V}\lrcorner\mathrm{d}\mathrm{d}\omega+\mathrm{d}(\canon^\mathrm{V}\lrcorner\mathrm{d}\omega)=\mathrm{d}(\canon^\mathrm{V}\lrcorner\mathrm{d}\omega)=\mathrm{d}(\lie_{\canon^\mathrm{V}}(\omega))-\mathrm{d}\mathrm{d}(\canon^\mathrm{V}\lrcorner\omega)=0.
	\] 
	For the last assertion, it suffices to see that $-\mathrm{d}\underline{\omega}$ satisfies the property that defines $\underline{\mathrm{d}\omega}$. Using the same properties as above,
	\[
	\left(\mathbb{P}^+\right)^\ast(-\mathrm{d}\underline{\omega})=-\mathrm{d}\left(\mathbb{P}^+\right)^\ast\underline{\omega}=-\mathrm{d}(\canon^\mathrm{V}\lrcorner\omega)=-\lie_{\canon^\mathrm{V}}(\omega)+\canon^\mathrm{V}\lrcorner\mathrm{d}\omega=\canon^\mathrm{V}\lrcorner\mathrm{d}\omega, 
	\]
	so indeed $-\mathrm{d}\underline{\omega}=\underline{\mathrm{d}\omega}$. 
\end{proof}

\subsection{Homogeneous connections}  There are a number of equivalent ways of defining the connections that we work with; most of them were discussed in \cite{gelocor}. Here, motivated by the spirit of the variational calculus, we choose alternative definitions that present the connections as sections of certain affine bundles over $A$. Then we pass to their coordinates, to ensure that we indeed are working with the same objects as in \cite[(5) and (12)]{gelocor}. This conveys notational differences: for instance, when anisotropic connections are regarded as sections, we denote them by $\an$, and when they are regarded as \textit{Koszul covariant derivations}, we denote them by $\nabla$. As a last comment, we will always work with homogeneous objects (even if we keep mentioning their homogeneity), so from now onward we assume that $A$ is conic. 

 Consider affine connections on $M$ (i.e., linear connections for $\mathrm{T}M\longrightarrow M$). Their Christoffel symbols $\an_{ij}^k(x)$ have the transformation cocycle
\begin{equation}
	\bar{\an}_{ij}^{k}(x)=\frac{\partial\bar{x}^{k}}{\partial x^{c}}(x)\,\frac{\partial^{2}x^{c}}{\partial\bar{x}^{i}\,\partial\bar{x}^{j}}(x)+\frac{\partial\bar{x}^{k}}{\partial x^{c}}(x)\,\frac{\partial x^{a}}{\partial\bar{x}^{i}}(x)\,\frac{\partial x^{b}}{\partial\bar{x}^{j}}\,\an_{ab}^{c}(x)
	\label{eq:cocycle gamma}
\end{equation}
under changes of charts. Using an analogous of \cite[\S 6.4]{michor}, one can check that this cocycle determines an affine bundle $\mathbf{C}M\longrightarrow M$, which is so that its sections are precisely the affine connections on $M$.\footnote{A more specific presentation of this affine bundle is given as follows. Given $p\in M$, say that two affine connections on $M$ are \textit{equivalent at p} if when they act on any vector fields on $M$, the results coincide at $p$ for both connections. Then the equivalence classes are the elements of the fiber $\mathbf{C}_pM$. Hence, it is clear that an affine connection yields such an element at each $p$. } 

\begin{defn}
	A \textit{homogeneous $A$-anisotropic connection} is a section $\an$ of the pullback affine bundle $\pi_A^\ast(\mathbf{C}M)\longrightarrow A$  (hence a map $v\in A\longmapsto\an_v\in\mathbf{C}_{\pi(v)}M$)  subject to $\an_{\lambda\,v}=\an_{v}$. 
\end{defn} 

\begin{rem}  The construction of $\mathbf{C}M\longrightarrow M$ guarantees that such a $\an$ has natural coordinates $\an_{ij}^k(x,y)$, while the condition $\an_{\lambda\,v}=\an_{v}$ translates into the $0$-homogeneity of those. This means that a (homogeneous) anisotropic connection in the sense above is equivalent to a collection of ($0$-homogeneous) functions $\an_{ij}^k$ on $A\cap\mathrm{T}U$ associated with each chart such that, under changes $(U,x)\rightsquigarrow(\bar{U},\bar{x})$, \eqref{eq:cocycle gamma} is satisfied with $\bar{\an}_{ij}^{k}(x,y)$, $\an_{ab}^{c}(x,y)$ in place of $\bar{\an}_{ij}^{k}(x)$, $\an_{ab}^{c}(x)$. By \cite[Prop. 1 (2)]{gelocor}, it is also equivalent to a (homogeneous) anisotropic connection $\nabla$ in the sense of \cite[Def. 4]{gelocor}, \cite[Def. 3.1]{anisotropic}. Hence, as announced, the viewpoint here is unified with the one of those references and all the developments in \cite{gelocor,anisotropic} can be applied. 
\end{rem}

 Consider now the $1$-jet prolongation $\mathbf{J}^1A\longrightarrow A\longrightarrow M$; one is referred to \cite[\S 12]{michor} for a systematic treatment of jets. Recall that for $p\in M$, two local $A$-valued vector fields $V$, $V^{\prime}$ on $M$ \textit{determine the same $1$-jet at $p$} if they and their first order partial derivatives (on any chart) coincide at $p$. These \textit{$1$-jets} (equivalence classes) $\jmath_{p}^1V$ are the elements of the fiber $\mathbf{J}^1_pA$ of $\mathbf{J}^1A\longrightarrow M$, but also $\jmath_{p}^1V\longmapsto V_p$ is a well-defined projection and one obtains $\mathbf{J}^1A\longrightarrow A$, which is an affine bundle. The following definition is standard in the theory of fibered manifolds, see \cite[\S 17.1]{michor} for instance. 

\begin{defn}
	A \textit{homogeneous nonlinear} (or \textit{Ehresmann}) \textit{connection for $A\longrightarrow M$} is a section $\N$ of $\mathbf{J}^1A\longrightarrow A$  (hence a choice of $1$-jet $\N_v=\jmath_{\pi(v)}^1V$ with $V_{\pi(v)}=v$ at each $v\in A$)  with the requirement that if $\N_v=\jmath_{\pi(v)}^1V$, then $\N_{\lambda\,v}=\jmath_{\pi(\lambda\,v)}^1\left(\lambda\,V\right)$.
\end{defn}

\begin{rem} \!
	 (A) Knowing that $V_{\pi(v)}=v$, the $1$-jet $\N_v=\jmath_{\pi(v)}^1V$ is determined by the partial derivatives $\N_i^k(v)=-\partial_iV^k(\pi(v))$; these are functions $\N_i^k(x,y)$, while the condition $\N_{\lambda\,v}=\jmath_{\pi(\lambda\,v)}^1\left(\lambda\,V\right)$ translates into their $1$-homogeneity. This means that a (homogeneous) nonlinear connection is equivalent to a collection of ($1$-homogeneous) functions $\N_{i}^k$ on $A\cap\mathrm{T}U$ associated with each chart such that, under changes $(U,x)\rightsquigarrow(\bar{U},\bar{x})$, the transformation cocycle 
	\begin{equation}
	\bar{\N}_{i}^{k}(x,y)=\frac{\partial\bar{x}^{k}}{\partial x^{c}}(x)\,\frac{\partial^{2}x^{c}}{\partial\bar{x}^{i}\,\partial\bar{x}^{b}}(x)\,\bar{y}^{b}+\frac{\partial\bar{x}^{k}}{\partial x^{c}}(x)\,\frac{\partial x^{a}}{\partial\bar{x}^{i}}(x)\,\N_{a}^{c}(x,y)
	\label{eq:cocycle N}
	\end{equation}
	is satisfied. By \cite[Rem. 3]{gelocor}, it is also equivalent to a (homogeneous) nonlinear connection in any of the usual senses; for instance, that of an \textit{(invariant by homotheties) horizontal distribution} $\mathrm{H}A\longrightarrow A$, where 
	\begin{equation}
	\mathrm{H}_{v}A:=\mathrm{Span}\left\{ \left.\delta_i\right|_{v}\right\}\subseteq\mathrm{T}_vA,\qquad\left.\delta_i\right|_{v}:=\left.\partial_i\right|_{v}-\N_{i}^{a}(v)\,\left.\dot{\partial}_a\right|_{v}.
	\label{eq:horizontal distribution}
	\end{equation}
	Hence, the perspective here is unified with the one of references such as \cite[\S 4]{gelocor}, \cite[\S 3]{minguzzi}, \cite[\S 4]{dahl} and \cite[Ch. 2]{bucataru}\footnote{ Even though the $\N_{i}^k$'s in this reference are not the same as ours (see the different cocycle \cite[(2.8)]{bucataru}), they necessarily are in correspondence with ours. }. The $\N$-horizontal distribution provides the \textit{$\N$-horizontal isomorphism}
	\begin{equation}
	X_{v}=X^{a}(v)\left.\partial_a\right|_{\pi(v)}\in\mathrm{T}_{\pi(v)}M\longleftrightarrow X_{v}^{\mathrm{H}}:=X^{a}(v)\left.\delta_a\right|_{v}\in\mathrm{H}_{v}A,
	\label{eq:horizontal isomorphism}
	\end{equation}
	which identifies $\mathrm{h}^\alpha\mathcal{T}_{0}^{1}(M_A)$ with the space of $\left(\alpha+1\right)$-homogeneous horizontal vector fields on $A$. 
	
	(B) From the cocycles \eqref{eq:cocycle gamma} (for $\an_{ik}^k(x,y)$) and \eqref{eq:cocycle N}, the affine structures of the spaces of homogeneous anisotropic and nonlinear connections are given respectively as follows. For a fixed $\an_0$ and $Q\in\mathrm{h}^0\mathcal{T}_{2}^{1}(M_A)$, $\an:=\an_0+Q$ has coordinates $\left(\an_0 \right)_{ij}^k+Q_{ij}^k$, while for a fixed $\N_0$ and $J \in\mathrm{h}^1\mathcal{T}_{1}^{1}(M_A)$, $\N:=\N_0+J$ has coordinates $\left(\N_0 \right)_{i}^k+J_{i}^k$.
\end{rem}

\begin{defn} \label{def:underlying 1} \!
	\begin{enumerate}
	\item By \cite[Th. 2 (1)]{gelocor}, any homogeneous anisotropic connection $\an$ induces canonically a homogeneous nonlinear connection of coordinates $\N_i^k=\an_{i\,a}^k\,y^a$. We call it the \textit{underlying nonlinear connection of $\an$}.
	\item By \cite[Th. 2 (2)]{gelocor}, any homogeneous nonlinear connection $\N$ induces canonically a homogeneous anisotropic connection of coordinates $\an_{ij}^k=\N_{i\,\cdot j}^k=\dot\partial_j\N_{i}^k$. We call it the \textit{vertical differential} or \textit{Berwald anisotropic connection of $\N$} and denote it by $\dv\N$.
	\end{enumerate}
\end{defn}

 Given any homogeneous anisotropic connection $\an$, the corresponding covariant derivative $\nabla$ maps $\mathrm{h}^\alpha\mathcal{T}_{s}^{r}(\MA)$ to $\mathrm{h}^\alpha\mathcal{T}_{s+1}^{r}(\MA)$. For $T\in\mathrm{h}^\alpha\mathcal{T}_{s}^{r}(\MA)$, $\nabla T$ is given in coordinates by 
\begin{equation}
\nabla_{j_{s+1}}T_{j_1,...,j_s}^{i_1,...,i_r}:=\delta_{j_{s+1}}T_{j_1,...,j_s}^{i_1,...,i_r}+\underset{\mu}{\sum}\an_{j_{s+1}a}^{i_\mu}\,T_{j_1,...,j_s}^{i_1,...,\overset{(\mu)}{a},...,i_r}-\underset{\mu}{\sum}\an_{j_{s+1}j_\mu}^a\,T_{j_1,...,\underset{(\mu)}{a},...,j_s}^{i_1,...,i_r},
\label{eq:covariant derivative}
\end{equation}
where the $\delta_j$ are those of \eqref{eq:horizontal distribution} for the underlying nonlinear connection (and thus underlying horizontal distribution) $\N$ of $\an$. In particular, for $f\in\mathrm{h}^\alpha\mathcal{F}(A)$ and $X\in\mathrm{h}^\alpha\mathcal{T}_{0}^{1}(\MA)$, $\nabla_Xf=X^\mathrm{H}(f)$ only depends on that underlying nonlinear connection. 

\begin{prop} \label{prop:covariante canon = 0}
	 For any anisotropic connection, $\nabla\canon=0$, i.e., $\nabla_j y^i=0$. 
\end{prop}

\begin{proof}
	 $\canon=y^a\,\partial_a\in\mathrm{h}^1\mathcal{T}_{0}^{1}(\MA)$, so by \eqref{eq:covariant derivative}, $\nabla\canon$ has coordinates
	\[
	\nabla_j y^i=\delta_j y^i+\an_{ja}^i\,y^a=\partial_j y^i-\N_j^a\,\dot\partial_a\,y^i+\an_{ja}^i\,y^a=-\N_j^a\,\delta_a^i+\an_{ja}^i\,y^a=0,
	\]
	where $\delta_a^i$ is the usual Kronecker's and only the fact that $\N$ is the underlying nonlinear connection of $\an$ was used for the last equality. 
\end{proof}

The \textit{curvature}, the \textit{(Finslerian) Ricci scalar} and the \textit{torsion\footnote{ Note that when defining, as in \cite[Def. 5]{gelocor}, the \textit{torsion} of any homogeneous anisotropic connection $\an$ by $\an_{ij}^k-\an_{ji}^k$, the torsion of $\N$ turns out to be just that of $\dv\N$. However, in this work we will reserve the notation $\torN$ for the torsion of a nonlinear connection. Compare with more abstract references such as \cite[\S 3.3]{minguzzi}, \cite[\S 7]{modugno}. } of a homogeneous nonlinear connection $\N$} can be regarded as homogeneous anisotropic tensors $\RN\in\mathrm{h}^{1}\mathcal{T}_2^1(M_A)$, $\ri\in\mathrm{h}^{2}\mathcal{F}(A)$ and $\torN\in\mathrm{h}^{0}\mathcal{T}_2^1(M_A)$ respectively, with coordinates
\begin{equation}
	\RN_{ij}^k=\delta_j\N_i^k-\delta_i\N_j^k,\qquad\ri=  y^b\,\RN_{ba}^a   ,\qquad\torN_{ij}^k=\N_{i\,\cdot j}^k-\N_{j\,\cdot i}^k
	\label{eq:curvature and torsion}
\end{equation}
(recall \eqref{eq:horizontal isomorphism}). We say that $\N$ is \textit{symmetric} when $\torN=0$.  By direct computation, one has the following commutation formulas: 
\begin{equation}
	\left[\delta_{i},\delta_{j}\right]=\RN_{ij}^{k}\,\dv_{k},\qquad\left[\delta_{i},\dv_{j}\right]=\N_{i\,\cdot j}^{k}\,\dv_{k},\qquad\left[\dv_{i},\dv_{j}\right]=0.
	\label{eq:commutation formulas}
\end{equation} 

\begin{rem} \label{rem:affine connections}
   Anisotropic connections $\an$ can actually be \textit{isotropic}, in the sense that $\an_{ij}^k(x,y)=\an_{ij}^k(x)$, while   nonlinear   connections $\N$ can actually be \textit{linear}, in the sense that $\N_{i}^k(x,y)=\an_{ia}^k(x)\,y^a$. In either case, the $\an_{ij}^k(x)$'s are some functions that necessarily define an affine connection (as a section of $\mathbf{C}M\longrightarrow M$, see \eqref{eq:cocycle gamma} and \eqref{eq:cocycle N}) and $\an$ or $\N$ is homogeneous. Hence, there is a natural identification between affine connections on $M$, isotropic $\an$'s and linear $\N$'s. Under this identification, each isotropic $\an$ gets identified with its underlying $\N$, which turns out to be linear, and then $\an=\dv\N$. This is consistent with \cite[Th. 2 (4)]{gelocor}. 
\end{rem}

\begin{rem}
	  Let $\nabla_{\partial_k}\nabla_{\partial_j}\partial_i-\nabla_{\partial_j}\nabla_{\partial_k}\partial_i=\mathrm{R}_{ijk}^l(x)\,\partial_l$ define the classical curvature of an affine connection $\an\vcentcolon M\longrightarrow\mathbf{C}M$ with the convention of \cite{oneill}. If, as above, one identifies this with a connection $\N$ of curvature $\RN$, then it is straightforward to prove that 
	\begin{equation}
	y^a\,\mathrm{R}_{ajk}^l(x)=\RN_{jk}^l(x,y), \qquad	y^a\,y^b\,\mathrm{R}_{abc}^c(x)=\ri(x,y),
	\label{eq:classical curvature}
	\end{equation}
	so the symmetric part of the classical Ricci tensor is 
	\[
	\frac{1}{2}\left(\mathrm{R}_{ijc}^c(x)+\mathrm{R}_{jic}^c(x)\right)=\frac{1}{2}\left(y^a\,y^b\,\mathrm{R}_{abc}^c(x)\right)_{\cdot i \cdot j}=\frac{1}{2}\,\ri_{\cdot i \cdot j}(x,y)
	\]
	and the scalar curvature constructed with any pseudo-Riemannian metric $g$ on $M$ is 
	\begin{equation}
	\mathrm{Scal}(x)=\frac{1}{2}\,g^{ab}(x)\,\left(\mathrm{R}_{abc}^c(x)+\mathrm{R}_{bac}^c(x)\right)=\frac{1}{2}\,g^{ab}(x)\,\ri_{\cdot a \cdot b}(x).
	\label{eq:classical scalar}
	\end{equation}
	Observe that we follow the same sign convention for $\RN$ as in \cite[\S II A]{pfeifer-wolfhart}, \cite[\S II B]{HPV} but our sign for $\ri$ is the standard one in Riemannian Geometry and thus opposite to that of the cited references.   
	
\end{rem}

\subsection{Sprays}
 In this subsection, we will present the sprays as sections of an affine bundle, unifying later this viewpoint with the more classical one discussed in \cite[\S 6.1]{gelocor}. 

 $\mathrm{T}A$ has natural coordinates $(x,y,z,w)$, where $(x,y)$ are the natural coordinates of any $v\in A$ and then we write $z^a\,\partial_a+w^a\,\dot\partial_a$ for the elements of $\mathrm{T}_vA$. The vertical distribution $\mathrm{V}A$ is described on them by $\left\{z^i=0\right\}$, which implies that it is a vector subbundle of $\mathrm{T}A\longrightarrow A$. Analogously, it follows that the set $\mathrm{S}A$ described by $\left\{z^i=y^i\right\}$ is an affine subbundle of $\mathrm{T}A\longrightarrow A$. In \cite[\S 2]{minguzzi}, this is referred to as the \textit{symmetrized bundle}. 


\begin{defn}
	A \textit{spray on A} is a section $\G$ of $\mathrm{S}A\longrightarrow A$, $2$-homogeneous as a vector field on $A$ (see Def. \ref{def:homogeneity} (ii) and Prop. \ref{prop:euler} (ii)).
\end{defn}

\begin{rem}
	(A) These are exactly the fields of the form 
	\[
	\G=y^a\,\partial_a-2\,\G^a\,\dot{\partial}_a
	\]
	for certain $2$-homogeneous coefficients $\G^k(x,y)$. This means that a spray is equivalent to a collection of $2$-homogeneous functions $\G^k$ on $A\cap\mathrm{T}U$ associated with each chart such that, under changes $(U,x)\rightsquigarrow(\bar{U},\bar{x})$,
	\begin{equation}
	\bar{\G}^{k}=\frac{1}{2}\,\frac{\partial\bar{x}^{k}}{\partial x^{c}}\,\frac{\partial^{2}x^{c}}{\partial\bar{x}^{a}\,\partial\bar{x}^{b}}\,\bar{y}^{a}\,\bar{y}^{b}+\frac{\partial\bar{x}^{k}}{\partial x^{c}}\,\G^{c}.
	\label{eq:cocycle G}
	\end{equation}
	
	(B) From the cocycle \eqref{eq:cocycle G}, the affine structure of the space of sprays is given as follows: for a fixed spray $\G_0$ and $Z:=Z^a\,\partial_a\in\mathrm{h}^2\mathcal{T}_{0}^{1}(\MA)$, $\G=\G_0-2\,Z$ has coordinates $\G_0^k+Z^k$.  The cause of this discrepancy is that we have decided to maintain the standard convention that $\G$ (and not $-2\,\G$) equals $y^a\,\partial_a-2\,\G^a\,\dot{\partial}_a$, whereas the anisotropic vector with coordinates $-2\,Z^i$ is $-2\,Z$ (and not $Z$). 
\end{rem}

\begin{defn} \label{def:underlying 2} \!
	\begin{enumerate}
		\item  By \cite[Prop. 3 (1)]{gelocor}, any homogeneous nonlinear connection $\N$ induces canonically a spray of coordinates $\G^i=\N_{a}^i\,y^a/2$. We call it the \textit{underlying spray of $\N$}. 
		\item  By \cite[Prop. 3 (2)]{gelocor}, any spray $\G$ induces canonically a symmetric homogeneous nonlinear connection of coordinates $\N_{i}^k=\G_{\cdot i}^k=\dot\partial_i\G^k$. We call it the \textit{vertical differential} or \textit{Berwald nonlinear connection of $\G$} and denote it by $\dv\G$. 
	\end{enumerate}
\end{defn}



The (projections to $M$ of the) integral curves of a spray $\G$ are its \textit{geodesics}. Its \textit{pregeodesics} are those curves in $M$ that can be (positively) reparametrized to be geodesics.


\begin{prop} \label{prop:sprays sharing pregeodesics} A spray $\G=\G_0-2\,Z$ shares pregeodesics with $\G_0$ if and only if $Z=\rho\,\canon$ for some $\rho\in\mathrm{h}^{1}\mathcal{F}(A)$.
\end{prop}
 For a  proof see \cite[Lem. 12.1.1]{spray and finsler}. 

\subsection{Pseudo-Finsler metrics} \label{sec:pseudo-finsler}
\begin{defn} \label{def:pseudo-finsler}
A \textit{(conic) pseudo-Finsler metric} defined on the open and conic  $A\subseteq\mathrm{T}M\setminus\mathbf{0}$ with $\pi(A)=M$ is an $L\in\mathrm{h}^{2}\mathcal{F}(A)$ whose \textit{fundamental tensor} $g=\dv^2L/2\in\mathrm{h}^0\mathcal{T}_2^0(\MA)$ is non-degenerate at every $v\in A$.
\end{defn}

\begin{rem}
	 Taking into account the nature of the variational problem that we will pose, we shall assume that our pseudo-Finsler metrics do not have \textit{lightlike} directions in the fixed $A$, namely $L(v)\neq 0$ for all $v\in A$. 
\end{rem}

We always denote $F:=\sqrt{\left|L\right|}\in\mathrm{h}^1\mathcal{F}(A)$; indices of tensors are lowered and raised with $g_{ij}$ and $g^{ij}$ respectively. By direct computation, one has the following identites: 
\[
L_{\cdot i}=2\,y_i(:=2\,g_{ia}\,y^a),\qquad y_{i\,\cdot j}=g_{ij},
\]
\[
F_{\cdot i}=\frac{\mathrm{sgn}(L)}{F}\,y_i,\qquad\left(\frac{y_{i}}{L}\right)_{\cdot j}=\frac{g_{ij}}{L}-2\,\frac{y_{i}}{L}\,\frac{y_{j}}{L}=\left(\frac{y_{j}}{L}\right)_{\cdot i}.
\]
From these and the $2$-homogeneity of $L$, it follows that
\[
L=\frac{1}{2}\,L_{\cdot a \cdot b}\,y^a\,y^b=g_{ab}\,y^a\,y^b=y_b\,y^b.
\]

\begin{defn} \label{defn:proper} (A) We say that a pseudo-Finsler metric $L$ defined on $A$ is \textit{proper} if  
	\begin{enumerate}
		\item Each fiber $A_p$ ($p\in M$) is connected with $L>0$ on $A$,
		\item $L$ extends smoothly to $\overline{A}\subseteq\mathrm{T}M\setminus\mathbf{0}$ with $L(v)=0$ and $g_v$ non-degenerate for $v\in\partial A:=\overline{A}\setminus A$.   
	\end{enumerate}
Then $g$ has a constant signature on $\overline{A}$.

(B) When that signature is Lorentzian $(+,-,...,-)$, $L$ is \textit{(properly) Lorentz-Finsler}. A \textit{Finsler spacetime} is any triple $(M,A,L)$ with $L$ Lorentz-Finsler.

(C) When the signature is positive definite, necessarily $A=\mathrm{T}M\setminus\mathbf{0}$ and $L$ is \textit{Finsler}.
\end{defn}

\begin{rem} \label{rem:proper}  Let us comment the parts of the last definition: 

(A) $g$ has constant signature on $\overline{A}$ because the connectedness of  $M$ together with (i) implies that $A$ is connected.  Moreover, the \textit{indicatrix} $\left\{L=1\right\}$ and (thanks to (ii)) the \textit{lightcone} $\partial A=\left\{L=0\right\}$ are smooth hypersurfaces: 
\[
\mathrm{d}L_v(u^\mathrm{V})=u^a\,L_{\cdot a}(v)=2\,u^a\,y_{a}(v)=2\,u^a\,g_{ab}(v)\,v^b=2\,g_v(u,v)
\]
for $u\in\mathrm{T}_{\pi(v)}M$, so $\mathrm{d}L_v$ never vanishes identically for $v\in\overline{A}\subseteq\mathrm{T}M\setminus\mathbf{0}$. 

(B) We want such an $L$ to be defined only on \textit{future causal vectors} (so $L\geq0$ together with $(+,-,...,-)$ as the Lorentzian signature is a choice of convention).  There is a Physics motivation for this assumption \cite[\S 1]{observers}, but it also has interesting mathematical implications.  For instance, $\overline{A_p}\subseteq\mathrm{T}_pM\setminus0$ is contained in an open half-space: there is a vector hyperplane $\varPi_p$ that does not intersect $\overline{A_p}$; thus, $A$ already determines a \textit{time orientation}. For this and other geometric consequences (such as convexity) for $\overline{A}$ of $L$ being Lorentz-Finsler, see \cite[Props. 2.6 and 3.4]{cones}. \footnote{Additionally, in \cite{minguzzi2} it is proven that one can actually extend $L$ to a pseudo-Finsler metric with Lorentzian fundamental tensor on the whole $\mathrm{T}M\setminus\mathbf{0}$ (in a highly non-unique way in contrast to the extension to $\overline{A}$).} 

(C)  The positive definiteness of $g$ together with (ii) implies that actually $\partial A=\emptyset$, so necessarily $A=\mathrm{T}M\setminus\mathbf{0}$. 
\end{rem}

A key geometric object associated with a pseudo-Finsler metric $L$ defined on $A$ is its \textit{metric spray} $\sprL$,  
\begin{equation}
	\left(\sprL\right)^{i}:=\frac{1}{4}\,g^{ia}\left(2\,\partial_{c}g_{ab}-\partial_a g_{bc}\right)y^{b}\,y^{c}.
	\label{eq:metric spray}
\end{equation}
 The Berwald $\LC:=\dv\sprL$ is the \textit{metric nonlinear connection}. From now on, given any anisotropic connection $\an$, it will be convenient to write $\covan$ instead of just $\nabla$ for its corresponding covariant derivative, $\covN$ in case that $\an=\dv\N$ for a nonlinear connection $\N$, and $\covL$ in case that $\an=\dv\N^L$ (this is the \textit{Berwald anisotropic connection of $L$} \cite[\S 4.3]{anisotropic}, \cite[Ch. 7]{spray and finsler}). Due to Defs. \ref{def:underlying 2} (ii) and \ref{def:underlying 1} (ii), the notions of \textit{$\an$-(pre)geodesics} and \textit{$\N$-(pre)geodesics} make sense, and due to \eqref{eq:metric spray}, so does that of \textit{$L$-(pre)geodesics}.  When using $\LC$, which is always symmetric, the  curvature and the Ricci scalar in \eqref{eq:curvature and torsion} will be denoted $\RL$ and  $\riL$ resp., as they can be associated with\footnote{For a Finsler $L$ ($g$ is positive definite), $\riL$ coincides on $\left\{L=1\right\}$ with the Ricci scalar defined as a sum of $n-1$ \textit{flag curvatures} as in \cite[(7.6.2a)]{BCS}.} $L$. 
 

The \textit{Cartan tensor} is 
\[
\C:=\frac{1}{2}\,\dv g\in\mathrm{h}^{-1}\mathcal{T}_{3}^{0}(\MA).
\]
It is symmetric, so it makes sense to define the \textit{mean Cartan tensor} as its metric trace, with components
\[
\C_{i}:=g^{ab}\,\C_{abi}.
\]
By vertically differentiating $g_{ia}\,g^{ak}=\delta_i^k$, one obtains the following identities: 
\[\C_{i}^{jk}=-\frac{1}{2}\,g_{\cdot i}^{jk},\qquad\C^{j}=-\frac{1}{2}\,g_{\cdot a}^{ja}.
\]
The \textit{Landsberg tensor} is 
\[
\la:=\frac{1}{2}\,\covL g\in\mathrm{h}^{0}\mathcal{T}_{3}^{0}(\MA)
\]
 (it can also be defined in terms of the \textit{Berwald tensor} \cite[(37)]{anisotropic}, however, $\la=\covL g/2$ is the way in which it will arise in this work).  Note that here it has the same sign as in \cite{anisotropic,mediterranean,pfeifer-wolfhart} and the opposite in \cite{spray and finsler,BCS,HPV}.  The Landsberg tensor is symmetric too, so it makes sense to define the \textit{mean Landsberg tensor}, with components  
\[
\la_{i}:=g^{ab}\,\la_{abi}. 
\]

\begin{rem}
 A pseudo-Finsler $L$ is equivalent to a symmetric and non-degenerate $g\in\mathrm{h}^0\mathcal{T}_2^0(\MA)$ with totally symmetric Cartan tensor \cite[Th. 3.4.2.1]{AIM}. This justifies being able to identify $L$ with $g$   whenever it is needed.   For instance, $L$ can be \textit{pseudo-Riemannian}, in the sense that $g$ is such kind of metric. This is equivalent to $g$ being isotropic and to $L$ being \textit{quadratic}, namely $L(x,y)=\varPsi_{ab}(x)\,y^a\,y^b/2$ for some isotropic and symmetric tensor $\varPsi/2$ that then necessarily equals $g$. 
\end{rem}

\section{Metric-affine variational calculus} \label{s.3}

For the remainder of the manuscript, $\N$ and $L$ are, respectively, a homogeneous nonlinear connection and a pseudo-Finsler metric defined on the open and conic $A$ with $L>0$ there. Our metric-affine formalism is akin to the metric formalism of \cite{HPV}. Its steps are: determination of a volume form on $A$, divergence formulas, choice of a Lagrangian function, induction (according to Prop. \ref{prop:inducing forms}) of forms on\footnote{  Integrating on this projectivization as in \cite{HPV}, instead of the indicatrix $\left\{L=1\right\}$, solves the technical issue of the integration domain depending on the variable $L$, present in \cite{pfeifer-wolfhart}.  } $\mathbb{P}^+A$ to construct an action there, and variation of this with respect to $\N$ and with respect to $L$.

Given $(\N,L)$, there is a natural way of constructing a $0$-homogeneous
volume form on $A$. The $\N$-horizontal and vertical isomorphisms allow us to define scalar products on $\mathrm{H}_{v}A$ and $\mathrm{V}_{v}A$:
\begin{equation} 
g_{v}^{\mathrm{H}}(X_{v}^{\mathrm{H}},Y_{v}^{\mathrm{H}}):=g_{v}(X_{v},Y_{v}),\qquad g_{v}^{\mathrm{V}}(X_{v}^{\mathrm{V}},Y_{v}^{\mathrm{V}}):=g_{v}(\frac{X_v}{F(v)},\frac{X_v}{F(v)})=\frac{g_{v}(X_v,Y_v)}{L(v)}
\label{eq:g^H and g^V}
\end{equation} 
for $X,Y\in\mathcal{T}_{0}^{1}(\MA)$. Each one has its own volume form:  
\[
\vol_{v}^{\mathrm{H}}:=\sqrt{\left|\det g_{v}^{\mathrm{H}}\left(\left.\delta_i\right|_{v},\left.\delta_j\right|_{v}\right)\right|}\,\mathrm{d}x^1_v\wedge...\wedge\mathrm{d}x^\di_v=:\sqrt{\left|\det g_{ij}(v)\right|}\,\mathrm{d}x_v,
\]
\[
\vol_{v}^{\mathrm{V}}:=\sqrt{\left|\det g_{v}^{\mathrm{V}}\left(\left.\dot{\partial}_i\right|_{v},\left.\dot{\partial}_j\right|_{v}\right)\right|}\,\delta y^1_v\wedge...\wedge\delta y^\di_v=:\frac{\sqrt{\left|\det g_{ij}(v)\right|}}{F(v)^\di}\,\delta y_v,
\]
where the $\mathrm{d}x^i_v$ and $\delta y^i_v:=\mathrm{d}y^i_v+\N^i_a(v)\,\mathrm{d}x^a_v$ are restricted to the horizontal and vertical subspaces respectively. A $2\di$-form is induced on $\mathrm{T}_vA=\mathrm{H}_vA\oplus\mathrm{V}_vA$: 
\begin{equation}
	\vol_v:=\vol_v^\mathrm{H}\wedge\vol_v^\mathrm{V}=\frac{\left|\det g_{ij}(v)\right|}{F(v)^\di}\,\mathrm{d}x_v\wedge\delta y_v.
	\label{eq:vol}
\end{equation}

\begin{rem} \label{rem:vol independent of N}
Even though we used $\N$ and $L$ to construct $\vol$, this turns out to depend on $L$ alone, as 
\[
\begin{split} \mathrm{d}x\wedge\delta y& =\mathrm{d}x^1\wedge...\wedge\mathrm{d}x^\di\wedge\left(\mathrm{d}y^1+\N^1_{a_1}\,\mathrm{d}x^{a_1}\right)\wedge...\wedge\left(\mathrm{d}y^\di+\N^\di_{a_\di}\,\mathrm{d}x^{a_\di}\right)   \\
& = \mathrm{d}x^1\wedge...\wedge\mathrm{d}x^\di\wedge\mathrm{d}y^1\wedge...\wedge\mathrm{d}y^\di \\
& =  \mathrm{d}x\wedge\mathrm{d}y.\\
\end{split}
\]
Taking the nature of our variational approach into account, it was of the most theoretical importance to define our volume form a priori in terms of both the connection and the metric.   On the other hand, by \eqref{eq:vol}, $\vol$ is the volume form of the \textit{Sasaki-type metric} $g_{v}^{\mathrm{H}}\overset{\perp}{\oplus}g_{v}^{\mathrm{V}}$, and by the previous observation, it also coincides with the volume form of the \textit{Sasaki metric of $g$} (that is, $g_{v}^{\mathrm{H}}\overset{\perp}{\oplus}g_{v}^{\mathrm{V}}$ for $\N=\LC$). Note that the definition of $g_{v}^{\mathrm{V}}$ dividing by $F$ as in \eqref{eq:g^H and g^V} is what guarantees the $0$-homogeneity of $\vol$.   
\end{rem}

This $\vol$ allows us to define the divergence of any vector field $\mathscr{X}$ on $A$ as 
\[
\dive(\mathscr{X})\,\vol:=\lie_\mathscr{X}(\vol)=\mathrm{d}(\mathscr{X}\lrcorner\vol).
\]
In the case of a $1$-homogeneous $\mathscr{X}$, by Prop. \ref{prop:inducing forms} (iii), one has the property that justifies discarding the divergence terms in the variational calculus:
\[
\underline{\dive(\mathscr{X})\,\vol}=-\mathrm{d}(\underline{\mathscr{X}\lrcorner\vol}).
\]
The following divergence formulas, generalizing \cite[(24) and (25)]{HPV}, are the key to the derivation of our equations. Their proof is in Appendix \ref{A1}. 

\begin{prop} \label{prop:divergence formulas}
	For $X\in\mathcal{T}_{0}^{1}(\MA)$, 
	
	\begin{equation}
	\dive(X^{\mathrm{H}})=X^{c}\left\{ \left(g^{ab}-\frac{\di}{2}\,\frac{1}{L}\,y^{a}\,y^{b}\right)\covN_c g_{ab}+\torN_{ca}^{a}\right\} +\covN_a X^{a},
	\label{eq:horizontal divergence}
	\end{equation}	
	
	\begin{equation}
	\dive(X^{\mathrm{V}})=\left(2\,\C_{a}-\di\,\frac{y_{a}}{L}\right)X^{a}+X_{\cdot a}^{a}.
	\label{eq: vertical divergence}
	\end{equation}
	
	If $X\in\mathrm{h}^{0}\mathcal{T}_{0}^{1}(\MA)$, then $\underline{\dive(X^\mathrm{H})\,\vol}=-\mathrm{d}(\underline{X^\mathrm{H}\lrcorner\vol})$ on $\mathbb{P}^+A$, and if $X\in\mathrm{h}^{1}\mathcal{T}_{0}^{1}(\MA)$, then  $\underline{\dive(X^\mathrm{V})\,\vol}=-\mathrm{d}(\underline{X^\mathrm{V}\lrcorner\vol})$.
\end{prop}

\begin{defn} \label{def:accion}
	Let $\dom\subseteq\mathbb{P}^+A$\footnote{$\vol$ defines a global orientation on $A$, the one making $(\partial_1,...,\partial_\di,\dot{\partial}_1,...,\dot{\partial}_\di)$ positive, regardless of the ones that we chose for $\vol^\mathrm{H}$, $\vol^\mathrm{V}$ and without requiring $M$ to be orientable. As  $\underline{\vol}$ is again a volume form  (see the comment at the end of Prop. \ref{prop:inducing forms} (ii)), an orientation on $\mathbb{P}^+A$ is inherited.} be a relatively compact subset. Along this article and relative to $D$, the  \textit{action functional} will be 
	\[
	\accion[\N,L]:=\int_{\dom}\underline{L^{-1}\,\ri\,\vol}
	\]
	and the \textit{alternative action functional} will be 
	\[
	\accion_{\star}[\N,L]:=\int_{\dom}\underline{g^{ab}\,\ri_{\cdot a \cdot b}\,\vol}.
	\]
\end{defn}

The relation between these two is due to \cite[Lem. 3]{HPV}. We state it in our notation. 

\begin{prop}
	For $f\in\mathrm{h}^{0}\mathcal{F}(A)$, one has
	\[
	\left\{g^{ab}\left(Lf\right)_{\cdot a \cdot b}-2\di f\right\}\vol=\dive(X^\mathrm{V})\,\vol,
	\]
	where $X^\mathrm{V}$ is the vertical field corresponding to $X:=L\,g^{ab}\,f_{\cdot b}\,\partial_a\in\mathrm{h}^1\mathcal{T}_{0}^{1}(\MA)$. As a consequence, the functionals that we are considering are equal up to a factor of $2\di$ and a boundary term: 
	\[
	\accion_{\star}[\N,L]-2\di\,\accion[\N,L]= -\int_{\partial\dom}\underline{X^\mathrm{V}\lrcorner\vol}. 
	\]
\end{prop}

\begin{proof}
	 As in the proof of \cite[Lem. 3]{HPV}, using the $0$-homogeneity of $f$, one directly computes 
	\[
	g^{ab}\left(Lf\right)_{\cdot a \cdot b}=2\di f+L\,g^{ab}\,f_{\cdot a \cdot b}.
	\]
	On the other hand, by \eqref{eq: vertical divergence}, 
	\[
	\begin{split}
	\dive(X^\mathrm{V})=&\left(2\,\C_{a}-\di\,\frac{y_{a}}{L}\right)L\,g^{ab}\,f_{\cdot b}+\left(L\,g^{ab}\,f_{\cdot b}\right)_{\cdot a} \\
	=&2L\,\C^b\,f_{\cdot b}+\left(2\,y_a\,g^{ab}\,f_{\cdot b}+L\,g^{ab}_{\cdot a}\,f_{\cdot b}+L\,g^{ab}\,f_{\cdot a \cdot b}\right) \\
	=&L\,g^{ab}\,f_{\cdot a \cdot b};
	\end{split}
	\]
	the $0$-homogeneity of $f$ was used twice and $g^{ab}_{\cdot a}=-2\,\C^b$ (\S \ref{sec:pseudo-finsler}) was used once. 
\end{proof}

We shall work with $\accion$, as it is of first order on $\N$ and second order on $L$ while $\accion_{\star}$ is of third order on $\N$. The advantage of the latter, on the other hand, is that it  is closer to  the Einstein-Hilbert-Palatini action, the functional of the classical metric-affine formalism \cite{palatini}   (compare with \cite[Prop. 6]{HPV}).  

\begin{prop} \label{prop:classical EHP action}
	Suppose that $\N$ is linear, $L$ is (positive definite) Riemannian and $\dom=\underset{p\in\dom_0}{\bigcup}\mathbb{P}^+(\mathrm{T}M\setminus\mathbf{0})_p$ for a relatively compact $\dom_0\subseteq M$. Then 
	\[
	\accion_{\star}[\N,L]= 2\,\mathrm{Vol}(\mathbb{S}^{\di-1})\int_{\dom_0}\mathrm{Scal}\,d\mathrm{V},
	\]
	where $\mathrm{Scal}$ is the scalar curvature constructed   with    $\N$ (regarded as an affine connection) and $g$,   $d\mathrm{V}$ is the $g$-volume element on $M$, and $\mathrm{Vol}(\mathbb{S}^{\di-1})$ is a universal constant. 
\end{prop} 

\begin{proof}
	 A standard argument   with    a partition of the unity on $\mathbb{P}^+(\mathrm{T}M\setminus\mathbf{0})$ induced by one on $M$ allows us to use Fubini's Theorem to obtain the following:   
	\[
	\begin{split}
	\accion_{\star}[\N,L]=&\int_{\mathbb{P}^+v\in\dom}\underline{g^{ab}\,\ri_{\cdot a \cdot b}\,\vol}_{\mathbb{P}^+v}\\
	=&\int_{\mathbb{P}^+v\in\dom}g^{ab}( \pi(v) )\,\ri_{\cdot a \cdot b}( \pi(v) )\,\underline{\vol}_{\mathbb{P}^+v}\\
	=&\int_{p\in\dom_0}g^{ab}(p)\,\ri_{\cdot a \cdot b}(p)\left(\int_{\mathbb{P}^+v\in\dom_p}\underline{\vol}_{\dom_p}\right)d\mathrm{V}_p \\
	=&\mathrm{Vol}(\mathbb{S}^{\di-1})\int_{p\in\dom_0}g^{ab}(p)\,\ri_{\cdot a \cdot b}(p)\,d\mathrm{V}_p \\
	=& 2\,\mathrm{Vol}(\mathbb{S}^{\di-1})\int_{p\in\dom_0}\mathrm{Scal}(p)\,d\mathrm{V}_p, 
	\end{split}
	\]
	where we used   \eqref{eq:classical scalar} and the fact    that each fiber $\dom_p=\mathbb{P}^+(\mathrm{T}M\setminus\mathbf{0})_p$ inherits a metric that makes it isometric to the round sphere $\mathbb{S}^{\di-1}$.   Indeed, $\mathbb{P}^+(\mathrm{T}M\setminus\mathbf{0})$ is naturally identified with the sphere bundle $\left\{L=1\right\}$, where the metric is induced by $g^\mathrm{H}\overset{\perp}{\oplus}g^{\mathrm{V}}$, the Sasaki metric of $g$. Moreover, the induced $\underline{\vol}_{\dom_p}$ is the volume form of the round metric on $\dom_p$ because $\vol$ is the volume form of $g^\mathrm{H}\overset{\perp}{\oplus}g^{\mathrm{V}}$ (see Rem. \ref{rem:vol independent of N}).     
\end{proof} 

In the non-definite case, it is not possible to integrate on a compact fiber with universal volume at each $p\in M$. Hence, one does not seem to be able to actually recover the Einstein-Hilbert-Palatini action in general. Nonetheless, the positive definiteness of $g$ and the compactness of the fibers are superfluous when it comes to our variational calculus, for all of it is local on $\mathbb{P}^+A$ and formally the same in every signature. Thus, Prop. \ref{prop:classical EHP action} indeed guarantees a priori the consistency of our equations with the (vacuum) EHP ones.

\begin{rem}
	Let us sum up the reasons for choosing $L^{-1}\,\ri$ as our metric-affine Lagrangian function. 
	\begin{enumerate}
		\item It is the first and most natural ($0$-homogeneous) curvature scalar that is derived from $\N$.
		\item The second most natural scalar, $g^{ab}\,\ri_{\cdot a \cdot b}$, turns out to be variationally equivalent to it.
		\item Moreover, $g^{ab}\,\ri_{\cdot a \cdot b}$ reduces to the EHP Lagrangian in the classical case.
		\item The metric Lagrangian of \cite{HPV,pfeifer-wolfhart} is $L^{-1}\,\riL$.
	\end{enumerate}
\end{rem}

\begin{defn} \label{def:variations} \! (A) A \textit{variation of $\N$} is a smooth one-parameter family of homogeneous nonlinear connections $\N(\tau)$ with $\N(0)=\N$. Its \textit{variational field} is
	\[
	\N^{\prime}=\left.\frac{\partial}{\partial\tau}\right|_{\tau=0}\N(\tau)\in\mathrm{h}^{1}\mathcal{T}_{1}^{1}(\MA)
	\]
(see \eqref{eq:cocycle N}). Analogously for a \textit{variation of $L$}, whose \textit{variational field} is
	\[
	L^{\prime}=\left.\frac{\partial}{\partial\tau}\right|_{\tau=0}L(\tau)\in\mathrm{h}^{2}\mathcal{F}(A).
	\]
	
(B)  Given a relatively compact subset $\dom\subseteq\mathbb{P}^{+}A$, we say that a variation $\N(\tau)$ is \textit{$\dom$-admissible} if the projectivized support of its variational field, $\mathbb{P}^{+}(\overline{\left\{v\in A:\,\N^{\prime}_v\neq 0\right\}}^A)$, is contained in $\dom$. In such a case, without loss of generality, we shall assume that $D$ is open with smooth boundary $\partial\dom\subseteq\mathbb{P}^{+}A$. We say that $\N(\tau)$ is \textit{admissible} if it is $\dom$-admissible for some $\dom$.  Analogously for $L(\tau)$. 
\end{defn}

In terms of the metric connection, we write 
\[
\N=\LC+\J,\qquad\J\in\mathrm{h}^{1}\mathcal{T}_{1}^{1}(\MA).
\]
The computations needed to derive our equations are in Appendices \ref{app:A} and \ref{app:B}. 

\begin{thm}[Metric-affine Finslerian Einstein equations] \! \label{thm:variational equations}
\begin{enumerate}
\item \textsc{(Affine equation)} The equality
\[
\left.\frac{\partial}{\partial\tau}\right|_{\tau=0}\accion[\N(\tau),L]=0
\]
is fulfilled for all admissible variations $\N(\tau)$ of $\N$ if and
only if the equality of homogeneous anisotropic tensors 
\begin{equation}
\begin{split}\left\{ 2\,\la_{b}+\left(\di+2\right)\,\frac{y_{a}}{L}\,\J_{b}^{a}-2\,\C_{a}\,\J_{b}^{a}-\left(\J_{b\,\cdot a}^{a}+\J_{a\,\cdot b}^{a}\right)\right\} \left(\delta_{i}^{b}\,y^{j}-y^{b}\,\delta_{i}^{j}\right)\\
-\left(\J_{i\,\cdot a}^{j}-\J_{a\,\cdot i}^{j}\right)y^{a} & =0
\end{split}
\label{eq:affine equation}
\end{equation}
is fulfilled on $A$. 
\item \textsc{(Metric equation)} The equality 
\[
\left.\frac{\partial}{\partial\tau}\right|_{\tau=0}\accion[\N,L(\tau)]=0
\]
is fulfilled for all admissible variations $L(\tau)$ of $L$ if and only
if the equality of homogeneous anisotropic scalars 
\begin{equation}
\left(\di+2\right)\ri-L\,g^{ab}\,\ri_{\cdot a\cdot b}=0\label{eq:metric equation}
\end{equation}
is fulfilled on $A$. 
\end{enumerate}
\end{thm}
\section{The affine equation} \label{sec:affine equation}

 Along this section, $L$ (and thus its associated $\LC$) is fixed. 

\begin{defn}\label{def41}
	 $\sol$ will be the \textit{space of solutions of the affine equation \eqref{eq:affine equation}}.
	That is, the set of those $\N$'s such that $\J:=\N-\LC \in\mathrm{h}^1\mathcal{T}_1^1(\MA) $ solves 
	\begin{equation}
	\left(2\,\la_{a}+2\,\B_{a}\right)\left(\delta_{i}^{a}\,y^{j}-y^{a}\,\delta_{i}^{j}\right)-\left(\J_{i\,\cdot a}^{j}-\J_{a\,\cdot i}^{j}\right)y^{a}=0\label{eq:affine equation 2}
	\end{equation}
	on $A$ (but not necessarily the metric equation \eqref{eq:metric equation}); here, 
	\begin{equation}
	\B_{i}:=\frac{\di+2}{2}\,\frac{y_{a}}{L}\,\J_{i}^{a}-\C_{a}\,\J_{i}^{a}-\frac{1}{2}\left(\J_{i\,\cdot a}^{a}+\J_{a\,\cdot i}^{a}\right),\qquad\B\in\mathrm{h}^{0}\mathcal{T}_{1}^{0}(\MA).
	\label{eq:B}
	\end{equation}
	$\solsim$ will be the \textit{space of symmetric solutions of the affine equation}. 
\end{defn}
 
\begin{rem} \label{rem:la_i = 0}
 When nonempty,  $\sol$ is an affine space directed by the space of solutions of 
\begin{equation}
2\,\mathcal{B}_{a}^{\J_{\ast}}\left(\delta_{i}^{a}\,y^{j}-y^{a}\,\delta_{i}^{j}\right)-\left\{ \left(\J_{\ast}\right)_{i\,\cdot a}^{j}-\left(\J_{\ast}\right)_{a\,\cdot i}^{j}\right\} y^{a}=0,\label{eq:affine equation 3}
\end{equation}
 while $\solsim$ is an affine subspace of $\sol$.
$\LC$ is in $\sol$ (and thus in $\solsim$) when $\J=0$ solves $\eqref{eq:affine equation 2}$, i.e., precisely when the mean Landsberg tensor vanishes ($\la_{i}=0$).   Notice that the vanishing of this tensor does not imply the vanishing of the whole $\la$, see \cite{LiShe07}.  
\end{rem}

\begin{rem} \label{rem:palatini}
	 Recall that the affine connections solving the classical metric-affine formalism (see \cite[(17)]{palatini} and references therein) are a Levi-Civita $\nabla^g$   (with Christoffel symbols $\left(\Gamma^g\right)_{ij}^k(x)$)   plus any tensor of the form $\A\otimes\mathrm{Id}$ with $\A$ an isotropic $1$-form. These affine connections can be regarded either as isotropic $\an$'s or linear $\N$'s (Rem. \ref{rem:affine connections}); from the latter viewpoint, they are of the form $\LC+\A\otimes\canon$. In other words, the isotropic connection $\left(  \an^g  \right)_{ij}^k(x)+ \A_i(x)\,\delta_j^k$ is identified with its underlying linear connection $\left(  \an^g  \right)_{ib}^k(x)\,y^b+ \A_i(x)\,y^k$). Thus, the map $\N\longmapsto\N+\A\otimes\canon$ is a translation on the space of solutions of the classical formalism whenever $\A$ is isotropic.   Here we shall prove the extension of this result to our formalism stating a previous lemma for further referencing.  
\end{rem}


\begin{lem} \label{lem:auxiliar}
 Let $\N=\LC+\J$ with $\J\in\mathrm{h}^{1}\mathcal{T}_{1}^{1}(\MA)$. Then:
\begin{enumerate}
	\item The torsion of $\N$ is given by 
	\begin{equation}
		\torN_{ij}^k=\J_{i\,\cdot j}^k-\J_{j\,\cdot i}^k.
		\label{eq:4 aux 7}
	\end{equation}
	\item The curvature of $\N$ is given in terms of that of $\LC$ by 
	\begin{equation}
		\RN_{ij}^k=\left(\RL\right)_{ij}^k+\left(\covL_j\J_{i}^k-\J_{i\,\cdot a}^k\,\J_j^a\right)-\left(\covL_i\J_{j}^k-\J_{j\,\cdot a}^k\,\J_i^a\right).
		\label{eq:4 aux 12}
	\end{equation}
	  \item The Ricci scalar of $\N$ is given in terms of that of $\LC$ by 
	\begin{equation}
	\ri=\riL-   y^{b}\,\covL_{b}\,\J_a^a+  \covL_{a}\left(\J^a_b\,y^b\right)+   y^b\,\J_{c\,\cdot a}^c\,\J_b^a-   y^b\,\J_{b\,\cdot a}^c\,\J_c^a.
	\label{eq:4 aux 15}
	\end{equation}  
	\item The $\N$-covariant derivative of $g$ is given by
	\begin{equation}
		\covN_k g_{ij}=2\,\la_{ijk}-2\,\C_{ija}\,\J_{k}^{a}-\J_{k\,\cdot i}^{a}\,g_{aj}-\J_{k\,\cdot j}^{a}\,g_{ia}.
		\label{eq:4 aux 13}
	\end{equation}
\end{enumerate} 
\end{lem}

\begin{proof} 
(i) This comes from the definition \eqref{eq:curvature and torsion} together with the symmetry of $\LC$. 

(ii) Using \eqref{eq:horizontal distribution}, 
\[
\begin{split}
\delta_j\N_i^k =&\left(\delta_j^L-\J_j^a\,\dv_a\right)\left\{\left(\LC\right)_{i}^k+\J_i^k\right\} \\
=&\delta_j^L\left(\LC\right)_{i}^k+\delta_j^L\J_{i}^k-\left(\LC\right)_{i\,\cdot a}^k\,\J_j^a-\J_{i\,\cdot a}^k\,\J_j^a,
\end{split}
\]
and completing $\delta_j^L\J_{i}^k$ to $\covL_j\J_{i}^k$ (see \eqref{eq:covariant derivative}), 
\[
\delta_j\N_i^k=\delta_j^L\left(\LC\right)_{i}^k+\covL_j\J_{i}^k+\left(\LC\right)_{j\,\cdot i}^a\,\J_a^k-\left(\LC\right)_{j\,\cdot a}^k\,\J_i^a-\left(\LC\right)_{i\,\cdot a}^k\,\J_j^a-\J_{i\,\cdot a}^k\,\J_j^a.
\]
Hence, again   by   the symmetry of $\LC$, \eqref{eq:curvature and torsion} yields \eqref{eq:4 aux 12}. 

(iii) This also comes from the definition \eqref{eq:curvature and torsion}, this time together with \eqref{eq:4 aux 12} and the fact that $\covL_iy^j=0$ (Prop. \ref{prop:covariante canon = 0}).  

(iv) Again using \eqref{eq:covariant derivative} and \eqref{eq:horizontal distribution},
\[
\begin{split} \covN_k g_{ij} & =\delta_{k}g_{ij}-\N_{k\,\cdot i}^{a}\,g_{aj}-\N_{k\,\cdot j}^{a}\,g_{ia}\\
& =\delta_{k}^{L}g_{ij}-\J_{k}^{a}\,\dv_{a}g_{ij}-\left(\LC\right)_{k\,\cdot i}^{a}\,g_{aj}-\J_{k\,\cdot i}^{a}\,g_{aj}-\left(\LC\right)_{k\,\cdot j}^{a}\,g_{ia}-\J_{k\,\cdot j}^{a}\,g_{ia},
\end{split}
\]
from where the definitions $\C=\dv g/2$ and $\la=\covL g/2$ yield \eqref{eq:4 aux 13}. 
\end{proof}

\begin{lem} \label{lem:translations}
	For any $\A\in\mathrm{h}^0\mathcal{T}_1^0(\MA)$, the map $\N\longmapsto\N+\A\otimes\canon$   preserves the Ricci scalar of all homogeneous nonlinear connections. As a consequence, such a map   is a translation on $\sol$, i.e., $\left(\J_\ast\right)_i^k:=\A_i\,y^k$ solves \eqref{eq:affine equation 3}. 
\end{lem}

\begin{proof}
	  For $\N=:\LC+\J$, the Ricci scalar of $\N_{\ast}:=\N+\A\otimes\canon$ can be computed with \eqref{eq:4 aux 15} by putting $\J_{\ast}:=\J+\A\otimes\canon$ in place of $\J$. Using $\covL_iy^j=0$ (Prop. ~\ref{prop:covariante canon = 0}), the $1$-homogeneity of $\J$ and the $0$-homogeneity of $\A$,
	\[
	y^{b}\,\covL_{b}\left(\J_{\ast}\right)_a^a=y^{b}\,\covL_{b}\,\J_a^a+y^{b}\,\covL_{b}\left(\A_a\,y^a\right),
	\]
	\[
	\covL_{a}\left(\J_{\ast}\right)^a_b\,y^b=\covL_{a}\left(\J^a_b\,y^b\right)+y^{a}\,\covL_{a}\left(\A_b\,y^b\right),
	\]
	\[
	\begin{split}
	y^b\left(\J_{\ast}\right)_{c\,\cdot a}^c\left(\J_{\ast}\right)_b^a=&y^b\left(\J_{c\,\cdot a}^c+\A_{c\,\cdot a}\,y^c+\A_{c}\,\delta^c_a\right)\left(\J_b^a+\A_b\,y^a\right)\\
	=&y^b\left(\J_{c\,\cdot a}^c\,\J_b^a+\A_{c\,\cdot a}\,y^c\,\J_b^a+\A_c\,\J^c_b+\J_c^c\,\A_b+\A_a\,y^a\,\A_b\right),
	\end{split}
	\]
	\[
	\begin{split}
	y^b\left(\J_{\ast}\right)_{b\,\cdot a}^c\left(\J_{\ast}\right)_c^a=&y^b\left(\J_{b\,\cdot a}^c+\A_{b\,\cdot a}\,y^c+\A_{b}\,\delta^c_a\right)\left(\J_c^a+\A_c\,y^a\right) \\
	=&y^b\left(\J_{b\,\cdot a}^c\,\J_c^a+\A_{b\,\cdot a}\,y^c\,\J_c^a+\A_b\,\J^c_c+\J_b^c\,\A_c+\A_a\,y^a\,\A_b\right).
	\end{split}
	\]
	Putting these together, 
	\[
	\begin{split}
	\ri_\ast=&\riL-   y^{b}\,\covL_{b}\,\J_a^a+  \covL_{a}\left(\J^a_b\,y^b\right)+   y^b\,\J_{c\,\cdot a}^c\,\J_b^a-   y^b\,\J_{b\,\cdot a}^c\,\J_c^a \\
	&-   y^{b}\,\covL_{b}\left(\A_a\,y^a\right)+   y^{a}\,\covL_{a}\left(\A_b\,y^b\right) \\
	&+   y^b\left(\A_{c\,\cdot a}\,y^c\,\J_b^a+\A_c\,\J^c_b+\J_c^c\,\A_b+\A_a\,y^a\,\A_b\right) \\
	&-   y^b\left(\A_{b\,\cdot a}\,y^c\,\J_c^a+\A_b\,\J^c_c+\J_b^c\,\A_c+\A_a\,y^a\,\A_b\right) \\
	=&\riL-   y^{b}\,\covL_{b}\,\J_a^a+\covL_{a}\left(\J^a_b\,y^b\right)+ y^b\,\J_{c\,\cdot a}^c\,\J_b^a-y^b\,\J_{b\,\cdot a}^c\,\J_c^a \\
	=&\ri.
	\end{split}
	\]

	  Having established that the translation by $\A\otimes\canon$ preserves the Ricci scalar, recall Th. \ref{thm:variational equations} (ii) and Def. \ref{def:accion}. Clearly, $\accion[\N+\A\otimes\canon,L]=\accion[\N,L]$ for any nonlinear connection $\N$, so, as it is standard in Variational Calculus, the translation maps critical points of the action to critical points. Indeed, if $\N\in\sol$, then every ($\dom$-admissible) variation of $\N+\A\otimes\canon$ is of the form $\N(\tau)+\A\otimes\canon$ for a ($\dom$-admissible) variation of $\N$, so 
	\[
	\left.\frac{\partial}{\partial\tau}\right|_{\tau=0}\accion[\N(\tau)+\A\otimes\canon,L]=\left.\frac{\partial}{\partial\tau}\right|_{\tau=0}\accion[\N(\tau),L]=0;
	\]
	by Th. \ref{thm:variational equations} (ii), $\N+\A\otimes\canon$ solves \eqref{eq:affine equation 2} too and so $\A\otimes\canon$ solves \eqref{eq:affine equation 3}.  
\end{proof}

\subsection{Reduction to the symmetric case}
Keep in mind that a homogeneous nonlinear connection is symmetric if and only if it is the vertical differential (also called \textit{Berwald nonlinear connection}) of a spray, see \cite[Prop. 3 (4)]{gelocor}. This is the case for $\LC$, so a homogeneous nonlinear connection is symmetric if and only if it is of the form $\LC+\dv\ZZ$ for some $\ZZ\in\mathrm{h}^{2}\mathcal{T}_{0}^{1}(\MA)$.  The next result provides the geometric invariants of the type of non-symmetric connections that will be relevant when reducing the affine equation to the symmetric case. 

\begin{prop} \label{prop:invariants solution} Suppose that $\N=\LC+\dv\ZZ+\A\otimes\canon$
	for some $\ZZ\in\mathrm{h}^{2}\mathcal{T}_{0}^{1}(\MA)$ and $\A\in\mathrm{h}^{0}\mathcal{T}_{1}^{0}(\MA)$. Then:
\begin{enumerate}
		\item  Its torsion, underlying spray 
		and covariant derivative of $g$ are given respectively by
		\begin{equation}
		\torN_{ij}^{k}=\left(\A_{i\,\cdot j}-\A_{j\,\cdot i}\right)y^{k}+\A_{i}\,\delta_{j}^{k}-\A_{j}\,\delta_{i}^{k},
		\label{eq:tor solution}
		\end{equation}
		\begin{equation}
		\G^{i}=\left(\sprL\right)^{i}+\ZZ^{i}+\frac{1}{2}\,\A_{a}\,y^{a}\,y^{i},
		\label{eq:spray solution}
		\end{equation}
		\begin{equation}
		\begin{split}\covN_k g_{ij} & =2\,\la_{ijk}-2\,\C_{ija}\,\ZZ^a_{\cdot k}-\left(\ZZ^a_{\cdot k \cdot i}\,g_{aj}+\ZZ^a_{\cdot k \cdot j}\,g_{ai}\right) \\
		& \quad-\left(\A_{k\,\cdot i}\,y_j+\A_{k\,\cdot j}\,y_i\right)-2\,g_{ij}\,\A_k.
		\label{eq:covariant g solution}
		\end{split}
		\end{equation} 
		
		\item  The torsion of $\N$ determines $\A$ as 
		\begin{equation}
		 2\left(\di-1\right)\A_i\,y^k=\left(\di-1\right)\torN_{ib}^k\,y^b-\left(\torN_{ab}^a\,y^b\right)_{\cdot i}y^k-\torN_{ab}^a\,y^b\,\delta_i^k. 
		\label{eq:torsion determines A}
		\end{equation} 
		
		\item  $\N$ shares pregeodesics with another $\N_{0}=\LC+\dv\ZZ_0+\A_0\otimes\canon$ if and only if $\ZZ=\ZZ_{0}+\varrho\,\canon$
		for some $\varrho\in\mathrm{h}^{1}\mathcal{F}(A)$. 
\end{enumerate}
\end{prop} 
\begin{proof}
 (i) Formula \eqref{eq:tor solution} is obtained by substituting $\J_i^k=\ZZ_{\cdot i}^k+\A_i\,y^k$ in Lem. \ref{lem:auxiliar} (i) and using that $\ZZ_{\cdot i\,\cdot j}^k=\ZZ_{\cdot j\,\cdot i}^k$. Formula \eqref{eq:spray solution} follows from Def. \ref{def:underlying 2} (i) and the $2$-homogeneity of $\ZZ$ (the underlying spray of $\LC$ is $\sprL$). 
Finally, formula \eqref{eq:covariant g solution} is obtained by substitution in Lem. \ref{lem:auxiliar} (iii) of the term
\[
\begin{split} &2\,\la_{ijk}-2\,\C_{ija}\,\J_{k}^{a}-\J_{k\,\cdot i}^{a}\,g_{aj}-\J_{k\,\cdot j}^{a}\,g_{ia} \\
=&2\,\la_{ijk}-2\,\C_{ija}\left(\ZZ_{\cdot k}^a+\A_k\,y^a\right) \\
&\quad-\left(\ZZ_{\cdot k \cdot i}^a+\A_{k\,\cdot i}\,y^a+\A_k\,\delta_i^a\right)g_{aj}-\left(\ZZ_{\cdot k \cdot j}^a+\A_{k\,\cdot j}\,y^a+\A_k\,\delta_j^a\right)g_{ia};
\end{split}
\]
using $\C_{ija}\,y^{a}=0$ yields the result. 

 (ii) From \eqref{eq:tor solution}, one computes 
\begin{equation}
\torN_{ib}^k\,y^b=-\A_{b\,\cdot i}\,y^{b}\,y^{k}+\A_{i}\,y^{k}-\A_{b}\,y^{b}\,\delta_{i}^{k}=-\left(\A_{b}\,y^{b}\right)_{\cdot i}\,y^{k}+2\,\A_{i}\,y^{k}-\A_{b}\,y^{b}\,\delta_{i}^{k},
\label{eq:4 aux 1}
\end{equation}
\[
\torN_{ab}^a\,y^b=-\left(\di-1\right)\A_{b}\,y^{b}
\]
(the $0$-homogeneity of $\A$ and the $1$-homogeneity of $\A_{b}\,y^{b}$ were used). Substituting $\A_{b}\,y^{b}$ back in \eqref{eq:4 aux 1}, multiplying everything by $\left(\di-1\right)$ and rearranging produces \eqref{eq:torsion determines A}. 

(iii) This follows from applying \ref{prop:sprays sharing pregeodesics} to sprays $\G$ and $\G_0$ of the form \eqref{eq:spray solution}. 
\end{proof}


\begin{thm} \label{thm:reduction symmetric}
$\N\in\sol$ if and only if it is of the form $\N=\LC+\dv\ZZ+\A\otimes\canon$ for some $\ZZ\in\mathrm{h}^{2}\mathcal{T}_{0}^{1}(\MA)$ such that $\LC+\dv\ZZ\in\solsim$ and $\A\in\mathrm{h}^{0}\mathcal{T}_{1}^{0}(\MA)$. In such a case, $(\ZZ,\A)$ is unequivocally determined by $\N$ as 
\begin{equation}
	\ZZ^j=\frac{1}{2}\,\J^j_a\,y^a-\B_a\,y^a\,y^j,\qquad\A_i=\la_i+\B_i+\left(\B_a\,y^a\right)_{\cdot i},
	\label{eq:Z and A}
\end{equation}
where $\J:=\N-\LC$ and $\B$ is defined by \eqref{eq:B}. 
\end{thm}
\begin{proof}
 We observe that, using the $1$-homogeneity of $\J$, the affine equation \eqref{eq:affine equation 2} can be rewritten as
\[
\J_{i}^{j}=\left(\la_a+\B_{a}\right)\left(\delta_{i}^{a}\,y^{j}-y^{a}\,\delta_{i}^{j}\right)+\frac{1}{2}\left(\J_{a}^{j}\,y^{a}\right)_{\cdot i}
\]
and that this allows one to derive the form of the general solution. Indeed, using that $\la_a\,y^a=0$, 
\[
\begin{split}\J_{i}^{j} &=\la_{i}\,y^{j}+\B_{i}\,y^{j}-\B_a\,y^a\,\delta_i^j+\frac{1}{2}\left(\J_{a}^{j}\,y^{a}\right)_{\cdot i} \\ 
& =\la_{i}\,y^{j}+\B_{i}\,y^{j}-\left(\B_{a}\,y^{a}\,y^{j}\right)_{\cdot i}+\left(\B_{a}\,y^{a}\right)_{\cdot i}y^{j}+\frac{1}{2}\left(\J_{a}^{j}\,y^{a}\right)_{\cdot i} \\
& =\left(\frac{1}{2}\,\J_{a}^{j}\,y^{a}-\B_{a}\,y^{a}\,y^{j}\right)_{\cdot i}+\left\{\la_i+\B_{i}+\left(\B_{a}\,y^{a}\right)_{\cdot i}\right\} y^{j},
\end{split}
\] 
which tells us that $\J=(\N-\LC)=\dv\ZZ+\A\otimes\canon$ together with \eqref{eq:Z and A}. Lemma \ref{lem:translations} ensures that $\N=\LC+\dv\ZZ+\A\otimes\canon$ is in $\sol$ if and only if the symmetric part $\LC+\dv\ZZ$ is. 

 We derive the uniqueness of the pair $(\ZZ,\A)$ from Prop. \ref{prop:invariants solution} (ii): the torsion of $\N=\LC+\dv\ZZ+\A\otimes\canon$ determines $\A$, which in turn determines $\dv\ZZ$, and from here $\ZZ$ is determined due to its $2$-homogeneity. 
\end{proof}

Now we characterize the elements of $\solsim$.

\begin{prop} \label{prop:torsion-free affine equations 0}  $\LC+\dv\ZZ\in\solsim$ if and only if $\ZZ$ solves
	\begin{equation}
	\la_{i}+\frac{\di+2}{2}\,\frac{y_a}{L}\,\ZZ_{\cdot i}^{a}-\C_a\,\ZZ_{\cdot i}^{a}-\left\{\left(\di+2\right)\frac{y_a}{L}\,\ZZ^{a}-2\,\C_a\,\ZZ^{a}\right\}_{\cdot i}=0,
	\label{eq:affine equation 7}
	\end{equation}
	\begin{equation}
	\left(\di+2\right)\frac{y_a}{L}\,\ZZ^{a}-2\,\C_a\,\ZZ^{a}-\ZZ_{\cdot a}^{a}=0.
	\label{eq:affine equation 8}
	\end{equation} 
\end{prop} 

\begin{proof}
	 We restrict the affine equation \eqref{eq:affine equation 2} to symmetric connections (see Lem. \ref{lem:auxiliar} (i)). As for these connections $\J_{i\,\cdot k}^{j}-\J_{k\,\cdot i}^{j}=\torN_{ik}^{j}=0$, using also $\la_a\,y^a=0$, the equation reads
	\begin{equation}
	0=\left(\la_a+\B_{a}\right)\left(\delta_{i}^{a}\,y^{j}-y^{a}\,\delta_{i}^{j}\right)=\left(\la_i+\B_{i}\right)y^{j}-\B_{a}\,y^{a}\,\delta_{i}^{j}.
	\label{eq:4 aux 5}
	\end{equation}
	This is trivially implied by $\la_i+\B_{i}=0$, but the converse is also true, for taking the trace of \eqref{eq:4 aux 5} yields $-\left(\di-1\right)\B_{a}\,y^{a}=0$. Thus, recalling \eqref{eq:B} and writing $\J_i^k=\ZZ_{\cdot i}^k$, $\ZZ^a_{\cdot i\,\cdot a}+\ZZ^a_{\cdot a\,\cdot i}=2\,\ZZ^a_{\cdot a\,\cdot i}$, the equation describing $\solsim$ is 
	\begin{equation}
	\la_{i}+\frac{\di+2}{2}\,\frac{y_{a}}{L}\,\ZZ_{\cdot i}^a-\C_{a}\,\ZZ_{\cdot i}^a-\ZZ_{\cdot a\cdot i}^{a}\left(=\la_i+\B_{i}\right)=0.\label{eq:4 aux 6}
	\end{equation}
	Clearly, \eqref{eq:affine equation 7}+\eqref{eq:affine equation 8} are sufficient for this. However, they are also necessary: \eqref{eq:affine equation 8} is obtained by contracting \eqref{eq:4 aux 6} with $y^i$ and using $\la_a\,y^a=0$, the $2$-homogeneity of $\ZZ$, and the $1$-homogeneity of $\ZZ_{\cdot a}^{a}$. 
\end{proof}

 In Prop. \ref{prop:torsion-free affine equations 0}, we have obtained two \textit{torsion-free affine equations} with somewhat complicated expressions. Next, we are going to formulate them in a way that it is much more convenient for our main results (those of \S \ref{sec:proper solutions}). 

\begin{defn}  \label{def:sigma and K}
	For $\ZZ\in\mathrm{h}^{2}\mathcal{T}_{0}^{1}(\MA)$, we denote 
	\begin{equation}
	\sigma^{\ZZ}:=\frac{y_{a}}{L}\,\ZZ^{a}=\frac{g(\ZZ,\canon)}{L}\in\mathrm{h}^{1}\mathcal{F}(A)
	\label{eq:sigma}
	\end{equation}
	and 
	\begin{equation}
	\mathcal{K}_{i}^{\ZZ}:=-\frac{2}{\di+2}\left(2\,\C_{a\,\cdot i}\,\ZZ^{a}+\C_{a}\,\ZZ_{\cdot i}^{a}\right),\qquad\mathcal{K}^{\ZZ}\in\mathrm{h}^{0}\mathcal{T}_{1}^{0}(\MA).
	\label{eq:K}
	\end{equation} 
\end{defn}

\begin{rem} \label{rem:property K}
	 Thanks to the $(-1)$-homogeneity of the mean Cartan tensor and the $2$-homogeneity of $\ZZ$, one has the important property 
	\[
	\mathcal{K}^\ZZ_a\,y^a=0,
	\]
	exactly the same as for the mean Landsberg tensor. 
\end{rem}

\begin{lem} \label{prop: torsion-free affine equations}  $\LC+\dv\ZZ\in\solsim$ if and only if $\ZZ$ solves 
\begin{equation}
\ZZ^{i}=2\sigma^{\ZZ}\,y^{i}-L\,g^{ia}\left(\sigma_{\cdot a}^{\ZZ}+\mathcal{K}_{a}^{\ZZ}\right)+\frac{2}{\di+2}L\,\la^{i},\label{eq:affine equation 4}
\end{equation}
\begin{equation}
\left(\di+2\right)\sigma^{\ZZ}-2\,\C_{a}\,\ZZ^{a}-\ZZ_{\cdot a}^{a}=0.\label{eq:affine equation 5}
\end{equation}
Moreover, when assuming the form \eqref{eq:affine equation 4} for $\ZZ$, \eqref{eq:affine equation 5} reads  
\begin{equation}
	\left(\di-2\right)\sigma^{\ZZ}-L\,g^{ab}\left(\sigma^{\ZZ}_{\cdot a\cdot b}+\mathcal{K}^\ZZ_{a\,\cdot b}-\frac{2}{\di+2}\,\la_{a\,\cdot b}\right)=0.
	\label{eq:affine equation 6}
\end{equation} 
\end{lem}
\begin{proof}
	 In the notation introduced in Def. \ref{def:sigma and K}, \eqref{eq:affine equation 8} becomes \eqref{eq:affine equation 5}. For the reexpression of \eqref{eq:affine equation 7} as \eqref{eq:affine equation 4}, recall from \S \ref{sec:pseudo-finsler} that
	\[
	\left(\frac{y_{j}}{L}\right)_{\cdot i}=\frac{g_{ij}}{L}-2\,\frac{y_{i}}{L}\,\frac{y_{j}}{L}.
	\]
	By completing $L^{-1}\,y_a\,\ZZ^a_{\cdot i}$ to a derivative of $\sigma^{\ZZ}=L^{-1}\,y_a\,\ZZ^a$ and simplifying, the left hand side of \eqref{eq:affine equation 7} becomes
	\[
	\begin{split}
	&\la_{i}+\frac{\di+2}{2}\,\frac{y_a}{L}\,\ZZ_{\cdot i}^{a}-\C_a\,\ZZ_{\cdot i}^{a}-\left\{\left(\di+2\right)\frac{y_a}{L}\,\ZZ^{a}-2\,\C_a\,\ZZ^{a}\right\}_{\cdot i} \\
	=&\la_{i}+\frac{\di+2}{2}\,\sigma^{\ZZ}_{\cdot i}-\frac{\di+2}{2}\left(\frac{y_{a}}{L}\right)_{\cdot i}\ZZ^{a}-\C_{a}\,\ZZ_{\cdot i}^{a}-\left(\di+2\right)\sigma^\ZZ_{\cdot i}+2\left(\C_{a}\,\ZZ^{a}\right)_{\cdot i}\\
	=&-\frac{\di+2}{2}\,\frac{g_{ia}}{L}\,\ZZ^{a}+\left(\di+2\right)\frac{y_{a}}{L}\,\ZZ^{a}\,\frac{y_{i}}{L}-\frac{\di+2}{2}\,\sigma^{\ZZ}_{\cdot i} \\
	&+2\,\C_{a\,\cdot i}\,\ZZ^{a}+\C_{a}\,\ZZ_{\cdot i}^{a}+\la_{i}, \\
	=&-\frac{\di+2}{2}\,\frac{g_{ia}}{L}\,\ZZ^{a}+\left(\di+2\right)\sigma^{\ZZ}\,\frac{y_{i}}{L}-\frac{\di+2}{2}\,\sigma^{\ZZ}_{\cdot i} \\
	&-\frac{\di+2}{2}\,\mathcal{K}^{\ZZ}_i+\la_{i}. 
	\end{split}
	\]
	Thus, after multiplying by $2\left(\di+2\right)^{-1}L$ and raising the index, \eqref{eq:affine equation 7} becomes \eqref{eq:affine equation 4}. 
	
	 Let us reexpress \eqref{eq:affine equation 5} as \eqref{eq:affine equation 6}. For $\ZZ$ of the form 
	\[
	\ZZ_{i}=2\sigma^{\ZZ}\,y_{i}-L\left(\sigma_{\cdot i}^{\ZZ}+\mathcal{K}_{i}^{\ZZ}\right)+\frac{2}{\di+2}L\,\la_{i},
	\]
	using $y_{i\,\cdot j}=g_{ij}$ and $L_{\cdot j}=2\,y_j$, one has 
	\[
	\begin{split}
	\quad\ZZ_{i\,\cdot j}=&2\,y_{i}\,\sigma^{\ZZ}_{\cdot j}+2\sigma^{\ZZ}\,g_{ij}-2\left(\sigma_{\cdot i}^{\ZZ}+\mathcal{K}_{i}^{\ZZ}\right)y_j-L\left(\sigma_{\cdot i\cdot j}^{\ZZ}+\mathcal{K}_{i\,\cdot j}^{\ZZ}\right)\\
	&+\frac{4}{\di+2}\,\la_{i}\,y_j+\frac{2}{\di+2}L\,\la_{i\,\cdot j} 
	\end{split}
	\] 
	Using now the $1$-homogeneity of $\sigma^{\ZZ}$, $\mathcal{K}^\ZZ_a\,y^a=0$ (see Rem. \ref{rem:property K}) and $\la_a\,y^a=0$, 
	\[
	\begin{split}
		g^{ab}\,\ZZ_{a\,\cdot b}&=2\sigma^{\ZZ}+2\di\,\sigma^{\ZZ}-2\sigma^{\ZZ}-L\,g^{ab}\left(\sigma^{\ZZ}_{\cdot a\cdot b}+\mathcal{K}^\ZZ_{a\,\cdot b}\right)+\frac{2}{\di+2}\,L\,g^{ab}\,\la_{a\,\cdot b}\\
		&=2\di\,\sigma^{\ZZ}-L\,g^{ab}\left(\sigma^{\ZZ}_{\cdot a\cdot b}+\mathcal{K}^\ZZ_{a\,\cdot b}-\frac{2}{\di+2}\,\la_{a\,\cdot b}\right).
	\end{split}
	\]
	On the other hand, it is also true that 
	\[
	g^{ab}\,\ZZ_{a\,\cdot b}=g^{ab}\left(g_{ac}\,\ZZ^c\right)_{\cdot b}=g^{ab}\left(2\,\C_{abc}\,\ZZ^c+g_{ac}\,\ZZ^c_{\cdot b}\right)=2\,\C_{a}\,\ZZ^{a}+\ZZ_{\cdot a}^{a}.
	\]
	Taking into account the last two formulas, the left hand side of \eqref{eq:affine equation 5} becomes 
	\[
	\begin{split}
	&\left(\di+2\right)\sigma^{\ZZ}-2\,\C_{a}\,\ZZ^{a}-\ZZ_{\cdot a}^{a} \\
	=&\left(\di+2\right)\sigma^{\ZZ}-\left\{2\di\,\sigma^{\ZZ}-L\,g^{ab}\left(\sigma^{\ZZ}_{\cdot a\cdot b}+\mathcal{K}^\ZZ_{a\,\cdot b}-\frac{2}{\di+2}\,\la_{a\,\cdot b}\right)\right\}.
	\end{split}
	\]
	Thus, after simplifying and rearranging, \eqref{eq:affine equation 5} becomes \eqref{eq:affine equation 6}. 
\end{proof}

\subsection{Pregeodesics and Ricci scalar of solutions}
\begin{cor} \label{cor:projection}
There is a well-defined projection 
\[
\begin{split}\prsim\vcentcolon\sol & \longrightarrow\solsim,\\
\N=\LC+\dv\ZZ+\A\otimes\canon & \longmapsto\LC+\dv\ZZ,
\end{split}
\]
 with the following properties:  
\begin{enumerate}
\item For $\LC+\dv\ZZ\in\solsim$, the only symmetric representative of the fiber
$\left(\prsim\right)^{-1}(\LC+\dv\ZZ)$ is $\LC+\dv\ZZ$ itself.
\item Two elements $\N,\N_{0}\in\sol$ share pregeodesics if and only if
they are on the same fiber. 
\item The pregeodesics of $\N\in\sol$ are those of $L$ only in case that $\prsim(\N)=\LC$.
\item All the representatives of a fiber share Ricci scalar. 
\end{enumerate}
\end{cor}
\begin{proof}
$\prsim$ is well-defined due to Th. \ref{thm:reduction symmetric}. \footnote{ It could be defined on any connection of the form $\N=\LC+\dv\ZZ+\A\otimes\canon$ with $\ZZ\in\mathrm{h}^{2}\mathcal{T}_{0}^{1}(\MA)$ and $\A\in\mathrm{h}^{0}\mathcal{T}_{1}^{0}(\MA)$, for the argument that we used to prove the uniqueness of $(\ZZ,\A)$ is independent of $\N$ being in $\sol$ (see the proof of the mentioned theorem). }

(i) By Prop. \ref{prop:invariants solution} (ii), if $\N=\LC+\dv\ZZ+\A\otimes\canon$ is symmetric, then $\A=0$. 

(ii) $\N=\LC+\dv\ZZ+\A\otimes\canon$ and $\N_{0}=\LC+\dv\ZZ_0+\A_0\otimes\canon$ being on the same fiber of $\prsim$ means that $\ZZ=\ZZ_0$, from where Prop. \ref{prop:invariants solution} (iii) tells us that they share pregeodesics. Conversely, if this happens, then $\ZZ=\ZZ_{0}+\varrho\,\canon$ with $\LC+\dv\ZZ,\LC+\dv\ZZ_0\in\solsim$ and $\varrho\in\mathrm{h}^1\mathcal{F}(A)$.  By Lem. \ref{prop: torsion-free affine equations}, both $\ZZ$ and $\ZZ_0$ solve \eqref{eq:affine equation 5}, so 
\[
\begin{split}
0 &=\left(\di+2\right)\sigma^{\ZZ}-2\,\C_{a}\,\ZZ^{a}-\ZZ_{\cdot a}^{a} \\
&=\left(\di+2\right)\sigma^{\ZZ_0}+\left(\di+2\right)\varrho-2\,\C_{a}\left(\ZZ_0\right)^{a}-\left(\ZZ_0\right)_{\cdot a}^{a}-\left(\varrho_{\cdot a}\,y^a+\varrho\,\delta_{ a}^{a}\right) \\
&=\left(\di+2\right)\sigma^{\ZZ_0}-2\,\C_{a}\left(\ZZ_0\right)^{a}-\left(\ZZ_0\right)_{\cdot a}^{a}+\varrho \\
&=\varrho 
\end{split}
\]
(the definition \eqref{eq:sigma} of $\sigma^{\ZZ}$, $\C_{a}\,y^a=0$ and the $1$-homogeneity of $\varrho$ were used). Thus, $\ZZ=\ZZ_0$, which means that $\N$ and $\N_{0}$ are on the same fiber. 

(iii) Suppose that $\N=\LC+\dv\ZZ+\A\otimes\canon$ shares pregeodesics with $\LC$. This time, Prop. \ref{prop:invariants solution} (iii) gives us $\ZZ=\varrho\,\canon$ and analogous computations to the previous item yield $\varrho=0$. From here, $\prsim(\N)=\prsim(\LC+\A\otimes\canon)=\LC$. 

(iv) This is due to   Lem. \ref{lem:translations}.  
\end{proof}
\begin{rem}
  Despite the notation,   this projection $\prsim$ is not the same as the canonical one of (always homogeneous) nonlinear connections onto symmetric nonlinear connections; the latter is $\N=\dv\G+J\longmapsto\dv\G$ with $\G$ the underlying spray of $\N$. While $\N$ and $\dv\G$ actually share geodesics, they do not necessarily share Ricci scalar. 
\end{rem}

Let us focus briefly on those $\N\in\sol$ with $\prsim(\N)=\LC$ (i.e., $\dv\ZZ=0$ and, by homogeneity, $\ZZ=0$). 

\begin{defn} \label{formally classical}
 We refer to the elements of 
\[
\left(\prsim\right)^{-1}(\LC)=\begin{cases}
\left\{\LC+\A\otimes\canon\colon\;\A\in\mathrm{h}^{0}\mathcal{T}_{1}^{0}(\MA)\right\} & \text{if}\quad\la_i=0,\\
\emptyset & \text{otherwise},\\
\end{cases}
\]
as \textit{formally
classical solutions of the affine equation \eqref{eq:affine equation}}. Consistently, in case that $L$ is pseudo-Riemannian, we refer to those elements of $\left(\prsim\right)^{-1}(\LC)$ with $\A$ isotropic as \textit{classical solutions}. 
\end{defn}

\begin{rem}\label{REM_4.15}  $\left(\prsim\right)^{-1}(\LC)$ being nonempty is equivalent to  $\LC$ being in $\solsim$ and, in turn, to $\la_i=0$ (see Rem. \ref{rem:la_i = 0}), which in particular happens in case that $L$ is pseudo-Riemannian. When $\left(\prsim\right)^{-1}(\LC)\neq\emptyset$, its elements have the form of the (underlying linear connections of the) solutions of the classical Palatini formalism (see Rem. \ref{rem:palatini}). The difference is that our formalism allows for a non--pseudo-Riemannian $L$ and an anisotropic $\A$, hence the distintion between \textit{formally classical} and \textit{classical} solutions. 
\end{rem}

In Cor. \ref{cor:projection}, we have seen that the formally classical solutions are exactly those that share pregeodesics with $L$. Their Ricci scalar is   the metric one $\riL$   and, when they do exist, the only symmetric one among them is $\LC$ itself. Their importance can be recognized also from the Physics viewpoint. If one wants to model the free fall of particles in a Finsler spacetime equipped with $\N$, in principle they could choose between two different postulates: either particles follow $\N$-geodesics or they follow $L$-geodesics. When $\N$ is formally classical, at least the trajectories and measured proper times coincide for both options.

For these reasons, in the case $\la_i=0$ it is natural to ask whether actually all solutions are formally classical. In general, one can ask if there is only one fiber (equiv., only one symmetric solution). This is studied in \S \ref{sec:proper solutions}, where a positive answer is provided in many cases of interest.

\subsection{Metric compatibility conditions}

When $g$ and $\an$ are isotropic, the compatibility of the connection with the metric just means $\covan_k g_{ij}=0$. When one further restricts to solutions of the classical metric-affine formalism, either one of the conditions of vanishing torsion or $\covan_k g_{ij}=0$ suffices to select the Levi-Civita connection; moreover, $g^{ab}\,\covan_k g_{ab}=0$ also suffices \cite[(18)]{palatini}. 

In the general Finslerian setting, vanishing torsion together with $\covan_k g_{ij}=0$ determines $\an$ as the Levi-Civita--Chern anisotropic connection of $g$ \cite{gelocor,anisotropic,chern connection,spray and finsler}. Nevertheless, there are at least seven nonequivalent concepts of metric compatibility that one could think of. Each one is given by the vanishing of one of the following tensors, where we assume that $\N$ is the underlying nonlinear connection of $\an$:
\[
\covan_k g_{ij},\qquad\covN_k g_{ij},\qquad\covan_k y_{j}=\covan_k g_{aj}\,y^{a},\qquad\covN_k y_{j}=\covN_k g_{aj}\,y^{a},
\]
\[
\covan_{k}g_{ab}\,y^{a}\,y^{b}=\covan_k L=\covN_k L=\covN_{k}g_{ab}\,y^{a}\,y^{b},\qquad y^{c}\,\covan_{c}g_{ij},\qquad y^{c}\,\covN_{c}g_{ij};
\]
keep in mind that always $\nabla_k y^j=0$ (Prop. \ref{prop:covariante canon = 0}), but $\nabla_k y_{j}:=\nabla_k\left(g_{ja}\,y^a\right)\neq g_{ja}\,\nabla_k y^a$. When restricting to solutions of our affine equation, some metric compatibility conditions select a single element of each fiber $\left(\prsim\right)^{-1}(\LC+\dv\ZZ)$, much like $\torN_{ij}^k=0$ selects $\LC+\dv\ZZ$. This, in turn, has important consequences. 

 Until the end of this section, we use that $\N$ is of the form $\LC+\dv\ZZ+\A\otimes\canon$  for some $\ZZ\in\mathrm{h}^{2}\mathcal{T}_{0}^{1}(\MA)$ and $\A\in\mathrm{h}^{0}\mathcal{T}_{1}^{0}(\MA)$, which in particular holds true whenever $\N\in\sol$. 

\begin{lem}  For $\N=\LC+\dv\ZZ+\A\otimes\canon$, one has 
\begin{equation}
\begin{split}
\covN_i y_{k}(=\covN_i g_{bk} \,y^b)=-\left(\ZZ^a_{\cdot i}\,g_{ak}+y_a\,\ZZ^a_{\cdot i \cdot k}\right)-L\,\A_{i\,\cdot k}-2\,\A_i\,y_k,
\end{split}
\label{eq:4 aux 9}
\end{equation}
\begin{equation}
\covN_i  L(=\covN_{i}y_c\,y^{c})=-2\,y_a\,\ZZ^a_{\cdot i}-2L\,\A_i,
\label{eq:4 aux 10}
\end{equation}	
\begin{equation}
\left(\covN_i  L\right)_{\cdot k}=2\,\covN_i y_{k}.
\label{eq:4 aux 14}
\end{equation} 
\end{lem} 

\begin{proof}
	 In Prop. \ref{prop:invariants solution} we showed formula \eqref{eq:covariant g solution}, from where \eqref{eq:4 aux 9} follows by contracting with $y^j$ and using $\la_{ibk}\,y^b=0$, $\C_{bki}\,y^b=0$, the $1$-homogeneity of $\ZZ^j_{\cdot i}$, and the $0$-homogeneity of $\A$. Formula \eqref{eq:4 aux 10} follows from \eqref{eq:4 aux 9} by doing the same. Finally, from comparing the vertical differential of \eqref{eq:4 aux 10} with \eqref{eq:4 aux 9}, and using $y_{j\,\cdot k}=g_{jk}$ and $L_{\cdot k}=2\,y_k$, formula \eqref{eq:4 aux 14} follows. 
\end{proof}

\begin{prop} \label{prop:metric compatibility 1} The following are equivalent: 
	\begin{enumerate}
		\item $\covN_i  L=0$;
		\item  $\N$ is the underlying nonlinear connection of some anisotropic connection $\an$ for which
		$\covan_{k}g_{ij}=0$. In this case, one can choose $\an_{ij}^{k}=\N_{i\,\cdot j}^{k}+Q_{ij}^{k}$ with $Q_{ij}^{k}:=g^{ka}\,\covN_{i}g_{ja}/2$; 
		\item $\covN_i y_{k}=0$;
		\item $\A_{i}=-y_{a}\,\ZZ_{\cdot i}^{a}/L$. 
	\end{enumerate}
\end{prop}

\begin{proof} 
 (i)$\Longrightarrow$(iii) By \eqref{eq:4 aux 14}, $2\,\covN_i y_{k}=\left(\covN_i  L\right)_{\cdot k}=0$. 
	
 (iii)$\Longrightarrow$(ii) The condition $\covN_i y_{k}=0$ implies that the chosen $Q$ above fulfills $Q_{ib}^{k}\,y^b=0$, so the underlying nonlinear connection of $\an=\dv\N+Q$ is $\N$. Then, $\covan_{k}g_{ij}=0$ is obtained just by substituting our choice in the general expression 
\[
\covan_k g_{ij}=\delta_k g_{ij}-\an_{ki}^a\,g_{aj}-\an_{kj}^a\,g_{ia}=\covN_k g_{ij}-Q_{ki}^a\,g_{aj}-Q_{kj}^a\,g_{ia}.
\]
(see \eqref{eq:covariant derivative}). 

 (ii)$\Longrightarrow$(i) Note that for any $\an$, such as the one above, the covariant derivative of a function only depends on the underlying nonlinear connection $\N$. Together with $L=g_{ab}\,y^a\,y^b$ and $\nabla_i y^j=0$, this provides $\covN_i  L=\covan_i  L=\covan_{i}g_{bc}\,y^{b}\,y^{c}=0$. \footnote{Notice, thus, that (iii)$\Longrightarrow$(ii)$\Longrightarrow$(i) is true for connections of arbitrary form.} 

 (i)$\Longleftrightarrow$(iv) This is clear from \eqref{eq:4 aux 10}. 





\end{proof}

\begin{prop} \label{prop:metric compatibility 2}
	$L$ is constant along $\N$-geodesics if and only if $\A_a\,y^a=-2\,y_{a}\,\ZZ^a/L$. In particular, this is the case if $\covN_i  L=0$. 
\end{prop}
	
\begin{proof}
Let $\gamma(t)$ be an $\N$-geodesic, so that it solves
\[
0=\frac{\mathrm{d}\dot{\gamma}^k }{\mathrm{d}t}+2\,\G^k(\gamma,\dot{\gamma})=\frac{\mathrm{d}\dot{\gamma}^k }{\mathrm{d}t}+\N^k_c(\gamma,\dot{\gamma})\,\dot{\gamma}^c,
\]	
$\G$ being the underlying spray of $\N$.  Then, using that $\gamma$ solves the above equation,
\begin{align*}\frac{\mathrm{d}}{\mathrm{d}t}L(\gamma,\dot{\gamma})=&\dot\gamma^a \partial_a L(\gamma,\dot\gamma)+\frac{\mathrm{d}\dot{\gamma}^a }{\mathrm{d}t}\dot\partial_a L(\gamma,\dot\gamma)\\=&\dot\gamma^a \partial_a L(\gamma,\dot\gamma)-\N^a_c(\gamma,\dot\gamma)\dot\gamma^c\dot\partial_a L(\gamma,\dot\gamma)=\dot\gamma^a\covN_a L.
\end{align*}
Moreover, from \eqref{eq:4 aux 10} and the $2$-homogeneity of $\ZZ$,
\[
y^c\,\covN_c L=-4\,y_a\,\ZZ^a-2L\,\A_a\,y^a,
\]
which concludes the first equivalence. 
 In case that $\covN_i  L=0$, by Prop. \ref{prop:metric compatibility 1}, one has $\A_{i}=-y_{a}\,\ZZ_{\cdot i}^{a}/L$, and by the $2$-homogeneity of $\ZZ$, also $\A_a\,y^a=-2\,y_{a}\,\ZZ^a/L$. 
\end{proof}

\begin{rem} \label{rem:metric compatbility geodesics} From the beginning we assumed that the connections are defined on $A$, where $L$ does not vanish; however, $L$ and $\N$ could be defined further, on some set with vanishing $L$ (as in the case of Def. \ref{defn:proper}). Then 
Prop. \ref{prop:metric compatibility 2} still applies to it. The conclusion is that the tangent vectors to the $\N$-geodesics starting at $\left\{L=0\right\}$ remain in $\left\{L=0\right\}$ (and so the $\N$-geodesics starting at $\left\{L>0\right\}$ or $\left\{L<0\right\}$ remain in these sets as well). In fact, this is true for the pregeodesics of $\N=\LC+\dv\ZZ+\A\otimes\canon$ with arbitrary $\A$, for all of these $\N$'s share pregeodesics with another one that is of the form of Prop. \ref{prop:metric compatibility 1} (see Cor. \ref{cor:projection} (i)). In the case of proper solutions, this result will be improved by Th. \ref{thm:null geodesics}.
\end{rem}

 Next, we will not only use the form of $\N$, but also that it is a solution of the affine equation \eqref{eq:affine equation} (so $\prsim(\N):=\LC+\dv\ZZ\in\solsim$ and $\ZZ$ solves \eqref{eq:affine equation 7}+\eqref{eq:affine equation 8}, see Cor. \ref{cor:projection} and Prop. \ref{prop:torsion-free affine equations 0} respectively). 

\begin{prop} \label{prop:metric compatibility 3}
	For any $\N=\LC+\dv\ZZ+\A\otimes\canon\in\sol$, the following are equivalent: 
	\begin{enumerate}
		\item $g^{ab}\,\covN_i g_{ab}=0$,
		\item $\A_i=-\left(\di+2\right)y_{a}\,\ZZ_{\cdot i}^{a}/\left(2\di\,L\right)$.
	\end{enumerate}
\end{prop}

\begin{proof}
Contracting \eqref{eq:covariant g solution} with $g^{ij}$, 
\[
g^{ab}\,\covN_i g_{ab} =2\,\la_{i}-2\,\C_{a}\,\ZZ^a_{\cdot i}-2\,\ZZ^a_{\cdot a \cdot i}-2\di\,\A_i=-\left(\di+2\right)\frac{y_a}{L}\,\ZZ^a_{\cdot i}-2\di\,\A_i
\]
(the $0$-homogeneity of $\A$ and the fact that $\ZZ$ solves \eqref{eq:4 aux 6} were used).
\end{proof} 

\begin{prop} Let $\di\geq 3$ and, for any $\N=\LC+\dv\ZZ+\A\otimes\canon\in\sol$, consider the following conditions: $\torN_{ij}^k=0$, $\covN_k L=0$, $g^{ab}\,\covN_k g_{ab}=0$. If two of them hold, then actually $\N=\LC$ and the three of them hold. In particular, this is the case when $\covN_k g_{ij}=0$.  
\end{prop}

\begin{proof}
Due to Props. \ref{prop:invariants solution}, \ref{prop:metric compatibility 1} and \ref{prop:metric compatibility 3}, the conditions are equivalent to
\[
\A_i=0,\qquad\A_i=-\frac{y_a}{L}\,\ZZ_{\cdot i}^{a},\qquad\A_i=-\frac{\di+2}{2\di}\,\frac{y_a}{L}\,\ZZ_{\cdot i}^{a}
\]	
respectively, so combining any two of them results in
\[
0=\A_i=\frac{y_a}{L}\,\ZZ_{\cdot i}^{a},
\]
and, by the $2$-homogenity of $\ZZ$,
\[
y_a\,\ZZ_{\cdot b}^{a}\,y^b=2\,y_a\,\ZZ^{a}.
\]
With this, recall form \S \ref{sec:pseudo-finsler} that 
\[
\left(\frac{y_{j}}{L}\right)_{\cdot i}=\frac{g_{ij}}{L}-2\,\frac{y_{i}}{L}\,\frac{y_{j}}{L}
\]
so 
\[
0=\frac{y_a}{L}\,\ZZ_{\cdot i}^{a}=\left(\frac{y_a}{L}\,\ZZ^{a}\right)_{\cdot i}-\left(\frac{y_a}{L}\right)_{\cdot i}\,\ZZ^{a}=-\left(\frac{g_{ia}}{L}-2\,\frac{y_{a}}{L}\,\frac{y_{i}}{L}\right)\ZZ^{a}=-\frac{g_{ia}}{L}\ZZ^a.
\]
As both $\ZZ$ and $\A$ vanish, $\N$ is the metric connection $\LC$.
\end{proof}

\begin{rem}
	Imposing two conditions is required to select $\LC$ among $\sol$, whereas in the classical Palatini formalism only one suffices. While $\covN_k g_{ij}=0$ is enough to select the metric connection, in the Finslerian setting this should be viewed as a fairly strong requirement, for not even $\LC$ always fulfills it ($\covL_k g_{ij}=2\,\la_{ijk}$).
\end{rem}
	
\section{General results on proper solutions} \label{sec:proper solutions}

The standard theory on differential equations is applicable to the local existence of solutions of our affine and metric equations (Theorem \ref{thm:variational equations}), see for example \cite{titt} in the analytic case. So, generically, one would expect a high multiplicity of solutions, but these solutions would be defined only on a neighborhood of some directions in the tangent bundle. However, a more interesting behaviour occurs if one focuses on the global problem which arises when all the elements can be properly extended at $\partial A$. Notice also that, apart from its mathematical interest, this assumption will be relevant from the 
Physics standpoint in order to consider lightlike geodesics.

We will use two different types of techniques for these uniqueness results. The first one relies on a weak hypothesis of analyticity 
and the second one in the maximum principle. In both cases, the behavior of $L$ at $\partial A$ (or the fact that $\partial A=\emptyset$ in the positive definite case) becomes crucial. 

Along this section, 
we will work essentially in dimension 
$\di\geq 3$, which will be required for different reasons, and we will assume the existence of a  prescribed proper 
$L$ (recall Def. \ref{defn:proper} and Rem. \ref{rem:proper}). So,  
$\LC$ and the other metric objects, such as $\sprL$, $\riL$ and $\la$, are also smooth at the boundary\footnote{This is checked just by looking at the coordinate expression \eqref{eq:metric spray} of $\sprL$ and recalling that $\LC$, $\riL$ or $\la$ are constructed with derivatives of it). Note, however, that the assumption of non-degeneracy of $g$ at $\partial A$ becomes essential.}.  
Accordingly, we work with the solutions $\N=\LC+\J$ of the affine equation \eqref{eq:affine equation} that extend smoothly to $\partial A$ (that is, such that $\J$ does).

\begin{defn}\label{DEF_5.1}
	Given the proper pseudo-Finsler metric $L$, we say that $\N$ is a \textit{proper solution of \eqref{eq:affine equation}} if $\N\in\sol$ and it smoothly extends to all of $\overline{A}$. The set of these solutions will be denoted $\solprop$.
\end{defn} 

As a synthesis of \S \ref{sec:affine equation}, keep in mind that the elements of $\sol$ are of the form $\N=\LC+\dv\ZZ+\A\otimes\canon$ for some $\ZZ\in\mathrm{h}^2\mathcal{T}_0^1(\MA)$, $\A\in\mathrm{h}^0\mathcal{T}_1^0(\MA)$ and that then $\prsim(\N):=\LC+\dv\ZZ$ is in $\sol$ as well. In case that $\ZZ$ and $\A$ extend smoothly to $\overline{A}$, we will write $\ZZ\in\mathrm{h}^2\mathcal{T}_0^1(\MAA)$, $\A\in\mathrm{h}^0\mathcal{T}_1^0(\MAA)$, and analogously for anisotropic tensors of all types. The following result justifies restricting further our study to symmetric ($\A=0$) proper solutions. 

\begin{prop} \label{prop:symmetric proper reduction}
	Given $\N=\LC+\dv\ZZ+\A\otimes\canon\in\sol$, it is in $\solprop$ if and only if $\ZZ\in\mathrm{h}^2\mathcal{T}_0^1(\MAA)$ and $\A\in\mathrm{h}^0\mathcal{T}_1^0(\MAA)$. Consequently, $\prsim\vcentcolon\sol\longrightarrow\solsim$ maps $\solprop$ onto $\solsim\cap\solprop$.
\end{prop} 

\begin{proof} 
	Trivially, the smoothness at $\partial A$ of $\ZZ$ and $\A$ suffices for that of $\N$. Conversely, if $\N$ is smooth on  $\overline{A}$, then so is its torsion, from where \eqref{eq:torsion determines A} shows that so is $\A$ (this uses that the canonical $\canon=y^a\,\partial_a$ never vanishes on $\overline{A}$). As now $\N$, $\LC$ and $\A$ are smooth on $\overline{A}$, so must be $\dv\ZZ=\N-\LC-\A\otimes\canon$; by homogeneity, the smoothness of $\dv\ZZ$ anywhere is equivalent to that of $\ZZ$ (because $2\,\ZZ^i=\ZZ^i_{\cdot a}\,y^a$). 
	For the last assertion, if $\N=\LC+\dv\ZZ+\A\otimes\canon\in\solprop$, we have seen that the symmetric solution $\prsim(\N)=\LC+\dv\ZZ$ is smooth on $\overline{A}$ as well. 
\end{proof} 

\begin{rem}\label{r_miguel5}
	The space of proper solutions of the affine equation is the affine space $\solprop$, which is equal to the proper solutions of \eqref{eq:affine equation 2}. Its associated vector space 
	given by the proper solutions of \eqref{eq:affine equation 3}, that is, the equation obtained from  \eqref{eq:affine equation 2} \textit{dropping the Landsberg term} (recall Def.~\ref{def41} and Rem.~\ref{rem:la_i = 0}). 
	From Prop. \ref{prop:symmetric proper reduction} only the space $\solsim\cap\solprop$ will be relevant for the issues of uniqueness. As this is also an affine space, our aim will be to prove that $\Z:=\ZZ-\ZZ_0$ will vanish whenever $\LC+\ZZ,\LC+\ZZ_0\in \solsim\cap\solprop$. Taking into account Lem. \ref{prop: torsion-free affine equations},  the problem  is reduced to the uniqueness of $\Z=0$ as a  solution  of both eqn. \eqref{eq:affine equation 4} \textit{setting $\la_i=  0$ } and either 
	\eqref{eq:affine equation 5} or \eqref{eq:affine equation 6}.
\end{rem}

\subsection{ Fiberwise analytic solutions } \label{sec:analytic}
Taking into account Rem. \ref{r_miguel5}, let us study the uniqueness of $\Z$ on each fiber $A_{p}\subseteq\mathrm{T}_{p}M$, $p\in M$. 
Let $\Z\in\mathrm{h}^{2}\mathcal{T}_{0}^{1}(\MA)$ and define $\s\in\mathrm{h}^{1}\mathcal{F}(A)$, $\K\in\mathrm{h}^{0}\mathcal{T}_{1}^{0}(\MA)$ exactly as in \eqref{eq:sigma}, \eqref{eq:K} recalling 
$\K_a\,y^a=0$ (Rem. \ref{rem:property K}), so that $\Z$ satisfies: 
\begin{equation}
\Z^{i}=2\s\,y^{i}-L\,g^{ia}\left(\s_{\cdot a}+\K_{a}\right),\label{eq:E1}
\end{equation} 
\begin{equation}
\left(\di+2\right)\s - 2\,\C_{a}\,\Z^{a}-\Z_{\cdot a}^{a}=0,\label{eq:E2}
\end{equation}
the latter interchangeable with 
\begin{equation}
\left(\di-2\right)\s-L\,g^{ab}\left(\s_{\cdot a\cdot b}+\K_{a\,\cdot b}\right)=0.\label{eq:E2bis}
\end{equation} 

\begin{lem}
	\label{lem:powers of L} Suppose that  $\Z$ solves \eqref{eq:E1}, \eqref{eq:E2} on $A$,  it extends smoothly to $\overline{A}$ and $n\geq 3$. Then, $\Z$ is divisible up to the boundary by all the powers of $L$, that is,  $\Z=L^{\n}\,\widetilde{\Z}$ for all $\n\in\mathbb{N}$ with $\widetilde{\Z}$ smooth on\footnote{Whenever an anisotropic tensor is said to be  ``divisible by $L^{\n}$'', we mean that the quotient by this is a tensor that extends smoothly to $\partial A=\left\{L=0\right\}$,
	as it is trivially smooth on $A=\left\{L>0\right\}$.} $\overline{A}$.  
\end{lem} 

\begin{proof}
	Reasoning by induction, let $\n=1$. As
	the metric and $\Z$ are smooth on $\overline{A}$, so are $\K$
	(because of its definition \eqref{eq:K}) and $\s$ (because
	of \eqref{eq:E2}). Using this and $\di\geq 3$, \eqref{eq:E2bis} shows that $\s$
	is divisible by $L$: $\s=L\,\widetilde{\s}$ with $\widetilde{\s}$
	smooth on $\overline{A}$. Substituting this in \eqref{eq:E1}:
	\[
	\Z^{i}=L\left\{ 2\widetilde{\s}\,y^{i}-g^{ia}\left(\s_{\cdot a}+\K_{a}\right)\right\} =L\,\widetilde{\Z}^{i}
	\]
	with $\widetilde{\Z}$ smooth on $\overline{A}$. Let us suppose that $\Z$ is divisible by $L^{\n}$ and prove that $\Z$ is actually divisible by $L^{\n+1}$. We do this in five
	steps.
	
	\smallskip
	
	\noindent  \uline{Step \mbox{$1:$}} $\K$ is divisible by $L^{\n-1}$. Indeed, if we substitute $\Z=L^{\n}\,\widetilde{\Z}$ on the definition of $\K$ and use that $L_{\cdot i}=2\,y_{i}$, 
	\begin{equation}
	\begin{split}\K_{i} & =-\frac{2}{\di+2}\left\{ 2L^{\n}\,\C_{a\,\cdot i}\,\widetilde{\Z}^{a}+\C_{a}\left(L^{\n}\,\widetilde{\Z}^{a}\right)_{\cdot i}\right\} \\
	& =-\frac{2}{\di+2}\left\{ 2L^{\n}\,\C_{a\,\cdot i}\,\widetilde{\Z}^{a}+\C_{a}\left(2\n L^{\n-1}\,\widetilde{\Z}^{a}\,y_{i}+L^{\n}\,\widetilde{\Z}_{\cdot i}^{a}\right)\right\} \\
	& =-\frac{2}{\di+2}L^{\n-1}\left(2L\,\C_{a\,\cdot i}\,\widetilde{\Z}^{a}+2\n\,\C_{a}\,\widetilde{\Z}^{a}\,y_{i}+L\,\C_{a}\,\widetilde{\Z}_{\cdot i}^{a}\right)\\
	& =L^{\n-1}\,\widetilde{\K}_{i}
	\end{split}
	\label{eq:K =00003D L^nu-1 K tilde}
	\end{equation}
	with $\widetilde{\K}$ smooth on $\overline{A}$. From $\K_{a}\,y^{a}=0$ (Rem. \ref{rem:property K}),
	it follows that 
	\begin{equation}
	\widetilde{\K}_{a}\,y^{a}=0.\label{eq:K tilde annihilates y}
	\end{equation}
	
	\smallskip 
	
	\noindent \uline{Step \mbox{$2:$}} $\s$ is divisible by $L^{\n}$. First, it is divisible by $L^{\n-1}$: 
	\[
	\s=\frac{y_{a}}{L}\,\Z^{a}=\frac{y_{a}}{L}\,L^{\n}\,\widetilde{\Z}=L^{\n-1}\,\widetilde{\s}
	\]
	(by the definition \eqref{eq:sigma} and the induction hypothesis). It follows that $\widetilde{\s}$ is smooth on $\overline{A}$ and $\left(3-2\n\right)$-homogeneous. Now, rewrite the terms
	appearing in \eqref{eq:E2bis},
	first $L\,g^{ab}\,\K_{a\,\cdot b}$ and then $L\,g^{ab}\,\s_{\cdot a\cdot b}$.
	For the former, we use \eqref{eq:K tilde annihilates y} in the form
	$g^{ab}\,\widetilde{\K}_{a}\,y_{b}=0$ and again $L_{\cdot i}=2\,y_{i}$:
	\begin{equation}
	\begin{split}L\,g^{ab}\,\K_{a\,\cdot b}=L\,g^{ab}\left(L^{\n-1}\,\widetilde{\K}_{a}\right)_{\cdot b} & =L\,g^{ab}\left\{ 2\left(\n-1\right)L^{\n-2}\,\widetilde{\K}_{a}\,y_{b}+L^{\n-1}\,\widetilde{\K}_{a\,\cdot b}\right\} \\
	& =L^{\n}\,g^{ab}\,\widetilde{\K}_{a\,\cdot b}.
	\end{split}
	\label{eq:L g^ab K_a.b}
	\end{equation}
	For the latter,  
	\[
	\s_{\cdot i}=\left(L^{\n-1}\,\widetilde{\s}\right)_{\cdot i}=2\left(\n-1\right)L^{\n-2}\widetilde{\s}\,y_{i}+L^{\n-1}\,\widetilde{\s}_{\cdot i},
	\]	
	\[
	\begin{split}\s_{\cdot i\cdot j} & =2\left(\n-1\right)\left(L^{\n-2}\widetilde{\s}\,y_{i}\right)_{\cdot j}+\left(L^{\n-1}\,\widetilde{\s}_{\cdot i}\right)_{\cdot j}\\
	& =2\left(\n-1\right)\left\{ 2\left(\n-2\right)L^{\n-3}\widetilde{\s}\,y_{i}\,y_{j}+L^{\n-2}\,y_{i}\,\widetilde{\s}_{\cdot j}+L^{\n-2}\widetilde{\s}\,g_{ij}\right\} \\
	& \quad+2\left(\n-1\right)L^{\n-2}\,\widetilde{\s}_{\cdot i}\,y_{j}+L^{\n-1}\,\widetilde{\s}_{\cdot i\cdot j},
	\end{split}
	\]
	and using that $g^{ab}\,y_{a}\,y_{b}=L$ and the $\left(3-2\n\right)$-homogeneity
	of $\widetilde{\s}$,
	\begin{equation}
	\begin{split}L\,g^{ab}\,\s_{\cdot a\cdot b} & =2\left(\n-1\right)L\left\{ 2\left(\n-2\right)L^{\n-2}\widetilde{\s}+\left(3-2\n\right)L^{\n-2}\widetilde{\s}+\di L^{\n-2}\widetilde{\s}\right\} \\
	& \quad+2\left(\n-1\right)\left(3-2\n\right)L^{\n-1}\widetilde{\s}+L^{\n}\,g^{ab}\,\widetilde{\s}_{\cdot a\cdot b},\\
	& =-4\left(\n-1\right)^{2}L^{\n-1}\widetilde{\s}+2\di\left(\n-1\right)L^{\n-1}\widetilde{\s}+L^{\n}\,g^{ab}\,\widetilde{\s}_{\cdot a\cdot b}.\\
	& =-4\left(\n-1\right)^{2}\s+2\di\left(\n-1\right)\s+L^{\n}\,g^{ab}\,\widetilde{\s}_{\cdot a\cdot b}.
	\end{split}
	\label{eq:L g^ab sigma_.a.b}
	\end{equation}
	Substituting \eqref{eq:L g^ab K_a.b} and \eqref{eq:L g^ab sigma_.a.b} in \eqref{eq:E2bis} and rearranging yields 
	\[
	\left\{ 4\left(\n-1\right)^{2}-2\di\left(\n-1\right)+\left(\di-2\right)\right\} \s=L^{\n}\,g^{ab}\left(\widetilde{\s}_{\cdot a\cdot b}+\widetilde{\K}_{a\,\cdot b}\right).
	\]
	The polynomial $4\mathbf{X}^{2}-2\di\mathbf{X}+\left(\di-2\right)$ on $\mathbf{X}$ has no integer roots whenever $\di\neq2$.\footnote{Its roots are $\mathbf{X}=\frac{\di\pm\sqrt{\di^2-4\di+8}}{4}$, so if either of them was an integer, then $\di^2-4\di+8$ would be a perfect square, say $\di^2-4\di+\left(8-m^2\right)=0$  with $m$ integer.  This would mean that $\di=2\pm\sqrt{m^2-4}$, so $m^2-4$ and $m^2$ would be two perfect squares differing by $4$. This is impossible unless $m^2=4$, which corresponds to $\di=2$.} Thus, as required,
	\begin{equation}
	\s=L^{\n}\widetilde{\widetilde{\s}}\label{eq:sigma =00003D L^nu sigma tildetilde}
	\end{equation}
	with $\widetilde{\widetilde{\s}}$ smooth on $\overline{A}$.
	It also follows that $\widetilde{\widetilde{\s}}$ is $\left(1-2\n\right)$-homogeneous. 
	
	\smallskip 
	
	\noindent \uline{Step \mbox{$3:$}} $\K$ is divisible by $L^{\n}$. From the penultimate equality on \eqref{eq:K =00003D L^nu-1 K tilde}, 
	\begin{equation}
	\K_{i}=-\frac{2}{\di+2}L^{\n-1}\left(2L\,\C_{b\,\cdot i}\,\widetilde{\Z}^{b}+2\n\,\C_{b}\,\widetilde{\Z}^{b}\,y_{i}+L\,\C_{b}\,\widetilde{\Z}_{\cdot i}^{b}\right).\label{eq:penultimate equality}
	\end{equation}
	So, it suffices to show that $\C_{a}\,\widetilde{\Z}^{a}$ is divisible
	by $L$. Rewriting \eqref{eq:E1} using induction, 
	\[
	\widetilde{\Z}^{i}=\frac{\Z^{i}}{L^{\n}}=2\s\,\frac{y^{i}}{L^{\n}}-\frac{1}{L^{\n-1}}g^{ia}\left(\s_{\cdot a}+\K_{a}\right).
	\]
	As $\C_{a}\,y^{a}=0$, 
	\[
	\C_{a}\,\widetilde{\Z}^{a}=-\frac{1}{L^{\n-1}}\,\C^{a}\,\s_{\cdot a}-\frac{1}{L^{\n-1}}\,\C^{a}\,\K_{a}.
	\]
	Now we need to check that both $\C^{a}\,\s_{\cdot a}$ and $\C^{a}\,\K_{a}$
	are divisible $\n$ times. For the former, we use 
	\eqref{eq:sigma =00003D L^nu sigma tildetilde}
	and $\C^{a}\,y_{a}=0$: 
	\[
	\C^{a}\,\s_{\cdot a}=\C^{a}\left(L^{\n}\,\widetilde{\widetilde{\s}}\right)_{\cdot a}=\C^{a}\left(2\n L^{\n-1}\widetilde{\widetilde{\s}}\,y_{a}+L^{\n}\,\widetilde{\widetilde{\s}}_{\cdot a}\right)=L^{\n}\,\C^{a}\,\widetilde{\widetilde{\s}}_{\cdot a}.
	\]
	For the latter, again we use \eqref{eq:penultimate equality} and
	$\C^{a}\,y_{a}=0$: 
	\[
	\begin{split}\C^{a}\,\K_{a} & =-\frac{2}{\di+2}L^{\n-1}\left(2L\,\C^{a}\,\C_{b\,\cdot a}\,\widetilde{\Z}^{b}+2\n\,\C_{b}\,\widetilde{\Z}^{b}\,\C^{a}\,y_{a}+L\,\C^{a}\,\C_{b}\,\widetilde{\Z}_{\cdot a}^{b}\right)\\
	& =-\frac{2}{\di+2}L^{\n}\left(2\,\C^{a}\,\C_{b\,\cdot a}\,\widetilde{\Z}^{b}+\C^{a}\,\C_{b}\,\widetilde{\Z}_{\cdot a}^{b}\right).
	\end{split}
	\]
	Going back, these substeps and Rem. \ref{rem:property K}  prove the divisibility
	\begin{equation}
	\K=L^{\n}\,\widetilde{\widetilde{\K}}
	\qquad \hbox{with} \qquad 	
	\widetilde{\widetilde{\K}}_{a}\,y^{a}=0.\label{eq:K tildetilde annihilates y}
	\end{equation}
	
	\smallskip 
	
	\noindent \uline{Step \mbox{$4:$}} $\s$ is divisible by $L^{\n+1}$. Now that we know that $\s=L^{\n}\widetilde{\widetilde{\s}}$ and $\K=L^{\n}\,\widetilde{\widetilde{\K}}$,
	we turn our attention back to \eqref{eq:E2bis}. The analogous computation
	to that on \eqref{eq:L g^ab K_a.b}, this time using \eqref{eq:K tildetilde annihilates y},
	shows that 
	\[
	L\,g^{ab}\,\K_{a\,\cdot b}=L^{\n+1}\,g^{ab}\,\widetilde{\widetilde{\K}}_{a\,\cdot b}.
	\]
	The analogous computations to those leading to \eqref{eq:L g^ab sigma_.a.b},
	this time using the $\left(1-2\n\right)$-homogeneity of $\widetilde{\widetilde{\s}}$,
	shows that 
	\[
	L\,g^{ab}\,\s_{\cdot a\cdot b}=-4\n^{2}\s+2\di\,\nu\s+L^{\n+1}\,g^{ab}\,\widetilde{\widetilde{\s}}_{\cdot a\cdot b}.
	\]
	Substituting these in \eqref{eq:E2bis} and rearranging yields
	\[
	\left\{ 4\n^{2}-2\di\n+\left(\di-2\right)\right\} \s=L^{\n+1}\,g^{ab}\left(\widetilde{\widetilde{\s}}_{\cdot a\cdot b}+\widetilde{\widetilde{\K}}_{a\,\cdot b}\right),
	\]
	and the inexistence of integer roots of  $4\mathbf{X}^{2}-2\di\mathbf{X}+\left(\di -2\right)$ yields the divisibility 
	\[
	\s=L^{\n+1}\widetilde{\widetilde{\widetilde{\s}}}.
	\]
	
	\noindent \uline{Step \mbox{$5:$}} $\Z$ is divisible by $L^{\n+1}$. Substituting $\s=L^{\n+1}\widetilde{\widetilde{\widetilde{\s}}}$,
	$\K=L^{\n}\,\widetilde{\widetilde{\K}}$ in \eqref{eq:E1} and computing, one gets $\Z^{i}=L^{\n+1}\,\widetilde{\widetilde{\Z}}^{i}$ 
	with $\widetilde{\widetilde{\Z}}$ smooth on $\overline{A}$, which completes the proof. 
\end{proof}
\begin{rem}
	\label{rem:powers of L}  Assume that $\LC+\dv\ZZ\in\solsim\cap\solprop$ (so that $\ZZ\in\mathrm{h}^{2}\mathcal{T}_{0}^{1}(\MAA)$ solves  \eqref{eq:affine equation 4}, \eqref{eq:affine equation 5}) and that $\la_i$ is divisible up to $\partial A$ by $L^{\n}$, where $\n\in\mathbb{N}\cup\left\{0\right\}$. Then the argument above proves that $\ZZ$ is divisible by $L^{\n+1}$. In particular, $\ZZ$ always is divisible by $L$. 
\end{rem}
\begin{defn}\label{DEF_5.5}
	We say that an anisotropic tensor $T\in\mathrm{h}^{\alpha}\mathcal{T}_{s}^{r}(\MAA)$
	is \textit{fiberwise analytic on $\overline{A}$} if it is analytic when restricted to every $\overline{A_p}\subseteq\mathrm{T}_{p}M$.
\end{defn}
\begin{rem}\label{REM_5.6} In coordinates,  $T$ is fiberwise analytic when all $T^{i_1, \dots i_r}_{j_1,\dots j_s}(x,y)$ are analytic in  $y$. In particular,  this property holds for most explicit pseudo-Finsler metrics,  $L\equiv L(x,y)$, such as   pseudo-Riemannian or Randers ones. This notion does not require of any additional analytic structure to be well-defined: each $\mathrm{T}_pM$ has a canonical one as a vector space. By contrast, the notion of being \textit{analytic on $\overline{A}$} does.  
Anyway, obviously, ``analytic'' implies ``fiberwise analytic''. 
\end{rem}
\begin{thm} \label{thm:analytic}
	\label{thm: analytic uniqueness} Assume that the proper  pseudo-Finsler metric $L$ is of non-definite signature and $n\geq 3$. Then there exists at most one $\N=\LC+\dv\ZZ\in\solsim\cap\solprop$ such that the spray difference $-2\,\ZZ=\G-\sprL$ (equiv., the connection difference $\dv\ZZ=\N-\LC$) is fiberwise analytic on $\overline{A}$. 
\end{thm}
\begin{proof}
	The analyticity (resp., fiberwise analyticity) of $-2\,\ZZ$ is equivalent to that of $\dv\ZZ$ because this is constructed with fiber derivatives of $\ZZ$ but also $-2\,\ZZ^i=-\ZZ^i_{\cdot a}\,y^a$. 
	
	Let $\N_{0}=\LC+\dv\ZZ_{0}$ be another  solution with the same properties.  Then $\Z:=\ZZ-\ZZ_{0}$ is fiberwise analytic on $\overline{A}$ too. By Prop. \ref{prop:symmetric proper reduction}, $\Z$ is smooth there, and by Lem. \ref{prop: torsion-free affine equations}, it solves \eqref{eq:E1}+\eqref{eq:E2}. For all $\n\in\mathbb{N}$, Lem. \ref{lem:powers of L} allows us to write $\Z=L^{\n}\,\widetilde{\Z}$ with $\widetilde{\Z}$ smooth on $\overline{A}$. After restricting this to each $\overline{A_p}$, when one computes the vertical derivatives of the functions $\Z^{i}$ by induction,
	it becomes clear that $\Z_{\cdot j_{1}\cdot j_{2}...\cdot j_{\n-1}}^{i}=L\,T_{j_{1}...j_{\n-1}}^{i}$ with $T_{j_{1}...j_{\n-1}}^{i}$ a smooth function on $\overline{A_p}$. This shows that all derivatives of all orders vanish on $\partial A_p=\left\{v\in\overline{A_p}:\:L(v)=0\right\}$. Now we develop $\Z^{i}$ in Taylor series on an open subset of $\overline{A_p}$ around some $v\in\partial A_p$ (this exists due to the signature being non-definite). Clearly the analytic $\Z^{i}$ vanishes on that open set and, as $A_p$ is connected, it vanishes on all of $A_p$. Thus, $\ZZ_p=\left.\ZZ_{0}\right|_p+\Z_p=\left.\ZZ_{0}\right|_p$.  
\end{proof}

\begin{cor}
	 With the hypotheses of Th. \ref{thm:analytic}, in case that $L$ (equiv., $g$) is analytic on $\overline{A}$, there exists at most one symmetric and proper solution $\N=\LC+\dv\ZZ$ of the affine equation \eqref{eq:affine equation} analytic on $\overline{A}$. 
\end{cor}

\begin{proof}
	 The analyticity of $L$ is equivalent to that of $g$ by the analogous reasoning as in the theorem above. In case that $L$ is analytic, so are $\sprL$ and $\LC=\dv\sprL$ (recall the coordinate expression \eqref{eq:metric spray}), so the analyticity of $\N=\LC+\dv\ZZ$ becomes equivalent to that of $\dv\ZZ$ and implies its fiberwise analyticity. Thus, Th. \ref{thm:analytic} applies. 
\end{proof}

\begin{rem} \label{rem:nonexistence}
	The techniques above can be used to obtain nonexistence results for
	fiberwise analytic solutions in some cases. Namely, if $\la_{i}$
	is not $0$ but it is divisible by all the powers of $L$ (what implies that $\la_{i}$ is not fiberwise analytic on $\overline{A}$), then no proper solution $\N=\LC+\dv\ZZ$ with $\ZZ$ fiberwise analytic can exist (indeed, by Rem. \ref{rem:powers of L} such a $\ZZ$ would be divisible by all the powers of $L$ too and the same argument of Th. \ref{thm: analytic uniqueness} would prove that $\ZZ=0$, contradicting $\la_{i}\neq0$). 
\end{rem}

A relevant issue is whether the $\N$-geodesics will be defined on all the $L$-lightlike directions, which becomes obviously important for physical interpretations in Lorentzian signature. We will take advantage of the fact that $\ZZ$ is always divisible by $L$ (Rem. \ref{rem:powers of L}) to prove that every symmetric and proper solution of the affine equation \eqref{eq:affine equation} shares its lightlike geodesics with $L$, notably with their parametrizations included. In the Lorentz-Finsler case, they are the cone geodesics of the cone structure naturally associated with $L$ \cite[Th. 6.6]{cones} with distinguished parametrizations. Recall that the tangent vectors to the $L$-geodesics starting at $\partial A=\left\{L=0\right\}$ remain in $\partial A$ (this, for instance, follows from Prop. \ref{prop:metric compatibility 2} by taking $\ZZ=0$ and $\A =0$).

\begin{thm} \label{thm:null geodesics}
	Let $\N=\LC+\dv\ZZ\in\solsim\cap\solprop$. Then the unique $\N$-geodesic starting at each $v\in\partial A$ coincides with the corresponding (lightlike) $L$-geodesic. 
\end{thm}

\begin{proof}
	We saw that $\ZZ=L\,\widetilde{\ZZ}$ with $\widetilde{\ZZ}$ smooth on $\overline{A}$. Let $\gamma(t)$ be the unique $L$-geodesic with initial condition $\dot{\gamma}(0)=v$, so that it solves 
	\[
	\frac{\mathrm{d}{\dot{\gamma}}^i}{\mathrm{d}t}+2\left(\sprL\right)^{i}(\dot{\gamma}(t))=0.
	\]
	Then $L(\dot{\gamma}(t))=L(v)=0$ and $\ZZ_{\dot{\gamma}(t)}=L(\dot{\gamma}(t))\,\widetilde{\ZZ}_{\dot{\gamma}(t)}=0$, allowing us to write
	\[
	0=\frac{\mathrm{d}{\dot{\gamma}}^i}{\mathrm{d}t}+2\left(\sprL\right)^{i}(\dot{\gamma}(t))+2\,\ZZ^{i}(\dot{\gamma}(t))=\frac{\mathrm{d}{\dot{\gamma}}^i}{\mathrm{d}t}+2\,\G^{i}(\dot{\gamma}(t)).
	\]
	Recall that $\G$ is the underlying spray of $\N$, so $\gamma(t)$ turns out to be the $\N$-geodesic with initial condition $v$. 
\end{proof}

\begin{rem}\label{REM_5.10}  Although we have been working with proper metrics, as far as the results of this section \ref{sec:analytic} are concerned, this assumption can be somewhat weakened. Indeed, assume only: (i) each fiber $A_p$ ($p\in M$) is connected and $L\neq 0$ on it; (ii) $L$ extends smoothly to some conic $B$ with $A\subseteq B\subseteq\overline{A}\subseteq\mathrm{T}M\setminus\mathbf{0}$ and $g$  is non-degenerate therein; (iii) each $B_p\setminus A_p$ is nonempty and formed by $L$-lightlike directions.
	Accordingly, consider those $\N=\LC+\dv\ZZ\in\solsim$ that extend smoothly to $B$. Then Ths. \ref{thm: analytic uniqueness} and \ref{thm:null geodesics}, as well as Rem. \ref{rem:nonexistence}, still hold true. Moreover, Lem. \ref{lem:powers of L} and Th. \ref{thm: analytic uniqueness} could straightforwardly be stated for a single fiber $B_p$. Summing up, the  point here is that 
	the techniques of this subsection do not really require of any global hypothesis at the boundary of each $A_p$, but only the existence at each point of a lightlike direction to which $L$ and $\N$ can be smoothly extended. By contrast, those of the next subsection will actually require of solutions  defined on the whole $\overline{A_p}$. 
\end{rem}

\subsection{Results from scalar elliptic PDEs} \label{sec:elliptic} Inspired by \eqref{eq:metric equation} and \eqref{eq:affine equation 6}, we consider
the equation 
\begin{equation}
\ka\f-L\,g^{ab}\,\f_{\cdot a\cdot b}=0\label{eq:E3}
\end{equation}
with parameter $\ka\in\mathbb{R}$. This time we emphasize its study on each single fiber $A_p$ ($p\in M$) and   we   work in coordinates adapted to its homogeneity. Thus, regard (by restriction) $\f$ as an $\alpha$-homogeneous   smooth   function on $A_p$ and take another positive $1$-homogeneous function $\R$ there (in particular, we will take $\R=F_p=\sqrt{L_p}$ later). Consider the smooth \footnote{Regarding (also by restriction) $(y^1,...,y^\di)$ as linear coordinates on $T_pM\supseteq A_p$, by homogeneity one has $\mathrm{d}\R_v(\canon_v^\mathrm{V})=y^a(v)\,\R_{\cdot a}(v)=\R(v)=1\neq 0$ for $v\in\SR\subseteq A_p$.} hypersurface $\SR=\left\{ \R=1\right\} $,
so that 
\[
A_p\equiv\mathbb{R}^{+}\times\SR,\qquad v\equiv(\R(v),\frac{v}{\R(v)}).
\]
The indices $\II$, $\JJ$ will run in the set $\left\{1,...,\di-1\right\}$. Take coordinates $(\z_{\mathrm{\Sigma}}^{1},...,\z_{\mathrm{\Sigma}}^{\di-1})$
on $\SR$. Together with the natural coordinate on $\mathbb{R}^{+}$,
they induce coordinates on $A_p$. These turn out to be $(\R,\z_{A}^{1},...,\z_{A}^{\di-1})$,
where the $\z_{A}^{\II}$'s are the $\z_{\Sigma}^{\II}$'s extended by $0$-homogeneity:
\[
\z_{A}^{\II}(v)=\z_{\Sigma}^{\II}(\frac{v}{\R(v)}).
\]
We refer to $(\R,\z_{A}^{1},...,\z_{A}^{\di-1})$ as \textit{generalized polar coordinates}. 

By the $1$-homogeneity of $\R$ and the $0$-homogeneity of the $\z_{A}^{\II}$'s,
\[
\canon^\mathrm{V}=y^{a}\,\partial_{y^{a}}=y^{a}\left(\frac{\partial\R}{\partial y^{a}}\,\partial_{\R}+\frac{\partial\z_{A}^{\AAA}}{\partial y^{a}}\,\partial_{\z_{A}^{\AAA}}\right)=\R\,\partial_{\R}
\]
on $A_p$. For $v_{0}\in\SR$, one straightforwardly checks that $(v_{0},\left.\partial_{\z_{\Sigma}^{1}}\right|_{v_{0}},...,\left.\partial_{\z_{\Sigma}^{\di-1}}\right|_{v_{0}})$
is the dual basis of $(\mathrm{d}\R_{v_{0}},\left(\mathrm{d}\z_{A}^{1}\right)_{v_{0}},...,\left(\mathrm{d}\z_{A}^{\di-1}\right)_{v_{0}})$,
so $\left.\partial_{\z_{\Sigma}^{\II}}\right|_{v_{0}}=\left.\partial_{\z_{A}^{\II}}\right|_{v_{0}}$. From now on we will not distinguish between the $\z_{\Sigma}$ and the $\z_{A}$, denoting either of them by $\z$. For $f$, being $\alpha$-homogeneous means that 
\[
f(\R,\z^{1},...,\z^{\di-1})=f_{\SR}(\z^{1},...,\z^{\di-1})\,\R^{\alpha},
\]
so $\partial_{\z^{\II}}f$ is $\alpha$-homogeneous as well. 
\begin{lem}
	\label{lem:equation f} 
	Let $n\geq 2$. Any $\alpha$-homogeneous solution $\f$ of \eqref{eq:E3} on $A_p$ must be $f=0$ in any of the following two cases:
	
	  (A)   $L$ is Lorentz-Finsler, $f$ extends smoothly to $\overline{A_p}$, $\ka\neq0$, $\alpha\leq2$, and $\ka\leq\alpha\left(\alpha+\di-2\right)$
	with one of these inequalities being strict.
	
	  (B)   $L$ is Finsler (thus $A_p=\overline{A_p}=\mathrm{T}_pM\setminus 0$) and $\kappa >\alpha(\alpha+n-2)$.

\end{lem}
\begin{proof} \noindent \uline{Case   (A) .} First, rewrite \eqref{eq:E3} on $A_p$ in terms of $F=\sqrt{L} (>0)$, 
	
	\begin{equation}
	\ka\frac{\f}{F^{\alpha}}-F^{2-\alpha}\,g^{ab}\,\f_{\cdot a\cdot b}=0,\label{eq:E3 2}
	\end{equation}
	and this expression in terms of 
	\[
	\widetilde{\f}=\frac{\f}{F^{\alpha}}.
	\]
	Using $F_{\cdot i}=y_{i}/F$, $g^{ab}\,y_{a}\,y_{b}=F^2$ (\S \ref{sec:pseudo-finsler}) and the $0$-homogeneity of $\widetilde{\f}$,
	\[
	\f_{\cdot i}=\left(F^{\alpha}\widetilde{\f}\right)_{\cdot i}=\alpha F^{\alpha-2}\widetilde{\f}\,y_{i}+F^{\alpha}\,\widetilde{\f}_{\cdot i},
	\]
	\[
	\begin{split}\f_{\cdot i\cdot j} & =\alpha\left\{ \left(\alpha-2\right)F^{\alpha-4}\widetilde{\f}\,y_{i}\,y_{j}+F^{\alpha-2}\,y_{i}\,\widetilde{\f}_{\cdot j}+F^{\alpha-2}\widetilde{\f}\,g_{ij}\right\} \\
	& \quad+\alpha F^{\alpha-2}\,\widetilde{\f}_{\cdot i}\,y_{j}+F^{\alpha}\,\widetilde{\f}_{\cdot i\cdot j},
	\end{split}
	\]
	\[
	\begin{split}
	F^{2-\alpha}\,g^{ab}\,\f_{\cdot a\cdot b} & =\alpha\left(\alpha+\di -2\right)
	\widetilde{\f}+
	F^2 \,g^{ab}\,\widetilde{\f}_{\cdot a\cdot b}. 
	\end{split} 
	\]
	Substituting this and rearranging, \eqref{eq:E3 2} reads 
	\begin{equation}
	-L\,g^{ab}\,\widetilde{\f}_{\cdot a\cdot b}-\left\{ \alpha\left(\alpha+\di-2\right)-\ka\right\} \widetilde{\f}=0.\label{eq:E3 3}
	\end{equation}
	
	
	Now, rewrite \eqref{eq:E3 3} in generalized polar coordinates
	$(\R,\z^{1},...,\z^{\di-1})$ with $\R=F_p$, so that $\SR$
	is the indicatrix   of $L$   at $p$ and $(\z^{1},...,\z^{\di-1})$ are global coordinates on $\SR$ 
	with values in a relatively compact domain\footnote{As $\overline{A_p}$ is contained in an open half-space determined
	by some vector hyperplane $\varPi_p\subseteq\mathrm{T}_{p}M$ (Rem. \ref{rem:proper}   (B)), any hyperplane   $\varXi_p$  contained in that   half-space   and parallel to  $\varPi_p$ will be intersected exactly once by each ray in $\overline{A_p}$. These points give $D\subseteq  \varXi_p  $ and its  boundary $\partial D$, which is the intersection of the cone $\partial A_p$ with   $\varXi_p$.  } $D\subseteq\mathbb{R}^{n-1}$ which then are extended to $A_p$ by 0-homogeneity. 
	Using  $\partial_{\R}=\R^{-1}\,\canon^\mathrm{V}$ and $\canon^\mathrm{V}(\widetilde{\f})=0$ (0-homogeneity of $\widetilde{\f}$), 
	\[
	\widetilde{\f}_{\cdot i}=\partial_{y^{i}}\widetilde{\f}=\frac{\partial\R}{\partial y^{i}}\,\partial_{\R}\widetilde{\f}+\frac{\partial\z^{\AAA}}{\partial y^{i}}\,\partial_{\z^{\AAA}}\widetilde{\f}=\frac{\partial\z^{\AAA}}{\partial y^{i}}\,\partial_{\z^{\AAA}}\widetilde{\f}.
	\]
	Using that $\partial_{\z^{\II}}\widetilde{\f}$ is $0$-homogeneous
	too, 
	\[
	\begin{split}\widetilde{\f}_{\cdot i\cdot j}=\partial_{y^{j}}(\frac{\partial\z^{\AAA}}{\partial y^{i}}\,\partial_{\z^{\AAA}}\widetilde{\f}) & =\frac{\partial^{2}\z^{\AAA}}{\partial y^{i}\,\partial y^{j}}\,\partial_{\z^{\AAA}}\widetilde{\f}+\frac{\partial\z^{\AAA}}{\partial y^{i}}\,\partial_{y^{j}}(\partial_{\z^{\AAA}}\widetilde{\f})\\
	&  =\frac{\partial^{2}\z^{\AAA}}{\partial y^{i}\,\partial y^{j}}\,\partial_{\z^{\AAA}}\widetilde{\f}+\frac{\partial\z^{\AAA}}{\partial y^{i}}\,\frac{\partial\z^{\BB}}{\partial y^{j}}\,\partial_{\z^{\AAA}\z^{\BB}}^{2}\widetilde{\f}. 
	\end{split}
	\]
	From these, 
	\[
	\begin{split}L\,g^{ab}\,\widetilde{\f}_{\cdot a\cdot b} & =L\,g^{ab}\,\frac{\partial^{2}\z^{\AAA}}{\partial y^{a}\,\partial y^{b}}\,\partial_{\z^{\AAA}}\widetilde{\f}+L\,g^{ab}\,\frac{\partial\z^{\AAA}}{\partial y^{a}}\,\frac{\partial\z^{\BB}}{\partial y^{b}}\,\partial_{\z^{\AAA}\z^{\BB}}^{2}\widetilde{\f}\\
	& =L\,g^{ab}\,\frac{\partial^{2}\z^{\AAA}}{\partial y^{a}\,\partial y^{b}}\,\partial_{\z^{\AAA}}\widetilde{\f}+L\,g^{-1}(\mathrm{d}y^{a},\mathrm{d}y^{b})\,\frac{\partial\z^{\AAA}}{\partial y^{a}}\,\frac{\partial\z^{\BB}}{\partial y^{b}}\,\partial_{\z^{\AAA}\z^{\BB}}^{2}\widetilde{\f}\\
	& =L\,g^{ab}\,\frac{\partial^{2}\z^{\AAA}}{\partial y^{a}\,\partial y^{b}}\,\partial_{\z^{\AAA}}\widetilde{\f}+L\,g^{-1}(\mathrm{d}\z^{\AAA},\mathrm{d}\z^{\BB})\,\partial_{\z^{\AAA}\z^{\BB}}^{2}\widetilde{\f}.
	\end{split}
	\]
	Substituting this, \eqref{eq:E3 3} reads 
	\begin{equation}
	\begin{split}-L\,g^{-1}(\mathrm{d}\z^{\AAA},\mathrm{d}\z^{\BB})\,\partial_{\z^{\AAA}\z^{\BB}}^{2}\widetilde{\f}-L\,g^{ab}\,\frac{\partial^{2}\z^{\AAA}}{\partial y^{a}\,\partial y^{b}}\,\partial_{\z^{\AAA}}\widetilde{\f}\\
	-\left\{ \alpha\left(\alpha+\di-2\right)-\ka\right\} \widetilde{\f} & =0.
	\end{split}
	\label{eq:E3 4}
	\end{equation}
	To check that  the matrix  $g^{-1}(\mathrm{d}\z^{\II},\mathrm{d}\z^{\JJ})_{\SR}$ 
	is negative definite, notice that, for each  $v_{0}\in\SR$, $g_{v_{0}}$ is of signature
	$(+,-,...,-)$, the radial direction $v_0$ is positive definite and $g_{v_{0}}$-orthogonal to   $\mathrm{T}_{v_{0}}\Sigma_p=\mathrm{Span}\left\{\partial_{z^1}|_{v_0},...,\partial_{z^{n-1}}|_{v_0}\right\}$   and the $g_{v_{0}}$-flat isomorphism maps $\mathrm{T}_{v_{0}}\Sigma_p$ into   $\mathrm{Span}\left\{\mathrm{d}{z^1_{v_0}},...,\mathrm{d}{z^{n-1}_{v_0}}\right\}$.

	
	 The restriction $\widetilde{\f}_{\SR}$  satisfies  \eqref{eq:E3 4} on its domain $D$ with $L=1$:
	\begin{equation}
	\begin{split}-g^{-1}(\mathrm{d}\z^{\AAA},\mathrm{d}\z^{\BB})_{\SR}\,\partial_{\z^{\AAA}\z^{\BB}}^{2}\widetilde{\f}_{\SR}-\left(g^{ab}\,\frac{\partial^{2}\z^{\AAA}}{\partial y^{a}\,\partial y^{b}}\right)_{\SR}\partial_{\z^{\AAA}}\widetilde{\f}_{\SR}\\
	-\left\{ \alpha\left(\alpha+\di-2\right)-\ka\right\} \widetilde{\f}_{\SR} & =0.
	\end{split}
	\label{eq:E3 5}
	\end{equation}
	This equation is  uniformly elliptic on
	compact subsets, as $-g^{-1}(\mathrm{d}\z^{\II},\mathrm{d}\z^{\JJ})_{\SR}$ is continuous and
	positive definite 
	(see \cite[Ch. 3]{gilbarg-trudinger}). Moreover, one of our hypothesis
	is $-\left\{ \alpha\left(\alpha+\di-2\right)-\ka\right\} \leq0$,
	thus,
	the classical maximum principles \cite[\S 3.1 and 3.2]{gilbarg-trudinger} will be applicable to its solutions.  In particular, a standard application of the weak maximum principle 
	\cite[Th. 3.3]{gilbarg-trudinger} shows that $\widetilde{\f}_{\SR}$ and $f_0= 0$ are equal if $\widetilde{\f}_{\SR}$ is continuous and vanishes on  $\partial D$. These conditions follow from \eqref{eq:E3 2} when $\alpha<2$  (recall that $F^{2-\alpha}$ vanishes on $\partial A_p$ and $f$ is smooth therein  by hypothesis), while if $\alpha=2$,
	\eqref{eq:E3 2} still implies that $\widetilde{\f}$ is smooth on $\overline{A_p}$ and the result follows from \eqref{eq:E3 4} using the hypothesis of strict inequality for $\kappa$.
	

	  \uline{Case   (B).} Now, the coordinates $(z^1,...,z^{n-1})$ cannot cover the whole indicatrix   $\SR$ (which is compact)   but,  if $\widetilde{\f}_{\SR}$ is not constant, we can take them around any maximum $v_m\in   \SR  $ where $\widetilde{\f}_{\SR}$  is not locally equal to $c_m:=\widetilde{\f}_{\SR}(v_m)$.
	Reasoning as in the case   (A) , one  arrives at \eqref{eq:E3 4}  and (say, after an overall change of sign) strict uniform ellipticity follows  from the new hypothesis on $\kappa$.  If $c_m\geq 0$,
	a direct application of the strong maximum principle 
	\cite[Th. 3.5]{gilbarg-trudinger} shows that $\widetilde{\f}_{\SR}$ has to be locally equal to $c_m$. So, $\widetilde{\f}_{\SR}$ must be  constant and, by  \eqref{eq:E3 4}, equal to 0. If $c_m\leq 0$,  reason with $-\widetilde{\f}_{\SR}$. 
\end{proof}

%

In Th. \ref{thm: analytic uniqueness}, we obtained a general uniqueness result for solutions of the torsion-free affine equations \eqref{eq:affine equation 4}, \eqref{eq:affine equation 5} under the hypothesis of  fiberwise-analyticity. As a first application of Lemma 5.1,  this hypothesis is dropped in some particular cases.

\begin{thm}
	\label{thm:C_i =00003D 0 Lorentz-Finsler} Assume that $L$ is  Lorentz-Finsler and $n\geq 3$. If $\N=\LC+\dv\ZZ\in\solsim\cap\solprop$ and 
	\begin{equation}
	2\,\C_{a\,\cdot i}\,\ZZ^{a}+\C_{a}\,\ZZ_{\cdot i}^{a}+\la_{i}=0,\label{eq:K =00003D 0}
	\end{equation}
	then actually $\ZZ=0$ and thus $\la_{i}=0$. 
\end{thm}
\begin{proof}
	Using the notation \eqref{eq:K}, the hypothesis \eqref{eq:K =00003D 0} means 
	\[
	\mathcal{K}_{i}^{\ZZ}=\frac{2}{\di+2}\,\la_i.
	\]
	Thus, the equations \eqref{eq:affine equation 4}, \eqref{eq:affine equation 5}, \eqref{eq:affine equation 6} 
	read, respectively, 
	\begin{equation}
	\ZZ^{i}=2\sigma^{\ZZ}\,y^{i}-L\,g^{ia}\,\sigma_{\cdot a}^{\ZZ},
	\label{eq:5 aux 1}
	\end{equation}
	\begin{equation}
	\left(\di+2\right)\sigma^{\ZZ} =- \la_i , \label{eq:5 aux 2}
	\end{equation}
	\begin{equation}
	\left(\di-2\right)\sigma^{\ZZ}-L\,g^{ab}\,\sigma^{\ZZ}_{\cdot a\cdot b}=0.
	\label{eq:5 aux 3}
	\end{equation}
	 The function $\f:=\sigma^{\ZZ}_p$, which is smooth on $\overline{A_p}$ by \eqref{eq:5 aux 2}, solves \eqref{eq:E3} on $A_p$ with parameters $\alpha=1$, $\ka=\di-2$ (by \eqref{eq:5 aux 3}). Applying Lem. \ref{lem:equation f} (recall $\ka\neq 0$ as $\di\geq 3$) yields $\sigma^{\ZZ}_p=0$,  for all $p\in M$. Thus, \eqref{eq:5 aux 1} yields $\ZZ=0$. Finally, recall Rem. \ref{rem:la_i = 0}: $\LC$ being in $\sol$ implies $\la_i=0$.
\end{proof}

\begin{cor} \label{COR_5.14}
	If $L$ is Lorentz-Finsler with vanishing mean Cartan tensor ($\C_i=0$) and $n\geq 3$, then its associated nonlinear connection $\LC$ is the unique element of $\solsim\cap\solprop$.
\end{cor}

\begin{proof}
	As the mean Landsberg tensor can be written as a derivative of $\C_i$ (see \cite[(6.37)]{spray and finsler}), the hypothesis \eqref{eq:K =00003D 0} follows trivially and Th. \ref{thm:C_i =00003D 0 Lorentz-Finsler} applies.
\end{proof}

\begin{rem} \label{rem:C_i = 0 lorentzian}
	In \cite[Remark 5.3]{minguzzi}, the relevance of the condition $\C_{i}=0$ in the study of alternative  Finslerian Einstein equations is stressed, namely, it guarantees the symmetry of certain Ricci tensors. In the positive definite case,
	Deicke's   Theorem   \cite[Th. 14.4.1]{BCS} establishes that the only Finsler metrics with $\C_{i}=0$ are the Riemannian ones.  The Berwald-Moor metrics  \cite{berwald-moor} are improper Lorentz-Finsler counterexamples, as they cannot be properly extended to $\partial A$; as far as we know, no proper Lorentz-Finsler counterexamples appears in the literature. 
\end{rem}

In Lem. \ref{lem:equation f}, the case   (B)   provided a positive definite version of the case   (A). However,  it did so for $\ka>\alpha\left(\alpha+\di-2\right)$, which is the opposite inequality arising in the proof Th. \ref{thm:C_i =00003D 0 Lorentz-Finsler}; this prevents a result for Finsler instead of Lorentz-Finsler metrics. However, we are going to prove that the uniqueness of solutions in the Riemannian case can be obtained by means of a further study of the Laplacian of $f$, that is, the  solutions in the Riemannian Palatini approach agree with those in the Finslerian Palatini one. For the following result, recall that in the case of Finsler metrics, $A=\mathrm{T}M\setminus\mathbf{0}$; hence, all the corresponding solutions of the affine equation \eqref{eq:affine equation} are trivially proper. 

\begin{thm} \label{thm:C_i = 0 riemannian}
	Assume that $L$ is (positive definite) Riemannian and $n\geq 3$. Then $\LC$ is the only element of $\solsim=\mathrm{Sol}_L^{\mathscr{S}\mathrm{ym}}(\mathrm{T}M\setminus\mathbf{0})$.
\end{thm}
\begin{proof}
	Let $\N=\LC+\dv\ZZ\in\mathrm{Sol}_L^{\mathscr{S}\mathrm{ym}}(\mathrm{T}M\setminus\mathbf{0})$. By using, in Lem. \ref{prop: torsion-free affine equations}, the vanishing of the mean Cartan and Landsberg tensors, $\ZZ$ solves
	\begin{equation}
	\ZZ^{i}=2\sigma^{\ZZ}\,y^{i}-L\,g^{ia}\,\sigma^{\ZZ}_{\cdot a},
	\label{eq:5 aux 4}
	\end{equation}
	\begin{equation}
	\left(\di-2\right)\sigma^{\ZZ}-L\,g^{ab}\,\sigma^{\ZZ}_{\cdot a\cdot b}=0.
	\label{eq:5 aux 5}
	\end{equation}
	When rewritting \eqref{eq:5 aux 5}  in terms of 
	\[
	\widetilde{\sigma^{\ZZ}}=\frac{\sigma^{\ZZ}}{F}\in\mathrm{h}^0\mathcal{F}(\mathrm{T}M\setminus\mathbf{0})
	\]
	(put $\alpha=1$ and $\kappa=n-2$ in \eqref{eq:E3 3}),
	one gets 
	\begin{equation}
	L\,g^{ab}\,\widetilde{\sigma^{\ZZ}}_{\cdot a\cdot b}+\widetilde{\sigma^{\ZZ}}=0,
	\label{eq:5 aux 6}
	\end{equation}
	which in turn can be restricted to each $\mathrm{T}_p M\setminus0$. This time, $g_p$ is just a positive definite scalar
	product on $\mathrm{T}_pM$, its indicatrix being a round sphere:
	$\SL=\left\{v\in\mathrm{T}_p M\setminus0:\:L(v)=1\right\}\equiv\mathbb{S}^{\di-1}$. Thus, $g^{ab}\,\partial_{y^{a}\,y^{b}}^{2}$ is the Laplacian of the Euclidean $\mathbb{R}^{\di}$ and, as $\widetilde{\sigma^{\ZZ}}_p$ is $0$-homogeneous, it is well-known \cite[Prop. 22.1]{laplacian sphere}
	that 
	\[
	\left(g^{ab}\,\widetilde{\sigma^{\ZZ}}_{\cdot a\cdot b}\right)_{ \mathbb{S}^{n-1}}=\Delta_{ \mathbb{S}^{n-1}}\widetilde{\sigma^{\ZZ}}.
	\]
	Because of this, \eqref{eq:5 aux 6} restricted to $\mathbb{S}^{\di-1}$ becomes 
	\begin{equation}
	-\Delta_{\mathbb{S}^{\di-1}}\widetilde{\sigma^{\ZZ}}=\widetilde{\sigma^{\ZZ}}.
	\label{eq:5 aux 7}
	\end{equation}
	The set of eigenvalues of $-\Delta_{\mathbb{S}^{\di-1}}$ is 
	\[\mathrm{Spec}(-\Delta_{\mathbb{S}^{\di-1}})=\left\{ \n\left(\n+\di-2\right):\:\n\in\mathbb{N}\cup\left\{ 0\right\} \right\}\] (\cite[Th. 22.1]{laplacian sphere}, we follow the conventions of this reference).
	As $\di\geq3$, then $1\notin\mathrm{Spec}(-\Delta_{\mathbb{S}^{\di-1}})$ and $\widetilde{\sigma^{\ZZ}}=0$, as it solves
	\eqref{eq:5 aux 7}. Thus, $\ZZ=0$ from \eqref{eq:5 aux 4}, as required. 
\end{proof}

 The following last consequence of Lem. \ref{lem:equation f}  is relevant for the consistency of the  metric equation \eqref{eq:metric equation}.

\begin{thm} \label{TH_5.17} Let $L$ be Lorentz-Finsler and 
	$\N$  any nonlinear connection (non-necessarily in $\sol$)  which extends smoothly to $\overline A$. 
	If the Ricci scalar $\ri$ of $\N$ satisfies, for some $\kappa<2n$,  
	$$\kappa\,\ri-L\,g^{ab}\,\ri_{\cdot a\cdot b}=0,$$ then actually $\ri=0$.
	In particular, if $n\geq 3$ then the variational  metric eqn.  \eqref{eq:metric equation}, $\left(n+2\right)\,\ri-L\,g^{ab}\,\ri_{\cdot a\cdot b}=0$, implies $\ri=0$. 
\end{thm}
\begin{proof}
	$\f:=\ri_p$ is $\alpha$-homogeneous for $\alpha=2$, smooth on
	$\overline{A_p}$ (due to the hypothesis on $\N$) and solves \eqref{eq:E3}
	on $A_p$ for $\ka$. Thus,  Lem. \ref{lem:equation f} applies for the chosen $\kappa$. 
\end{proof}
\begin{rem}\label{REM_5.18}
	 (A) This result can be applied to pairs $(\N,L)$ which solve the variational equations. Recall that the Ricci scalar is equal for the solutions obtained starting at one $\N$ and making an  $\A$-translation in the space of solutions $\N+\A\otimes\canon$ (Prop. \ref{prop:invariants solution}). This ensures the consistency of such solutions  as in the classical Palatini case \cite{palatini}. In particular, when $\LC$ is a solution (i.e., when $\la_i=0$), $\ri$ becomes $\riL$.

	(B)   In any dimension $n\geq 3$, the classical vacuum Einstein equation for pseudo-Riemannian metrics $L(x,y)=g_{ab}(x)\,y^a\,y^b$ can be expressed as  
	\[
	4\,\riL-L\,g^{ab}\,\riL_{\cdot a\cdot b}=0
	\]
	(contract both of its indices with $\canon$, and use \eqref{eq:classical curvature} and \eqref{eq:classical scalar} with the Levi-Civita connection).    Thus, when interpreted as an equation for pseudo-Finsler metrics, this one would be the most direct extension of the Einstein equation. Notice that Th. ~\ref{TH_5.17} also applies to it, so for any proper Lorentz-Finsler metric it is equivalent to $\riL=0$ as well.   From a technical viewpoint, it is quite remarkable that this is a nontrivial Finslerian result which requires Lorentzian signature, while in the classical pseudo-Riemannian case an elementary algebraic argument suffices in any signature.

	(C)  The variational equation studied by Hohmann, Pfeifer,  Voicu and Wohlfarth \cite{HPV,pfeifer-wolfhart} agrees with our metric equation 
	when $\la_{i}=0$  (in any dimension)\footnote{Formulas (77) and (79) in \cite{HPV} are immediately generalized from dimension $4$, yielding the terms $-\left(\di+2\right)\riL$ and $L\,g^{ab}\,\riL_{\cdot a\cdot b}$ respectively, while it can be checked that (78) there still yields only terms that vanish when the mean Landsberg tensor does.}. The discrepancy when $\la_i\neq 0$ may be interesting, at least from a mathematical viewpoint. As we have seen, in this case no solution $\N$ of our affine equation can have the same pregeodesics as $\LC$ and it is not clear the role of $\LC$ then. However, no matter the affine solution one chooses, our metric equation is the vanishing of its  $\ri$. For the cited authors, however, it is a more complicated one  which involves $L$ and $\la$.  
	
	(D) Th. \ref{TH_5.17} also complements previous results obtained for the metric nonlinear connection of certain Berwald metrics \cite[Th. 3]{FuPaPf}, \cite[Prop. 4]{HePfFu}. The conclusion of our theorem holds even though the metrics there cannot be extended to $\partial A$ as properly Lorentz-Finsler.  
	
	(E) Previous comments strongly support that the natural generalization of  Einstein vacuum equations  
	must be the vanishing of the Ricci scalar for some solution $\N$ of the affine equation. When $\la_i=0$, $\N^L$ would be a distinguished solution which, in fact, it would be the unique symmetric one under the mild conditions studied before.   Let us point that $\riL=0$ as a vacuum equation was first proposed by Rutz \cite{Rutz} and has been further studied in some cases \cite{MaSh}.  

\end{rem}

\subsection{Recovery of the classical solutions}  Finally, let us restrict our attention to pseudo-Riemannian metrics and affine connections (or, equivalently, linear $\N$'s, $\N_i^k(x,y)=\an_{ib}^k(x)\,y^b$). Then the solutions  of the Finslerian metric-affine formalism (described by $\eqref{eq:affine equation}, \eqref{eq:metric equation}$) are exactly those of the classical one. This fact  will be  proved directly, even though we will give some hints to regard it as a corollary of our results in \S \ref{sec:analytic} and \S \ref{sec:elliptic}, which go way beyond the classical case. Keep in mind that the isotropic $\an$'s solving the classical metric-affine formalism \cite[(17)]{palatini} can be identified with their underlying linear $\N$'s, so   in Def. \ref{formally classical}   we refer as \textit{classical solutions} to those $\N=\LC+\A\otimes\canon$ with $L$ pseudo-Riemannian and $\A$ isotropic. 

\begin{thm}
	Assume that $L$ is pseudo-Riemannian, $\N$ is linear and $n\geq 3$. Then one has $\N\in\sol$ if and only if
	\[
	\N=\LC+\A\otimes\canon
	\]
	for some isotropic $\A$. For these connections, $\ri=\riL$ and $(\N,L)$ solves also the metric equation \eqref{eq:metric equation} if and only if $L$ solves the classical (vacuum) Einstein equation 
	\[
	\riL=0.
	\]
	
\end{thm}

\begin{proof}
	$L$ being pseudo-Riemannian, $\la_i=0$, so $\LC\in\sol$ (Rem. \ref{rem:la_i = 0}) and $\LC+\A\otimes\canon\in\sol$ (Lem. \ref{lem:translations}). Let us establish that these, with $\A$ isotropic, are all the linear elements of $\sol$. 
	
	Again because $L$ is pseudo-Riemannian, $\LC$ is linear ($\left(\LC\right)_i^k(x,y)=\left(  \an^g  \right)_{ib}^k(x)\,y^b$ with $  \an^g  $ the isotropic Levi-Civita connection), and because $\N=\LC+\dv\ZZ+\A\otimes\canon\in\sol$ is assumed linear too, $\A$ must be isotropic. Indeed, from the definition it is clear that the torsion of the linear $\N$ is isotropic, and from  \eqref{eq:torsion determines A},
	\[
	2\left(\di-1\right)\A_i\,y^k=\left(\di-1\right)\torN_{ib}^k\,y^b-\torN_{ai}^a\,y^k-\torN_{ab}^a\,y^b\,\delta_i^k, 
	\]
	\[
	2\left(\di-1\right)\left(\A_{i\,\cdot j}\,y^k+\A_i\,\delta^k_j\right)=\left(\di-1\right)\torN_{ij}^k-\torN_{ai}^a\,\delta^k_j-\torN_{aj}^a\,\delta_i^k, 
	\]
	\[
	2\di\left(\di-1\right)\A_i=\left(\di-1\right)\torN_{ia}^a-\di\,\torN_{ai}^a-\torN_{ai}^a=-2\di\,\torN_{ai}^a
	\]
	(we vertically differentiated, contracted the indices $k$ with $j$, and used the $0$-homogeneity of $\A$ and the antisymmetry of $\torN$). 
	
	As $\A$ is isotropic, it follows that $\ZZ$ is quadratic: $\ZZ^i(x,y)=\varPhi^i_{ab}(x)\,y^a\,y^b/2$ for some isotropic and symmetric $(1,2)$ tensor $\varPhi$. Indeed, formula \eqref{eq:spray solution} for the underlying spray $\G$ of $\N=\LC+\dv\ZZ+\A\otimes\canon$ can be written as 
	\[
	\frac{1}{2}\,\an_{ab}^i(x)\,y^a\,y^b=\frac{1}{2}\left(  \an^g  \right)_{ab}^i(x)\,y^a\,y^b+\ZZ^i(x,y)+\frac{1}{2}\,\A_a(x)\,\delta_{b}^i\,y^a\,y^b
	\]
	and the symmetric part of $\an_{jk}^i-\left(  \an^g  \right)_{jk}^i-\A_j\,\delta_{k}^i$ is an isotropic tensor. 
	
	Now, recalling that $\LC+\dv\ZZ\in\solsim$, one has two options. In a direct manner, using that $\ZZ$ solves \eqref{eq:affine equation 4}, \eqref{eq:affine equation 5}, \eqref{eq:affine equation 6} and the vanishing of the mean Cartan and Landsberg tensors,
	\[
	\left(\di+2\right)\sigma^{\ZZ}=\ZZ^a_{\cdot a}=\varPhi^a_{ab}\,y^b,
	\]
	\[
	\begin{split}
	0&=\left(\di-2\right)\sigma^{\ZZ}-L\,g^{ab}\,\sigma^{\ZZ}_{\cdot a\cdot b} \\
	&=\left(\di-2\right)\sigma^{\ZZ}-L\,g^{ab}\left(\frac{1}{\di+2}\,\varPhi^c_{cd}\,y^d\right)_{\cdot a\cdot b} \\
	&=\left(\di-2\right)\sigma^{\ZZ},
	\end{split}
	\]
	\[
	\ZZ^i=2\,\sigma^{\ZZ}\,y^i-L\,g^{ia}\,\sigma^{\ZZ}_{\cdot a\cdot b}=0
	\]
	(as $\di\geq 3$). Alternatively, one can use that, as  $\LC$ is linear and $\ZZ$  quadratic, also $\LC+\dv\ZZ\in\solprop$ and $\ZZ$ is fiberwise analytic on $\overline{A}$, so either Th. \ref{thm: analytic uniqueness} or Th. \ref{thm:C_i = 0 riemannian} (depending on the signature and again becuase $\di\geq 3$) can be applied \footnote{There would be the technical issue that in non-definite signature, one can regard a pseudo-Riemannian $g$ as a proper pseudo-Finsler $L$ only locally in general. Namely, under Def. \ref{defn:proper} one chooses a certain connected $A_p$ at each point, but the usual pseudo-Riemannian setting includes cases (i.e. non time-orientable Lorentzian metrics) where such a choice cannot carried out.    
	Anyway, the former approach of direct computations avoids this issue altogether.} to conclude that $\LC+\dv\ZZ=\LC$. 
	
	We have proven that if $\N\in\sol$, then $\N=\LC+\A\otimes\canon$ with $\A$ isotropic. As this $\N$ shares fiber in $\sol$ with $\LC$, Cor. \ref{cor:projection} $\mathrm{iv)}$ gives $\ri=\riL$.  The metric equation \eqref{eq:metric equation} for $(\N,L)$ thus reads 
	\begin{equation}
	\left(\di+2\right)\riL-L\,g^{ab}\,\riL_{\cdot a\cdot b}=0. 
	\label{eq:5 aux 8}
	\end{equation}
	However, once again as $L$ is pseudo-Riemannian, $\riL$ is quadratic too.   Indeed, $\riL=\varPsi_{ab}\,y^a\,y^b/2$
	with $\varPsi/2$ being the  (isotropic and symmetric)  classical Ricci tensor of $L$ (use \eqref{eq:classical curvature} with the Levi-Civita connection). Thus, \eqref{eq:5 aux 8} becomes 
	\[
	\begin{split}
	0=&\frac{\di+2}{2}\,\varPsi_{ab}(x)\,y^a\,y^b-L(x,y)\,g^{ab}(x)\left(\frac{1}{2}\,\varPsi_{cd}(x)\,y^c\,y^d\right)_{\cdot a\cdot b} \\
	=&\frac{\di+2}{2}\,\varPsi_{ab}(x)\,y^a\,y^b-L(x,y)\,g^{ab}(x)\,\varPsi_{ab}(x) \\
	=&\left(\frac{\di+2}{2}\,\varPsi_{cd}(x)-g^{ab}(x)\,\varPsi_{ab}(x)\,g_{cd}(x)\right)\,y^c\,y^d,
	\end{split}
	\] 
	which is clearly equivalent to 
	\[
	\frac{\di+2}{2}\,\varPsi_{ij}-g^{ab}\,\varPsi_{ab}\,g_{ij}=0.
	\]
	   By taking metric trace (and once again as $\di\geq 3$), one sees that this one is equivalent to $\varPsi=0$, but this is also true for the classical Einstein equation $\riL=0$. This completes the proof. 
\end{proof}

\begin{rem}
	As a last remark,   recall   that, apart from the classical solutions, a pseudo-Riemannian $L$     admits also    the   formally classical ones,    $\N=\LC+\A\otimes\canon$ with $\A$ anisotropic   and $0$-homogeneous. No other proper solutions can appear in the Lorentzian and Riemannian cases, by Cor. \ref{COR_5.14} and Th. \ref{thm:C_i = 0 riemannian} resp. For general non-definite signature, Th. \ref{thm: analytic uniqueness} establishes that there cannot appear other proper solutions with fiberwise analytic symmetric part $\prsim(\N)$.  
\end{rem} 

\appendix

\section{Proof of Prop. \ref{prop:divergence formulas} (Divergence formulas)} \label{A1}
In order to prove \eqref{eq:horizontal divergence},   we will   lift the anisotropic connection\footnote{This construction works for any anisotropic connection $\an$ in place of $\dv\N$.   In particular, taking   $\an$ as the Levi-Civita--Chern anisotropic connection of the metric \cite{gelocor,anisotropic,chern connection,spray and finsler}, this justifies regarding Chern-Rund's as a connection for $\mathrm{T}A\longrightarrow A$.}$\dv\N$ to a linear (Koszul) connection $\widehat{\nabla}^\N$ for $\mathrm{T}A\longrightarrow A$. For this, recall \cite[Th. 3]{gelocor}, \cite[\S 4.4]{anisotropic}, and the $\N$-horizontal and vertical isomorphisms \eqref{eq:horizontal isomorphism} and \eqref{e_isomorfismo vertical} respectively. One can regard the anisotropic $\dv\N$ as a vertically trivial linear connection for $\mathrm{V}A\longrightarrow A$   as in \cite[Th. 3]{gelocor}, resulting in
\[
\widehat{\nabla}^\N_{X^{\mathrm{H}}}\left(Y^{\mathrm{V}}\right):=\left(\covN_XY\right)^{\mathrm{V}}
\] 
for $X,Y\in\mathcal{T}^1_0(\MA)$. 
Imposing also 
\[
\widehat{\nabla}^\N_{X^{\mathrm{H}}}\left(Y^{\mathrm{H}}\right):=\left(\covN_XY\right)^{\mathrm{H}}
\] 
and maintaining the vertical triviality, $\widehat{\nabla}^\N$ extends unequivocally (by linearity) to act on any vector fields on $A$. Then, by construction, 
\begin{equation}
	\widehat{\nabla}^\N_{\delta_i}\delta_j=\N_{i\,\cdot j}^a\,\delta_a,\qquad\widehat{\nabla}^\N_{\delta_i}\dot{\partial}_j=\N_{i\,\cdot j}^a\,\dot{\partial}_a,\qquad\widehat{\nabla}^\N_{\dot{\partial}_i}\delta_j=0,\qquad\widehat{\nabla}^\N_{\dot{\partial}_i}\dot{\partial}_j=0.
	\label{eq:new connection}
\end{equation}
The \textit{torsion of $\widehat{\nabla}^\N$} is defined, for vector fields $\mathscr{X}$, $\mathscr{Y}$ on $A$, by
\[
\widehat{\torN}(\mathscr{X},\mathscr{Y})=\widehat{\nabla}_{\mathscr{X}}^\N\mathscr{Y}-\widehat{\nabla}_{\mathscr{Y}}^\N\mathscr{X}-\left[\mathscr{X},\mathscr{Y}\right].
\]
  Along the proof, the indices $\hat{i}$, $\hat{j}$, $\hat{k}$ will run in the set $\left\{ 1,...,2\di\right\}$ ($i$, $j$, $k$ remain in $\left\{ 1,...,\di\right\}$) and the local frame $(\delta_1,...,\delta_\di,\dot{\partial}_1,...,\dot{\partial}_\di)$ is denoted by $(E_1,...,E_{2\di})$ with the dual coframe $(\mathrm{d}x^1,...,\mathrm{d}x^\di,\delta y^1,...,\delta y^\di)$ being denoted by $(E^1,...,E^{2\di})$.   Putting, accordingly, $\widehat{\nabla}^{\N}_{E_{\hat{i}}} E_{\hat{j}}=:\widehat{\an}_{\hat{i}\hat{j}}^{\hat{k}}\,E_{\hat{k}}$ and taking \eqref{eq:new connection} into account, it follows that 
\begin{equation}
	\widehat{\an}_{\hat{i}\hat{j}}^{\hat{k}}=\begin{cases}
	\N_{i\,\cdot j}^k & \text{if}\quad(\hat{i},\hat{j},\hat{k})=(i,j,k)\quad\text{or}\quad(\hat{i},\hat{j},\hat{k})=(i,\di+j,\di+k),\\
	0 & \text{otherwise},
	\end{cases}
	\label{eq:chris levantados}
\end{equation}
while putting $\widehat{\torN}(\mathscr{X},\mathscr{Y})=:\mathscr{X}^{\hat{i}}\,\mathscr{Y}^{\hat{j}}\,\widehat{\torN}_{\hat{i}\hat{j}}^{\hat{k}}\,E_{\hat{k}}$, it follows that
\begin{equation}
	\widehat{\torN}_{\hat{i}\hat{j}}^{\hat{k}}=\widehat{\an}_{\hat{i}\hat{j}}^{\hat{k}}-\widehat{\an}_{\hat{j}\hat{i}}^{\hat{k}}-E^{\hat{k}}(\left[E_{\hat{i}},E_{\hat{j}}\right]).
	\label{eq:tor levantada}
\end{equation}
 In a standard manner, we can express any Lie derivative 
\[
\lie_{\mathscr{X}}(\vol)=\lie_{\mathscr{X}}(\vol)(E_1,...,E_{2\di})\,E^1\wedge...\wedge E^{2\di}=:\lie_{\mathscr{X}}(\vol)_E\,E^1\wedge...\wedge E^{2\di}
\]
where
\begin{equation}
\vol=\frac{\left|\det g_{ij}(v)\right|}{F(v)^\di}\,E^1\wedge...\wedge E^{2\di}=:\vol_E\,E^1\wedge...\wedge E^{2\di}
\label{eq:d mu E}
\end{equation}
in terms of $\widehat{\nabla}^\N$.  Indeed,
\[
\begin{split}
\lie_{\mathscr{X}}(\vol)_E
=&\lie_{\mathscr{X}}(\vol(E_{1},...,E_{2\di}))-\sum_{\hat{j}=1}^{2\di}\vol(E_{1},...,\lie_{\mathscr{X}}E_{\hat{j}},...,E_{2\di}) \\=&\mathscr{X}(\vol_{E})-\sum_{\hat{j}=1}^{2\di}\vol(E_{1},...,\left[\mathscr{X},E_{\hat{j}}\right],...,E_{2\di})\\
=&\mathscr{X}(\vol_{E})-\sum_{\hat{j}=1}^{2\di}\vol(E_{1},...,\widehat{\nabla}^\N_{\mathscr{X}}E_{\hat{j}}-\widehat{\nabla}^\N_{E_{\hat{j}}}\mathscr{X}-\widehat{\torN}(\mathscr{X},E_{\hat{j}}),...,E_{2\di}) \\
=&\mathscr{X}(\log\vol_{E})\,\vol_{E}-\sum_{\hat{j}=1}^{2\di}\vol(...,\mathscr{X}^{\hat{i}}\,\widehat{\an}_{\hat{i}\hat{j}}^{\hat{k}}\,E_{\hat{k}},...) \\
&+\sum_{\hat{j}=1}^{2\di}\vol(...,E_{\hat{j}}(\mathscr{X}^{\hat{i}})\,E_{\hat{i}}+\widehat{\an}_{\hat{j}\hat{i}}^{\hat{k}}\,\mathscr{X}^{\hat{i}}\,E_{\hat{k}},...)+\sum_{\hat{j}=1}^{2\di}\vol(...,\mathscr{X}^{\hat{i}}\,\widehat{\torN}_{\hat{i}\hat{j}}^{\hat{k}}\,E_{\hat{k}},...) \\
=&\left\{\mathscr{X}(\log\vol_{E})-\mathscr{X}^{\hat{i}}\,\widehat{\an}_{\hat{i}\hat{j}}^{\hat{j}}+\left(E_{\hat{j}}(\mathscr{X}^{\hat{j}})+\widehat{\an}_{\hat{j}\hat{i}}^{\hat{j}}\,\mathscr{X}^{\hat{i}}\right)+\mathscr{X}^{\hat{i}}\,\widehat{\torN}_{\hat{i}\hat{j}}^{\hat{j}}\right\}\vol_E,
\end{split}
\]
so 
\begin{equation}
\begin{split}
	&\dive(\mathscr{X})\,\vol \\ =&\lie_{\mathscr{X}}(\vol)_E\,E^1\wedge...\wedge E^{2\di}\\
	=&\left\{\mathscr{X}(\log\vol_{E})-\mathscr{X}^{\hat{i}}\,\widehat{\an}_{\hat{i}\hat{j}}^{\hat{j}}+\left(E_{\hat{j}}(\mathscr{X}^{\hat{j}})+\widehat{\an}_{\hat{j}\hat{i}}^{\hat{j}}\,\mathscr{X}^{\hat{i}}\right)+\mathscr{X}^{\hat{i}}\,\widehat{\torN}_{\hat{i}\hat{j}}^{\hat{j}}\right\}\vol_E\,E^1\wedge...\wedge E^{2\di} \\
	=& \left\{\mathscr{X}(\log\vol_{E})-\mathscr{X}^{\hat{i}}\,\widehat{\an}_{\hat{i}\hat{j}}^{\hat{j}}+\left(E_{\hat{j}}(\mathscr{X}^{\hat{j}})+\widehat{\an}_{\hat{j}\hat{i}}^{\hat{j}}\,\mathscr{X}^{\hat{i}}\right)+\mathscr{X}^{\hat{i}}\,\widehat{\torN}_{\hat{i}\hat{j}}^{\hat{j}}\right\}\vol
\end{split}
\label{eq:div fundamental}
\end{equation}
(and note that $\left(\mathscr{X}(\log\vol_{E})-\mathscr{X}^{\hat{i}}\,\widehat{\an}_{\hat{i}\hat{j}}^{\hat{j}}\right)\vol=\widehat{\nabla}^\N_\mathscr{X}\vol$).

One has the identities 
\[
E_{i}(\det g)=\det(g)\,g^{ab}\,\delta_{i}g_{ab}=\det(g)\left(g^{ab}\,\covN_i g_{ab}+2\,\N_{i\,\cdot a}^{a}\right)
\]
(using Jacobi's formula for the derivative of a determinant and \eqref{eq:covariant derivative}),
\[
\qquad E_{i}(F)=\frac{\mathrm{sgn}(L)}{2F}\,\delta_iL=\frac{\mathrm{sgn}(L)}{2F}\,\covN_iL=\frac{\mathrm{sgn}(L)}{2F}\,\covN_i g_{ab}\,y^{a}\,y^{b}
\]
(using $F=\sqrt{\left|L\right|}$, $L=g_{ab}\,y^a\,y^b$ and $\covN_iy^j=0$),
\[
E_{\di+i}(\det g)=2\det(g)\,\C_{i},
\]
(using again Jacobi and the definition of the mean Cartan tensor), and 
\[
\qquad E_{\di+i}(F)=\frac{\mathrm{sgn}(L)}{F}\,y_{i}
\]
(using again $F=\sqrt{\left|L\right|}$ and $L_{\cdot i}=2\,y_i$). From them and \eqref{eq:d mu E}, it follows that 
\begin{equation}
E_{i}(\log\vol_{E})=\frac{E_{i}(\vol_{E})}{\vol_{E}}=\left(g^{ab}-\frac{\di}{2}\,\frac{1}{L}\,y^{a}\,y^{b}\right)\covN_i g_{ab}+2\,\N_{i\,\cdot a}^{a},\label{eq: E_i(log(vol))}
\end{equation}
\begin{equation}
E_{\di+i}(\log\vol_{E})=\frac{E_{\di+i}(\vol_{E})}{\vol_{E}}=2\,\C_{i}-\di\,\frac{y_{i}}{L}.
\label{eq: E_n+i(log(vol))}
\end{equation}

We take $\mathscr{X}=X^{\mathrm{H}}=X^a\,E_a$. Using \eqref{eq: E_i(log(vol))}, \eqref{eq:chris levantados}, \eqref{eq:tor levantada} and the commutation formulas \eqref{eq:commutation formulas}, we have 
\[
\mathscr{X}(\log \vol_{E})=X^{c}\left(g^{ab}-\frac{\di}{2}\,\frac{1}{L}\,y^{a}\,y^{b}\right)\covN_c g_{ab}+2\,X^{c}\,\N_{c\,\cdot a}^{a},
\]
\[
-\mathscr{X}^{\hat{i}}\,\widehat{\an}_{\hat{i}\hat{j}}^{\hat{j}}=-\mathscr{X}^{\hat{i}}\,\widehat{\an}_{\hat{i}a}^{a}-\mathscr{X}^{\hat{i}}\,\widehat{\an}_{\hat{i}\,\di+a}^{\di+a}=-X^{c}\,\N_{c\,\cdot a}^{a}-X^{c}\,\N_{c\,\cdot a}^{a}=-2\,X^{c}\,\N_{c\,\cdot a}^{a},
\]
\[
\begin{split}
E_{\hat{j}}(\mathscr{X}^{\hat{j}})+\widehat{\an}_{\hat{j}\hat{i}}^{\hat{j}}\,\mathscr{X}^{\hat{i}}=&E_{a}(\mathscr{X}^{a})+E_{\di+a}(\mathscr{X}^{\di+a})+\widehat{\an}_{a\hat{i}}^{a}\,\mathscr{X}^{\hat{i}}+\widehat{\an}_{\di+a\,\hat{i}}^{\di+a}\,\mathscr{X}^{\hat{i}}\\
=&\delta_{a}X^{a}+\N_{a\,\cdot c}^{a}\,X^{c}\\
=&\covN_a X^{a},
\end{split}
\]
\[
\begin{split}
\mathscr{X}^{\hat{i}}\,\widehat{\torN}_{\hat{i}\hat{j}}^{\hat{i}}=&\mathscr{X}^{\hat{i}}\left(\widehat{\an}_{\hat{i}\hat{j}}^{\hat{j}}-\widehat{\an}_{\hat{j}\hat{i}}^{\hat{j}}-E^{\hat{j}}(\left[E_{\hat{i}},E_{\hat{j}}\right])\right)\\
=&\mathscr{X}^{\hat{i}}\left(\widehat{\an}_{\hat{i}a}^{a}+\widehat{\an}_{\hat{i}\,\di+a}^{\di+a}-\widehat{\an}_{a\hat{i}}^{a}-\widehat{\an}_{\di+a\,\hat{i}}^{\di+a}-E^{a}(\left[E_{\hat{i}},E_{a}\right])-E^{\di+a}(\left[E_{\hat{i}},E_{\di+a}\right])\right)\\
=&X^{c}\left(\N_{c\,\cdot a}^{a}+\N_{c\,\cdot a}^{a}-\N_{a\,\cdot c}^{a}-\mathrm{d}x^{a}(\left[\delta_{c},\delta_{a}\right])-\delta y^{a}(\left[\delta_{c},\dot{\partial}_{a}\right])\right) \\
=&X^{c}\left(2\,\N_{c\,\cdot a}^{a}-\N_{a\,\cdot c}^{a}-\N_{c\,\cdot a}^{a}\right)\\
=&X^{c}\,\torN_{ca}^{a}.
\end{split}
\]
Putting these together, \eqref{eq:div fundamental} proves \eqref{eq:horizontal divergence}. 

Now we take  $\mathscr{X}=X^{\mathrm{V}}=X^a\,E_{\di+a}$.  Using \eqref{eq: E_n+i(log(vol))}, and again \eqref{eq:chris levantados}, \eqref{eq:tor levantada} and the commutation formulas \eqref{eq:commutation formulas}, we have
\[
\mathscr{X}(\log\vol_{E})=\left(2\,\C_{c}-\di\,\frac{y_{c}}{L}\right)X^{c},
\]
\[
-\mathscr{X}^{\hat{i}}\,\widehat{\an}_{\hat{i}\hat{j}}^{\hat{j}}=-\mathscr{X}^{\hat{i}}\,\widehat{\an}_{\hat{i}a}^{a}-\mathscr{X}^{\hat{i}}\,\widehat{\an}_{\hat{i}\,\di+a}^{\di+a}=-X^{c}\,\widehat{\an}_{\di+c\,a}^{a}-X^c\,\widehat{\an}_{\di+c\,\di+a}^{\di+a}=0,
\]
\[
\begin{split}
E_{\hat{j}}(\mathscr{X}^{\hat{j}})+\widehat{\an}_{\hat{j}\hat{i}}^{\hat{j}}\,\mathscr{X}^{\hat{i}}=&E_{a}(\mathscr{X}^{a})+E_{\di+a}(\mathscr{X}^{\di+a})+\widehat{\an}_{a\hat{i}}^{a}\,\mathscr{X}^{\hat{i}}+\widehat{\an}_{\di+a\,\hat{i}}^{\di+a}\,\mathscr{X}^{\hat{i}}
=\dot{\partial}_{a}X^{a}
=X^{a}_{\cdot a},
\end{split}
\]
\[
\begin{split}
\mathscr{X}^{\hat{i}}\,\widehat{\torN}_{\hat{i}\hat{j}}^{\hat{i}}=&\mathscr{X}^{\hat{i}}\left(\widehat{\an}_{\hat{i}\hat{j}}^{\hat{j}}-\widehat{\an}_{\hat{j}\hat{i}}^{\hat{j}}-E^{\hat{j}}(\left[E_{\hat{i}},E_{\hat{j}}\right])\right)\\
=&\mathscr{X}^{\hat{i}}\left(\widehat{\an}_{\hat{i}a}^{a}+\widehat{\an}_{\hat{i}\,\di+a}^{\di+a}-\widehat{\an}_{a\hat{i}}^{a}-\widehat{\an}_{\di+a\,\hat{i}}^{\di+a}-E^{a}(\left[E_{\hat{i}},E_{a}\right])-E^{\di+a}(\left[E_{\hat{i}},E_{\di+a}\right])\right)\\
=&X^{c}\left(-\mathrm{d}x^{a}(\left[\dot{\partial}_{c},\delta_{a}\right])-\delta y^{a}(\left[\dot{\partial}_{c},\dot{\partial}_{a}\right])\right) \\
=0.
\end{split}
\]
Putting these together, \eqref{eq:div fundamental} proves \eqref{eq: vertical divergence},\footnote{  Notice, however, that \eqref{eq: vertical divergence} is a purely vertical identity independent of $\N$. So,   it could also have been proven by direct computation without any connection for $\mathrm{T}A\longrightarrow A$.}  and yields the proposition.

\section{Proof of Th. \ref{thm:variational equations} (Affine equation)} \label{app:A}

When varying $\N$ by $\N(\tau)$, taking Rem. \ref{rem:vol independent of N} into account, it is immediate to check that 
\begin{equation}
\begin{split}\left.\frac{\partial}{\partial\tau}\right|_{\tau=0}\mathscr{S}^{D}[\N(\tau),L] & =\int_{\dom}\left.\frac{\partial}{\partial\tau}\right|_{\tau=0}\underline{L^{-1}\,\ri(\tau)\,\vol} \\
 & =\int_{\dom}\underline{L^{-1}\left.\frac{\partial}{\partial\tau}\right|_{\tau=0}\ri(\tau)\,\vol}. 
 \label{eq:variation N}
\end{split}
\end{equation}
Using \eqref{eq:curvature and torsion} and \eqref{eq:horizontal distribution},
\[
\ri(\tau)
=\delta_{b}(\tau)\N_{a}^{c}(\tau)\left(\delta^a_c\,y^b-y^a\,\delta^b_c\right),\qquad\delta_{j}(\tau)\N_{i}^{k}(\tau)=\partial_{j}\N_{i}^{k}(\tau)-\N_{j}^{d}(\tau)\,\dot\partial_d\N_{i}^{k}(\tau);
\]
 here, $\delta_c^a$ is Kronecker's, in contrast to $\delta_j(\tau)$, which comes from $\N(\tau)$. 

Let us express the derivative of $\delta_{j}(\tau)\N_{i}^{k}(\tau)$ in terms of $\covN$ and $\torN_{ib}^k\,y^b=\left(\N_{i\,\cdot b}^k-\N_{b\,\cdot i}^k\right)y^b=\N_{a\,\cdot b}^k\left(\delta^a_i\,y^b-y^a\,\delta^b_i\right)$.  We do this by commuting $\left.\partial_{\tau}\right|_0$ with $\partial_j$ and $\dot\partial_d$, 
\[
\left.\frac{\partial}{\partial\tau}\right|_{\tau=0}\left\{ \delta_{j}(\tau)\N_{i}^{k}(\tau)\right\}=\partial_{j}\left(\N^\prime\right)_{i}^{k}-\left(\N^\prime\right)_{j}^{d}\,\N_{i\,\cdot d}^{k}-\N_{j}^{d}\left(\N^\prime\right)_{i\,\cdot d}^{k}
=\delta_{j}\left(\N^\prime\right)_{i}^{k}-\left(\N^\prime\right)_{j}^{d}\,\N_{i\,\cdot d}^{k}
\]
and then adding and substracting $-\N_{j\,\cdot i}^{d}\left(\N^\prime\right)_{d}^{k}+\N_{j\,\cdot d}^{k}\left(\N^\prime\right)_{i}^{d}$ so as to obtain the same terms as in \eqref{eq:covariant derivative},
\[
 \left.\frac{\partial}{\partial\tau}\right|_{\tau=0}\left\{ \delta_{j}(\tau)\N_{i}^{k}(\tau)\right\} =\covN_j\left(\N^\prime\right)_{i}^{k}+\N_{j\,\cdot i}^{d}\left(\N^\prime\right)_{d}^{k}-\N_{j\,\cdot d}^{k}\left(\N^\prime\right)_{i}^{d}-\N_{i\,\cdot d}^{k}\left(\N^\prime\right)_{j}^{d}.
\] 
With this,
\begin{equation}
	\begin{split}
	& \quad L^{-1}\left.\frac{\partial}{\partial\tau}\right|_{\tau=0}\ri(\tau) \\
	& = L^{-1}\left\{\covN_b\left(\N^\prime\right)_{a}^{c}+\N_{b\,\cdot a}^{d}\left(\N^\prime\right)_{d}^{c}-\N_{b\,\cdot d}^{c}\left(\N^\prime\right)_{a}^{d}-\N_{a\,\cdot d}^{c}\left(\N^\prime\right)_{b}^{d}\right\}\left(\delta^a_c\,y^b-y^a\,\delta^b_c\right) \\
	& = L^{-1}\left\{\covN_b\left(\N^\prime\right)_{a}^{c}\left(\delta^a_c\,y^b-y^a\,\delta^b_c\right)+\N_{b\,\cdot a}^{d}\left(\delta^a_c\,y^b-y^a\,\delta^b_c\right)\left(\N^\prime\right)_{d}^{c}\right\} \\
	& = L^{-1}\,\covN_c\left(\N^\prime\right)_{d}^{d}\,y^c-L^{-1}\,\covN_c\left(\N^\prime\right)_{d}^{c}\,y^d-L^{-1}\,\torN_{ca}^d\,y^a\left(\N^\prime\right)_{d}^{c}.
	\label{eq:variation ric}
	\end{split}
\end{equation}

Recall that, by Prop. \ref{prop:covariante canon = 0}, $\covN_i L=\covN_i g_{ab}\,y^a\,y^b$. Calling $X:=L^{-1}\left(\N^\prime\right)_{d}^{d}\,y^c\,\partial_{c}\in\mathrm{h}^{0}\mathcal{T}_{0}^{1}(M_A)$ and using \eqref{eq:horizontal divergence},  
\begin{equation}
\begin{split} 
& \quad L^{-1}\,\covN_c\left(\N^\prime\right)_{d}^{d}\,y^c \\
&=\covN_c(L^{-1}\left(\N^\prime\right)_{d}^{d}\,y^c)-\covN_c(L^{-1})\left(\N^\prime\right)_{d}^{d}\,y^c \\
&=\dive(X^\mathrm{H})-L^{-1}\left\{ \left(g^{ab}-\frac{\di}{2}\,\frac{1}{L}\,y^{a}\,y^{b}\right)\covN_c g_{ab}+\torN_{ca}^{a}\right\}y^{c}\left(\N^\prime\right)_{d}^{d} \\
& \quad+L^{-2}\,y^c\,\covN_c g_{ab}\,y^a\,y^b\,\left(\N^\prime\right)_{d}^{d} \\
&= \dive(X^\mathrm{H})-L^{-1}\left\{ \left(g^{ab}-\frac{\di+2}{2}\,\frac{1}{L}\,y^{a}\,y^{b}\right)\covN_c g_{ab}+\torN_{ca}^{a}\right\}y^{c}\left(\N^\prime\right)_{d}^{d}.
\label{eq:integration by parts 1}
\end{split}
\end{equation}
Analogously, calling $Y:=L^{-1}\left(\N^\prime\right)_{d}^{c}\,y^{d}\,\partial_{c}\in\mathrm{h}^{0}\mathcal{T}_{0}^{1}(M_A)$,
\begin{equation}
	\begin{split}
	& \quad L^{-1}\,\covN_c\left(\N^\prime\right)_{d}^{c}\,y^d \\
	& = \dive(Y^\mathrm{H})-L^{-1}\left\{ \left(g^{ab}-\frac{\di+2}{2}\,\frac{1}{L}\,y^{a}\,y^{b}\right)\covN_c g_{ab}+\torN_{ca}^{a}\right\}y^{d}\left(\N^\prime\right)_{d}^{c}.
	\label{eq:integration by parts 2}
	\end{split}
\end{equation}
Substituting \eqref{eq:integration by parts 1} and \eqref{eq:integration by parts 2} in \eqref{eq:variation ric},  
\[
\begin{split}
& \quad L^{-1}\left.\frac{\partial}{\partial\tau}\right|_{\tau=0}\ri(\tau) \\ &=\dive(X^{H})-\dive(Y^{H})-L^{-1}\left\{ \left(g^{ab}-\frac{\di+2}{2}\,\frac{1}{L}\,y^{a}\,y^{b}\right)\covN_c g_{ab}+\torN_{ca}^{a}\right\}y^{c}\left(\N^\prime\right)_{d}^{d} \\
& \quad +L^{-1}\left\{ \left(g^{ab}-\frac{\di+2}{2}\,\frac{1}{L}\,y^{a}\,y^{b}\right)\covN_c g_{ab}+\torN_{ca}^{a}\right\}y^{d}\left(\N^\prime\right)_{d}^{c}-L^{-1}\,\torN_{ca}^d\,y^a\left(\N^\prime\right)_{d}^{c}.
\end{split}
\]
 Prop. \ref{prop:divergence formulas} also guarantees that, upon integration on $\mathbb{P}^+A$, the divergence terms can be discarded. Indeed:
\[
\int_{\dom}\underline{\dive(X^\mathrm{H})\,\vol}=-\int_{\dom}\mathrm{d}(\underline{X^\mathrm{H}\lrcorner\vol})=-\int_{\partial\dom}\underline{X^\mathrm{H}\lrcorner\vol}
\]
(analogously for $\underline{\dive(Y^\mathrm{H})\,\vol}$) and, by the fact that $\N(\tau)$ is $\dom$-admissible (Def. \ref{def:variations}), $X$ and $Y$ vanish on $\left(\mathbb{P}^+\right)^{-1}(\partial\dom)$, so $\underline{X^\mathrm{H}\lrcorner\vol}$ and $\underline{Y^\mathrm{H}\lrcorner\vol}$ vanish on $\partial\dom$ (see the comment at the end of Prop. \ref{prop:inducing forms} (ii)).  The remaining terms, substituting back in \eqref{eq:variation N}, can be expressed as
\[
\begin{split}
&\quad\left.\frac{\partial}{\partial\tau}\right|_{\tau=0}\mathscr{S}^{D}[\N(\tau),L] \\
&=\int_\dom\underline{L^{-1}\left\{ \left(g^{ab}-\frac{\di+2}{2}\,\frac{1}{L}\,y^{a}\,y^{b}\right)\covN_c g_{ab}+\torN_{ca}^{a}\right\}\left(\delta^c_e\,y^d-y^c\,\delta^d_e\right)\left(\N^\prime\right)_{d}^{e}\,\vol} \\
& \quad-\int_\dom\underline{L^{-1}\,\torN_{ea}^d\,y^a\left(\N^\prime\right)_{d}^{e}\,\vol}.
\end{split}
\]

 The field $\N^\prime\in\mathrm{h}^1\mathcal{T}^1_1(\MA)$ with $\mathbb{P}^+(\mathrm{Supp}\,\N^\prime)$ relatively compact in $\mathbb{P}^+A$ is arbitrary: for any such $\N^\prime$, there exists a variation $\N(\tau)$ that has it as its variational field (for instance, $\N(\tau)=\N+\tau\,\N^\prime$). Thanks to this, the standard argument of the calculus of variations can be applied (on a $\dom$ around each $\mathbb{P}^+v\in\mathbb{P}^+A$). We conclude that the vanishing of all the $\left.\partial_\tau\right|_{0}\mathscr{S}^{D}[\N(\tau),L]$'s is equivalent to 
\begin{equation}
\left\{\left(g^{ab}-\frac{\di+2}{2}\,\frac{1}{L}\,y^{a}\,y^{b}\right)\covN_c g_{ab}+\torN_{ca}^{a}\right\}\left(\delta^c_i\,y^j-y^c\,\delta^j_i\right)-\torN_{ia}^j\,y^a=0
\label{eq:affine equation 0}
\end{equation}
on $A$. 

The only thing that remains is to reexpress this in terms of $\J:=\N-\LC$. Substituting \eqref{eq:4 aux 7} and \eqref{eq:4 aux 13} in \eqref{eq:affine equation 0} yields the required equation \eqref{eq:affine equation}. 

\section{Proof of Th. \ref{thm:variational equations} (Metric equation)
} \label{app:B}
When varying $L$ by $L(\tau)$, it is immediate that
\begin{equation}
	\begin{split}\left.\frac{\partial}{\partial\tau}\right|_{\tau=0}\mathscr{S}^{D}[\N,L(\tau)] & =\int_{\dom}\left.\frac{\partial}{\partial\tau}\right|_{\tau=0}\underline{L(\tau)^{-1}\,\ri\,\vol(\tau)} \\
	& =\int_{\dom}\underline{\left.\frac{\partial}{\partial\tau}\right|_{\tau=0}\left\{L(\tau)^{-1}\,\ri\,\vol(\tau)\right\}}\\
	& =-\int_{\dom}\underline{L^{-1}\,\frac{\ri}{L}\,L^\prime\,\vol}+\int_{\dom}\underline{L^{-1}\,\ri\left.\frac{\partial}{\partial\tau}\right|_{\tau=0}\vol(\tau)}.
	\label{eq:variation L}
	\end{split}
\end{equation}
By \eqref{eq:vol},  
\[
\vol(\tau)=\frac{\left|\det g_{ij}(\tau)\right|}{L(\tau)^{\frac{\di}{2}}}\,\mathrm{d}x\wedge\mathrm{d}y.
\]

We compute the derivative of this taking into account that 
\[
\left.\frac{\partial}{\partial\tau}\right|_{\tau=0}g_{ij}(\tau)=\frac{1}{2}\,L^\prime_{\cdot i \cdot j},\qquad\left.\frac{\partial}{\partial\tau}\right|_{\tau=0}L(\tau)^{\frac{\di}{2}}=\frac{\di}{2}\,L^{\frac{\di}{2}-1}\,L^\prime:
\]
by Jacobi's formula for the derivative of a determinant,
\[
\begin{split}
\left.\frac{\partial}{\partial\tau}\right|_{\tau=0}\vol(\tau) &=\left(\frac{1}{2}\,\frac{\left|\det g_{ij}\right|}{L^{\frac{\di}{2}}}\,g^{ab}\,L^\prime_{\cdot a\cdot b}-\frac{\di}{2}\,\frac{\left|\det g_{ij}\right|}{L^{\di}}\,L^{\frac{\di}{2}-1}\,L^\prime\right)\mathrm{d}x\wedge\mathrm{d}y \\
&=\left(\frac{1}{2}\,g^{ab}\,L^\prime_{\cdot a\cdot b}-\frac{\di}{2}\,\frac{1}{L}\,L^\prime\right)\vol.
\end{split}
\]
Substituting in \eqref{eq:variation L} and putting $\widetilde{\ri}:=L^{-1}\,\ri\in\mathrm{h}^{0}\mathcal{F}(A)$,
\begin{equation}
	\quad\left.\frac{\partial}{\partial\tau}\right|_{\tau=0}\mathscr{S}^{D}[\N,L(\tau)]=-\frac{\di+2}{2}\int_{\dom}\underline{L^{-1}\,\widetilde{\ri}\,L^\prime\,\vol}+\frac{1}{2}\int_{\dom}\underline{\widetilde{\ri}\,g^{ab}\,L^\prime_{\cdot a\cdot b}\,\vol}.
	\label{eq:variation L 2}
\end{equation}

Calling $X:=\widetilde{\ri}\,g^{ab}\,L^\prime_{\cdot a}\,\partial_{b}\in\mathrm{h}^{1}\mathcal{T}_{0}^{1}(M_A)$ and using \eqref{eq: vertical divergence}, $g^{ib}_{\cdot b}=-2\,\C^i$, and the $2$-homogeneity of $L^\prime$,
\[
\begin{split}
\widetilde{\ri}\,g^{ab}\,L^\prime_{\cdot a\cdot b} &=X^b_{\cdot b}-g^{ab}\,\widetilde{\ri}_{\cdot b}\,L^\prime_{\cdot a}-\widetilde{\ri}\,g^{ab}_{\cdot b}\,L^\prime_{\cdot a}\\
&=\dive(X^\mathrm{V})-\widetilde{\ri}\,g^{ab}\left(2\,\C_{b}-\di\,\frac{y_{b}}{L}\right)L^\prime_{\cdot a}-g^{ab}\,\widetilde{\ri}_{\cdot b}\,L^\prime_{\cdot a}+2\,\widetilde{\ri}\,\C^a\,L^\prime_{\cdot a} \\
&=\dive(X^\mathrm{V})+2\di L^{-1}\,\widetilde{\ri}\,L^\prime-g^{ab}\,\widetilde{\ri}_{\cdot b}\,L^\prime_{\cdot a}.
\end{split}
\]
Calling $Y:=L^\prime\,g^{ab}\,\widetilde{\ri}_{\cdot b}\,\partial_{a}\in\mathrm{h}^{1}\mathcal{T}_{0}^{1}(M_A)$ and again using \eqref{eq: vertical divergence}, $g^{ia}_{\cdot a}=-2\,\C^i$, and the $0$-homogeneity of $\widetilde{\ri}$,
\[
\begin{split}
\widetilde{\ri}\,g^{ab}\,L^\prime_{\cdot a\cdot b} &=\dive(X^\mathrm{V})+2\di L^{-1}\,\widetilde{\ri}\,L^\prime-Y^a_{\cdot a}+g^{ab}_{\cdot a}\,\widetilde{\ri}_{\cdot b}\,L^\prime+g^{ab}\,\widetilde{\ri}_{\cdot a \cdot b}\,L^\prime \\
&=\dive(X^\mathrm{V})+2\di L^{-1}\,\widetilde{\ri}\,L^\prime-\dive(Y^\mathrm{V}) \\
&\quad+g^{ab}\left(2\,\C_{a}-\di\,\frac{y_{a}}{L}\right)\widetilde{\ri}_{\cdot b}\,L^\prime-2\,\C^b\,\widetilde{\ri}_{\cdot b}\,L^\prime+g^{ab}\,\widetilde{\ri}_{\cdot a \cdot b}\,L^\prime \\
&=\dive(X^\mathrm{V})-\dive(Y^\mathrm{V})+2\di L^{-1}\,\widetilde{\ri}\,L^\prime+g^{ab}\,\widetilde{\ri}_{\cdot a \cdot b}\,L^\prime.
\end{split}
\]
Substituting this back in \eqref{eq:variation L 2} and dropping the divergence terms  (by the analogous reasoning as in Appendix \ref{app:A}), 
\[
\begin{split}
&\left.\frac{\partial}{\partial\tau}\right|_{\tau=0}\mathscr{S}^{D}[\N,L(\tau)] \\ =&-\frac{\di+2}{2}\int_{\dom}\underline{L^{-1}\,\widetilde{\ri}\,L^\prime\,\vol}+\di\int_{\dom}\underline{L^{-1}\,\widetilde{\ri}\,L^\prime\,\vol}+\frac{1}{2}\int_{\dom}\underline{g^{ab}\,\widetilde{\ri}_{\cdot a \cdot b}\,L^\prime\,\vol} \\
=&\frac{\di-2}{2}\int_{\dom}\underline{L^{-1}\,\widetilde{\ri}\,L^\prime\,\vol}+\frac{1}{2}\int_{\dom}\underline{g^{ab}\,\widetilde{\ri}_{\cdot a \cdot b}\,L^\prime\,\vol}.
\end{split}
\]

 The field $L^\prime\in\mathrm{h}^2\mathcal{F}(A)$ with $\mathbb{P}^+(\mathrm{Supp}\,L^\prime)$ relatively compact and small enough in $\mathbb{P}^+A$ is arbitrary: for any such $L^\prime$, there exists a variation $L(\tau)$ that has it as its variational field (for instance, $L(\tau)=L+\tau\,L^\prime$). Again, the standard argument of the calculus of variations can be applied around each $\mathbb{P}^+v\in\mathbb{P}^+A$, concluding that the vanishing of all the $\left.\partial_\tau\right|_{0}\mathscr{S}^{D}[\N,L(\tau)]$'s is equivalent to 
\[
\left(\di-2\right)L^{-1}\,\widetilde{\ri}+g^{ab}\,\widetilde{\ri}_{\cdot a \cdot b}=0.
\] 

 Finally, one straightforwardly rewrites 
\[
\left(\di-2\right)L^{-1}\,\widetilde{\ri}+g^{ab}\,\widetilde{\ri}_{\cdot a \cdot b}=-\left(\di+2\right)L^{-2}\,\ri+L^{-1}\,g^{ab}\,\ri_{\cdot a \cdot b};
\]
indeed, the right hand side of this becomes the left hand side by the same computations as in the beginning of the proof of Lem. \ref{lem:equation f}, yielding the required equation \eqref{eq:metric equation}.

\end{document}